\documentclass[twoside,letterpaper,draft,10pt]{amsart}
\usepackage{amssymb,amsmath,amsthm,latexsym, stmaryrd}
\usepackage{amsfonts}
\usepackage[twlrm]{rawfonts}

\usepackage[english]{babel}
\usepackage[utf8]{inputenc}
\usepackage{indentfirst}
\usepackage{graphicx}
\usepackage[usenames]{color}
\usepackage[matrix,arrow,curve]{xy}
\usepackage{lscape}
\usepackage{pdflscape}
\usepackage{multicol}
\newcommand{\mfp}{{\mathfrak p}}
\DeclareFontFamily{U}{skulls}{}
\DeclareFontShape{U}{skulls}{m}{n}{ <-> skull }{}

\usepackage{multirow}
\usepackage[fixlanguage]{babelbib}
\usepackage[font=small,format=plain,labelfont=bf,up,textfont=it,up]{caption}
\usepackage[usenames,svgnames,dvipsnames]{xcolor}
\usepackage[a4paper,top=4.0cm,bottom=2.0cm,left=3.5cm,right=2.0cm]{geometry}

\newtheorem{teo}{Theorem}[section]
\newtheorem{cor}[teo]{Corollary}
\newtheorem{lem}[teo]{Lemma}
\newtheorem{pro}[teo]{Proposition}
\newtheorem{obs}[teo]{Remark}
\newtheorem{afr}[teo]{Claim}

\newtheorem{defi}[teo]{Definition}
\newtheorem{exe}[teo]{Example}
\newenvironment{dem}[1][Proof]{\noindent\textbf{#1.} }{\hfill \rule{0.5em}{0.5em}}

\newcommand{\dsum}{\displaystyle\sum}

\newcommand{\B}{\mathcal{B}}

\newcommand{\D}{\Delta}
\newcommand{\N}{\mathbb{N}}
\newcommand{\F}{\mathcal{F}}
\newcommand{\T}{\Theta}
\newcommand{\C}{\mathcal{C}}
\newcommand{\Om}{\Omega}
\newcommand{\om}{\omega}
\newcommand{\Gm}{\Gamma}
\newcommand{\Si}{\Sigma}

\newcommand{\p}{\mathcal{P}}
\newcommand{\Z}{\mathcal{Z}}
\newcommand{\R}{\Rightarrow}
\newcommand{\al}{\alpha}
\newcommand{\U}{\mathcal{U}}

\newcommand{\m}{{-1}}

\newcommand{\Aut}{\mathbf{Aut}}
\newcommand{\End}{\mbox{End}}
\newcommand{\Pic}{\mathbf{Pic}}
\newcommand{\Pics}{\mathbf{PicS}}
\newcommand{\Hom}{\mbox{Hom}}
\newcommand{\G}{\Gamma}

\newcommand{\s}{\sigma}
\newcommand{\et}{\theta}
\newcommand{\e}{\varepsilon}
\newcommand{\ot}{\otimes}
\newcommand{\op}{{\rm op}}
\usepackage{float}

\begin{document}

\title[]{Partial generalized crossed products, Brauer groups and a comparison of  seven-term exact sequences}
\author[M. Dokuchaev]{Mikhailo Dokuchaev}
\address{Instituto de
Matem{\'a}tica e Estat\'\i stica,
Universidade de S\~ao Paulo, 
Rua do Mat\~ao, 1010,
05508-090, S\~ao Paulo, SP, Brasil}
\email{dokucha@ime.usp.br}

	\author[H. Pinedo]{H{\'e}ctor Pinedo}
	\address{Escuela de Matematicas, Universidad Industrial de Santander, Cra. 27 Calle 9  UIS
		Edificio 45, Bucaramanga, Colombia}
	\email{hpinedot@uis.edu.co}

\author[I.\ Rocha]{Itailma Rocha}
\address{Unidade Acad\^emica de Matem\'atica, Universidade Federal de Campina Grande - UAMat/UFCG, 
Avenida Apr\'{\i}gio Veloso, 882, 58429-900, Campina Grande, PB, Brasil.}
\email{itailma@mat.ufcg.edu.br}




\date{}
\begin{abstract} Given a unital partial action $\alpha $ of a group $G$ on a commutative ring $R$ we denote by $ \Pics _{R^{\alpha}}(R) $ the Picard monoid of the isomorphism classes of  partially invertible $R$-bimodules, which are central over the subring $R^{\alpha} \subseteq R$ of $\alpha$-invariant elements, and consider a specific unital partial representation 
$\T : G \to \Pics _{R^{\alpha}}(R), $  along with the abelian group $\C(\T/R)$ of the isomorphism classes of partial generalized crossed products related to $\T,$ which already showed their importance in obtaining a partial action analogue of the Chase-Harrison-Rosenberg seven-term exact sequence. We give a description of  $\C(\T/R)$ in terms  partial generalized products of the form 
$\D(f \T)$ where $f$ is partial $1$-cocycle of $G$ with values in a submonoid of  $ \Pics _{R^{\alpha}}(R) .$ Assuming that $G$ is finite and that 
$R^{\alpha} \subseteq R$ is a partial Galois extension, we prove that any Azumaya  $R^\al$-algebra, containing  $R$ as a maximal  commutative  subalgebra, is isomorphic to a partial generalized crossed product. Furthermore, we show that the relative Brauer group 
$\B(R/R^\al)$  can be seen as a quotient of $ \C(\T/R)$ by a subgroup isomorphic to  the Picard group of $R.$ Finally, we prove that the analogue of the Chase-Harrison-Rosenberg  sequence,  obtained earlier for partial Galois extensions of commutative rings, can be derived from a recent seven-term exact sequence established in a non-commutative setting.
\end{abstract}

\maketitle
\section{Introduction}

In \cite{chase1965galois}, given a Galois extension of commutative rings with a finite Galois group,   S.U. Chase, D.K. Harrison, A. Rosenberg gave a seven-term exact sequence involving Galois cohomology groups, Picard groups and the relative Brauer group. Their proof   is based on the seven-term exact sequence, involving  Amitsur cohomology,   obtained  in \cite{chase1965amitsur} by means of spectral sequences. A constructive proof for the Chase-Harrison-Rosneberg exact sequence was given by T. Kanzaki in \cite{kanzaki1968generalized},  using the novel notion of a generalized crossed product. The latter concept was extended  by Y. Miyashita in \cite{miyashita1973exact} and employed  to establish a seven-term exact sequence related
 to an extension of non-necessarily commutative unital rings $R\subseteq S$ and a representation of a group $G$  by invertible $R$-subbimodules of $S.$ The form of the Miyashita’s sequence is not the same as that of the Chase-Harrison-Rosenberg sequence, but the latter can be obtained from the former by taking $S$ to be the skew group ring constructed from the action of the Galois group. Using their results from \cite{el2010invertible},
 L. El Kaoutit, J. Gómez-Torrecillas extended in \cite{el2012invertible} the Myashita’s sequence for the context of rings with local units, which  a significantly  more complicated technical situation.

Technical challenges also arise when replacing  actions by partial actions on algebras and representations by partial representations. 
The latter notions appeared in the theory of $C^*$-algebras \cite{E-1}, \cite{Mc}, \cite{exel1997twisted}, \cite{QR},  \cite{exel1998partial}, motivating algebraic and $C^*$-algebraic developments, and possessing  remarkable applications  \cite{E6}.  In algebra the new concepts are useful to graded algebras, Hecke algebras, 
Leavitt path algebras, Thompson’s groups, inverse semigroups, restriction semigroups and automata (see the survey article \cite{D3} and the references therein).  
Among the recent applications, we mention the significance of partial actions and partial representations  to dynamical systems associated 
with  separated graphs and related C*-algebras \cite{AraE1},
\cite{AL}, to paradoxical decompositions \cite{AraE1}, to shifts \cite{AL}, 
to full or reduced C*-algebras of E-unitary or strongly E*-unitary inverse semigroups  \cite{MiSt}, to topological higher-rank graphs \cite{RenWil}, to Matsumoto and Carlsen–Matsumoto C*-algebras of arbitrary subshifts  \cite{DE2}, to ultragraph C*-algebras \cite{GR3}, 
embeddings of inverse semigroups \cite{Khry1}, Ehresmann semigroups \cite{KudLaan} and to expansions of monoids in the class of two-sided restriction monoids \cite{Kud2}, with the latter involving a previous construction from \cite{Kud} based  on partial actions. The reader may consult the R. Exel's book \cite{E6} and the surveys \cite{D3} and \cite{Ba2017} for the basics  on partial actions and partial representations, theoretical developments and applications. 

An action of a finite group $G$ on a commutative unital ring $R$  induces an action of $G$ on the Picard group  ${\bf Pic}_R(R)$ of the isomorphism classes of the finitely generated projective $R$-modules of rank $1.$ Assuming that $R^G \subseteq R$ is a Galois extension, where $R^G$ stands for the subring of the $G$-invariant elements of $R,$ the Chase-Harrison-Rosenberg sequence is of the form
\begin{align*}
&1\to H^1(G, \U(R)) {\to} {\bf Pic}_{R^G}(R^G) {\to} {\bf Pic}_R(R)^G {\to }
H^2(G, \U(R)  ) {\to} B(R/R^G) {\to} H^1(G,{\bf Pic}_R(R)){\to}\\ & H^3(G, \U(R)),
\end{align*} where $\U(R)$ stands for the group of invertible elements of $R,$  ${\bf Pic}_R(R)^G$ is the subgroup of the $G$-invariant elements of  ${\bf Pic}_R(R)$  and  $B(R/R^G)$ is the relative Brauer group of the equivalence  classes of the Azumaya $R^{G}$-algebras split by $R.$

In a series of two articles \cite{DoPaPi}, \cite{DoPaPi2} a generalization of the Chase-Harrison-Rosenberg sequence was constructed  for a partial Galois extension of commutative rings. Galois theory based on partial actions was introduced in \cite{dokuchaev2007partial},  with further developments in \cite{CaenDGr}, \cite{CaenJan}, \cite{PRS2011}, \cite{BP2012}, \cite{KuoSzeto2014}, \cite{PaqTam2018}, \cite{KS}, \cite{JKS}, \cite{CaenFier}, \cite{KuoSzeto2023}.
New ingredients in the sequence of \cite{DoPaPi2} are the   cohomology groups based on partial actions, introduced in \cite{dokuchaev2015partial} (see Section~\ref{subsec:ParAcParRep} and Section~\ref{sec:ParCohomol}), and the 
 Picard semigroup ${\bf PicS}_R(R)$ of the finitely generated projective $R$-modules of rank $\leq 1$ (see Section~\ref{sec:Picard}).

A unital partial action $\alpha$ of a finite group $G$ on a commutative ring $R$  induces a unital partial action $\alpha ^*$ of $G$ on ${\bf PicS}_R(R)$ (see Example~\ref{pacal}).
Given a Galois extension $R^{\alpha} \subseteq R$ of commutative rings with a unital partial action $\alpha $ of a finite group $G$ on $R$, where $R^{\alpha}$ is the subring of $\alpha$-invariant elements of $R$ (see Section~\ref{subsec:ParAcParRep} for the definitions) the generalization of the Chase-Harrison-Rosenberg exact sequence in \cite{DoPaPi2} has the following form:

\begin{align}\label{Intro:Exact}
&1 \longrightarrow H^1(G,\alpha , R)
\longrightarrow {\bf Pic}_{R^\alpha}(R^\alpha )\longrightarrow{\bf PicS}_R(R)^{\alpha^*}\cap {\bf Pic}_R(R) \longrightarrow H^2(G,\alpha, R) \longrightarrow \\& B(R/R^\alpha) \longrightarrow
 H^1(G,\alpha^*,{\bf PicS}_R(R)) \longrightarrow   
H^3 (G,\alpha , R).\nonumber
\end{align}

 The above mentioned Miyashita's sequence was extended to the context of partial representations  in \cite{DoRo}, and it is a natural question if the exact sequence \eqref{Intro:Exact}  can be recovered from that in \cite{DoRo}.  We give a positive  answer, but the proof is essentially more laborious than that in the usual case and requires a better  understanding of some of the involved ingredients.

Miyashita's sequence is  related  to  a general ring extension $R'\subseteq S$ with the same unity and a fixed homomorphism $\Theta :  G \to  {\bf Inv} _{R'}(S),$ where $G$ is  an arbitrary  group and 
$ {\bf Inv} _{R'}(S)$ is the group of invertible $R'$-subbimodules of $S.$ 
One of the crucial ingredients of  Miyashita's sequence is an appropriate analogue  
${\mathcal  B} (\Theta /R')$ of the Brauer group, which is a  quotient of the abelian group ${\mathcal C} (\Theta /R')$ of the isomorphism classes of certain generalized crossed products related to $\Theta .$  Another important new ingredient is the group $ \p(S/R')$ of isomorphism classes of certain objects  defined using $R'$-bilinear maps $\phi: P \to X,$  where $P$ is an $R'$-bimodule, $X$ is an $S$-bimodule, with $\phi$ satisfying  additional conditions.

The starting point of the sequence in \cite{DoRo} is an extension of non-necessarily commutative  rings $R'\subseteq S$ with he same unity and  
   a    unital partial representation
   $\G : G  \to \mathcal{S}_{R'}(S),$ where $G$ is an arbitrary group and  $\mathcal{S}_{R'}(S)$ is the monoid of the $R'$-subbimodules of $S,$ equipped with the element-wise multiplication of  subsets of $S$  (see formula \eqref{eq:MonoidS}).  The analogue 
$\mathcal{B}(\G/R')$ of the 
Brauer group is defined for a unital partial representation of the form  $\G:G\to \Pics(R'),$ where   $\Pics(R')$ is the monoid of the isomorphism classes of the partially invertible $R'$-bimodules, with the multiplication induced by $\ot _{R'}$ and neutral element $R'$ (see Section~\ref{sec:Picard} for the definition). Then  $\mathcal{B}(\G/R')$ is an appropriate quotient (see \eqref{GroupB}) of the abelian group $\C(\G/R')$ of the isomorphism classes of certain generalized partial crossed products  related to $\Gamma $ (see Section~\ref{subsec:ParAcParRep} and Section~\ref{sec:groupC} for the details).
A unital partial representation $\G:G\to \Pics(R')$  induces a unital partial action $\bar{\alpha}$ of $G$ on the center $\mathcal Z = \mathcal Z (R') $ of $R'$ (see \cite[ Proposition 3.11]{DoRo}) and a unital partial action 
$\alpha ^*$ of $G$ on $\Pics _{\mathcal Z}(R'),$ where 
$\Pics _{\mathcal Z}(R')$ is the submonoid of $\Pics (R'),$
formed by the isomorphism classes of those partially invertible $R'$-bimodules, which are central  over 
${\mathcal Z}$  (see Example~\ref{pacal}). Moreover, the partial action $\alpha ^*$ can be restricted to a unital partial action   $\bar{\alpha}^*$ of $G$ on $\Pics _0(R')$ \cite[p. 56]{DoRo}, where $\Pics _0(R')$ is a submonoid of $\Pics (R')$ defined by a natural divisibility property of the considered modules (see  \eqref{PicsZero}).

Given a  unital partial representation $\G : G  \to \mathcal{S}_{R'}(S),$ the exact sequence in \cite{DoRo} has the form:

\begin{align}\label{IntroNoncommutExact}
		&1 \longrightarrow   H^{1}(G, \overline\al,\Z) \longrightarrow   \p_\Z(\D(\Gamma)/R')^{(G)} \longrightarrow  \Pic_\Z(R')\cap\Pics_\Z(R')^{\al^*} \longrightarrow  H^2(G,\overline\al,\Z) 
			\longrightarrow \\ & \mathcal{B}(\Gamma/R')\longrightarrow 	 
			\overline{H}^1(G,\overline\al^*,\Pics_0(R'))\longrightarrow  H^3(G,\overline\al,\Z),\nonumber  
	\end{align} where $H^{i}(G, \overline\al,\Z),$ $i=1,2,3,$ are the partial cohomology groups (see Section~\ref{sec:ParCohomol}), $\overline{H}^1(G,\overline\al^*,\Pics_0(R'))$ is an appropriate quotient of the partial cohomology group ${H}^1(G,\overline\al^*,\Pics_0(R'))$ (see the defining sequence \eqref{GroupOverlineH}) and   
	$ \p_\Z(\D(\Gamma)/R')^{(G)}$ is a certain subgroup of  $ \p (\D(\Gamma)/R')$ (see  formula \eqref{pz}), where $\D(\Gamma)$  is the partial ge\-ne\-ra\-li\-zed crossed product related to $\Gamma .$

	The paper begins  with preliminaries on the 
Picard group ${\bf Pic} (R)$ and the Picard monoid 	
$ \Pics(R)$ of a non-necessarily commutative ring $R$ in Section~\ref{sec:Picard}, in particular,  the definitions of the subgroup 
${\bf Pic} _k (R) \subseteq {\bf Pic} (R)$ and that of the submonoid $\Pics _k(R) \subseteq \Pics(R)$ are recalled, when $R$ is an algebra over a  commutative ring $k.$ In Section~\ref{sec:Group P}, for an extension $R\subseteq S$ of  non-necessarily commutative rings with the same unity element the group $\p(S/R)$ is defined, as well as its subgroup $\p _k(S/R)$, when $R$ is a $k$-algebra. Under the latter condition, assuming that    $R$ is commutative and  that $k \subseteq  \Z(S),$ in Proposition~\ref{morfismoentrePicePSR} a group homomorphism 
$ \xi:  \Pic_k(k)  \to  \p_R(S/R)$ is constructed.

Section~\ref{subsec:ParAcParRep} deals with a background  on partial actions, partial representations, partial Galois extensions of commutative rings and partial generalized crossed products. For a unital partial action $\alpha $ of a group $G$   on a commutative ring $R$ an important partial representation
$\T : G \to \Pics _{R^{\alpha}}(R) $  is recalled, which is crucial for our comparison of  sequences \eqref{Intro:Exact} and \eqref{IntroNoncommutExact}.
 In particular, it is employed in  Example~\ref{pacal} to give a characterization of a unital partial action $\alpha ^*$ of $G$ on $\Pics _R (R)$ induced by $\alpha .$ In Lemma~\ref{fs} a convenient factor set for $\T$ is constructed, which is used, in particular,  to identify  the corresponding partial generalized crossed product $\Delta (\T )$ with the partial skew group ring
$ R \star _{\alpha }G$  in Lemma~\ref{pcgpisoskew}. Next, given  a partial Galois   extension $ R^\alpha \subseteq R$ of commutative rings the above mentioned homomorphism $ \xi $ is used in Theorem~\ref{isopic} to produce an isomorphism 
\begin{equation*}\label{Intro:IsoPic}
\Pic_{R^\al}(R^\al)  \to \p_R(S/R)^{(G)},
\end{equation*} where  $S= R\star_{\al}G.$ This is a first step towards the comparison of the exact sequences \eqref{Intro:Exact} and 
\eqref{IntroNoncommutExact}.

In Section~\ref{sec:ParCohomol} basics on partial group cohomology  and partial crossed products are given, whereas 
Section~\ref{sec:groupC} is dedicated to the group $\C(\T/R)$ of the isomorphism classes of partial generalized crossed products related to the abobve mentioned partial representation $\T.$ A very useful characterization of $\C(\T/R)$ is obtained in Theorem~\ref{prop:grupoC}, in which
$\C(\T/R)$ is described as the group of the isomorphism classes of the partial generalized products of the form $\D(f \T)$ where $f$ is partial $1$-cocycle of $G$ with values $\Pics _R (R).$

 For a partial Galois extension $R^{\alpha} \subseteq R$ of commutative rings 
our understanding of the relation between the relative Brauer group $B(R/R^\alpha)$ and partial generalized crossed products is refined in 
Section~\ref{sec:Brauer}. Firstly, in Theorem~\ref{Aisopcgp} for   an Azumaya  $R^\al$-algebra $A$, containing  $R$ as a maximal  commutative  subalgebra, we establish an isomorphism   $\D(f_A \Theta)\simeq A^{\op}$ of $R^\al$-algebras, where the cocycle $f_A\in Z^1(G,\al^*,\Pics_R(R))$ is constructed using $A.$
Then, after proving some additional facts, we establish one of our main results, Theorem~\ref{exac},  which says that there is an exact sequence  
$$\Pic_R(R)\to  \C(\T/R) \to  \B(R/R^\al)\to  1,$$ clarifying the relation between $   \B(R/R^\al) $ and $ \C(\T/R).$

Our comparison of sequences \eqref{Intro:Exact} and  \eqref{IntroNoncommutExact} is given in Section~\ref{sec:Comparing}. For this purpose we recall first some details from \cite{DoRo} considering   an arbitrary  extension of  non-necessarily commutative rings  $R'\subseteq S$ with the same unity  and   a unital  partial representation $\G : G \to  \mathcal{S}_{R'}(S),$ of a group $G$ in the semigroup of $R'$-subbimodules of $S.$ In particular, we  give the definitions of the groups 
$\Pics_0(R')$ (see \eqref{PicsZero}), $\mathcal{B}(\Gamma/R')$ (see \eqref{GroupB}) and 
$\overline{H^1}(G,\al^*,\Pics_0(R'))$ (see \eqref{GroupOverlineH}). Furthermore, we
	recall  the definition of the maps in  \eqref{IntroNoncommutExact}. Then considering 
a partial Galois extension $R^{\alpha} \subseteq R$ of commutative rings with a unital partial action $\alpha $ of a finite group $G$ on $R,$ we take  $S= R \star _{\alpha} G,$ $R'=  R$ and $\G = \T $  in \eqref{IntroNoncommutExact}, which is an idea similar to that used in the case of global  actions. Employing   results obtained in  previous sections we show that 
in this case the groups in  \eqref{Intro:Exact}    are either equal, or isomorphic to the corresponding groups of the sequence 
\eqref{IntroNoncommutExact}, and that the homomorphisms in \eqref{Intro:Exact} can be obtained from those in \eqref{IntroNoncommutExact}. Some final remarks are made in Section~\ref{sec:final}.

In all what follows, unless otherwise stated,  all rings will be considered associative  with unity and the modules over  rings will be unital. We write that $M$ is a {\it  f.g.p. $S$-module} if $M$ is a  projective and finitely generated (left) $S$-module. Moreover,  we shall say that a ring $S$ is an extension of a ring $R$ if $R$ is a subring of $S$ and $1_S=1_R.$ We denote by $ \Z(S)$  the center of a ring $S$ and by $\U(S)$ its group of invertible elements. 
If $S$ is commutative,  in several occasions we shall consider a left $S$-module $M$ as a central $S$-$S$-bimodule, i.e, an $S$-$S$-bimodule $M$
with $ms = sm$ for all $m\in M$ and $s\in S.$  However, in general an $S$-bimodule over a commutative ring $S$ will not be considered central, unless its centrality will be clear from the context. 
  Moreover,   $G$ will stand for a group with identity $1.$

\section{Background}\label{sec:Background}

In this section we recall some notions and results which shall be useful for us.
\subsection{The Picard group and the Picard monoid}\label{sec:Picard}
Let $M$ and $N$ be bimodules over a non-necessarily commutative ring $R.$ Then the sets $\Hom(M_R,N_R)$ and $\Hom(_RM,_RN)$ are $R$-bimodules via:
\begin{equation}
(r\cdot f)(m)=rf(m) \ \ \mbox{and} \ \ (f\cdot r)(m)=f(rm), \ \ f\in \Hom(M_R,N_R), m \in M, r \in R, \label{estrelaMebimodulo}
\end{equation} 
\begin{equation}
(r\cdot g)(m)=g(mr) \ \ \mbox{and} \ \ (g\cdot r)(m)=g(m)r, \ \ g \in \Hom(_RM,_RN), m\in M, r\in R .\label{Mestrelaebimodulo} 
\end{equation}

In particular, taking $N=R,$ we have that the sets  ${^*M}=\Hom(M_R,R_R)$ and 
${M^*}=\Hom(_RM,_RR)$ are $R$-bimodules via  (\ref{estrelaMebimodulo}) and (\ref{Mestrelaebimodulo}).

We recall from \cite[Chapter II]{Bass}   that an $R$-bimodule $P$ is called {\it invertible},  if  there  exists an $R$-bimodule $Q$ such that $P\ot_RQ\simeq R\simeq Q\ot_RP,$ as $R$-bimodules.  Then the    \textit{Picard group of R}, denoted by  $\Pic(R),$  is the set of the isomorphism classes  of  invertible $R$-bimodules with multiplication induced by $\otimes_R$.  By  \cite[Theorem 1.1]{morita1967endomorphism},  if $[P] \in \Pic(R),$ then ${^*P}\simeq P^*$ as $R$-bimodules, and hence   $[^*P]=[P^*]=[P]^\m$ in $\Pic(R)$. We shall write $[P]^\m=[P^\m]$.  For a $k$-algebra $R$ over a commutative ring $k$ we denote by  $\Pic_k(R)$ the set of the isomorphism classes of  invertible  $R$-bimodules which are central  over $k,$  i.e.,  $rp=pr$ for all  $r \in k$ and $p \in P$. In the case when  $R$  is commutative, we have that  $[P]\in \Pic_R(R)$ if and only if $P$ is a finitely generated projective $R$-module of rank $1$  {(see, for example, \cite[II, \S 5]{demeyer1971separable})}.


 Given an extension of rings $R\subseteq S$ and an $R$-bimodule $M$, we consider  $M\otimes_RS$ as an  $S$-bimodule via 
	\begin{equation}\label{eq:bimodStructureOnTensorPr}
	s_1\cdot (m\ot _R s)\cdot s_2=m\ot _R s_1ss_2, \ \mbox{for} \ m \in M \  \mbox{and} \ s,s_1,s_2 \in S.
	\end{equation}
Then there is an  $S$-bimodule isomorphism
\begin{equation}\label{eq:easyBimodMap}
S\otimes_S(M\otimes_RS)\simeq M\otimes_R S, \ \mbox{ given by } \ s\otimes_S(m\otimes_R s')\mapsto m\otimes_R s s'.
\end{equation}
Analogously, we consider  $S \otimes_R M$ as an  $S$-bimodule via 
	\begin{equation}\label{eq:bimodStructureOnTensorPr2}
	s_1\cdot (s\ot _R m)\cdot s_2= s_1ss_2 \ot _R m, \ \mbox{for} \ m \in M \  \mbox{and} \ s,s_1,s_2 \in S.
	\end{equation}
We proceed with the next.
\begin{lem}\label{PotSemPicS} Let $R\subseteq S$ be a ring extension and suppose that  $R\subseteq \Z(S)$. Then for  $[P]\in \Pic(R)$, one has that $[P\ot_RS]\in \Pic(S)$,  where the  $S$-bimodule structure of  $ P\ot_RS$ is given by \eqref{eq:bimodStructureOnTensorPr}.
	\end{lem}
	\begin{dem}
		Let $[P^\m]$ be the inverse of  $[P]$ in $\Pic(R)$. Then there is a chain of  $S$-bimodule isomorphisms
		\begin{equation*}
			(P\ot_RS)\ot_S(P^{\m}\ot_RS)\simeq  P\ot_R(S\ot_S(P^{\m}\ot_RS))\simeq P\ot_R(P^\m\ot_RS)\simeq R\ot_RS\simeq S,
		\end{equation*}  the condition  $R\subseteq \Z(S)$ being used to ensure that  the natural isomorphim $ R\ot_RS\simeq S$  is an  isomorphism of $S$-bimodules.
		Analogously one shows that  $(P^\m\ot_RS)\ot_S(P\ot_RS)\simeq  S$, as $S$-bimodules.
		Therefore, $[P\ot_RS]\in \Pic(S)$. \end{dem}

We say that an  
$R$-bimodule $P$ is \textit{partially invertible} if 

\begin{itemize}
	\item[(i)] $P$ is finitely generated projective left and right  $R$-module;
	\item[(ii)] The maps
	\begin{center}
		\begin{tabular}{c c c}
			$\begin{array}{c c l}
			R& \longrightarrow & \End(P_R)\\
			r & \longmapsto & (p\mapsto rp)
			\end{array}$ & and & $\begin{array}{c c l}
			R& \longrightarrow & \End(_RP)\\
			r & \longmapsto & (p\mapsto pr)
			\end{array}$
		\end{tabular}
	\end{center}
	are epimorphisms. 
	
	\end{itemize}

We denote by $\mathbf{PicS}(R)$ the set of the isomorphism classes $[P]$ of partially invertible $R$-bimodules, 
that is, 
$
\mathbf{PicS}(R)=\{[P]: P \ \mbox{is a partially invertible $R$-bimodule}\}. 
$

We recall the next.
\begin{pro}\cite[Proposition 2.6]{DoRo} The set  $\Pics(R)$ is a monoid with multiplication defined by $[P][Q]=[P\ot_RQ],$ for all  $[P],[Q]\in \Pics(R).$
\end{pro}

Obviously, the group of invertible elements of  $\Pics(R)$ is  $\Pic(R),$ since for $[P]\in \Pic(R)$ the maps in (ii) are well known isomorphisms. 

 Suppose that $R$ is a  $k$-algebra over a commutative ring $k.$ In a similar way to  $\Pic(R)$  we denote by  $\Pics_k(R)$ the set of the isomorphism classes of  partially invertible  $R$-bimodules which are central  over $k.$ 
If  $R$ is commutative, then it follows by (3.1) in  \cite{DoPaPi} that 
\begin{equation}\label{upic}
\Pics_R(R)=\bigcup_{e\in E(R)}\Pic_R(Re),
\end{equation}
where $E(R)$ denotes the set of idempotents of $R.$  Observe that if $R$ is commutative, then $\Pics_R(R)$ is the semigroup of the finitely generated projective $R$-modules 
of rank $\leq 1$ (see \cite[pp. 743-744]{DoPaPi}).

We shall also use the following facts. The proof of the first can be  obtained by localization,  for the second one we refer to the literature.

\begin{lem}\label{MMotNpfgimplicaNpfg} Let $R$ be a commutative ring,  $M,N$ be R-modules such that $M$ and  $M\ot_R N$ are f.g.p. R-modules, and $M_\mfp\ne 0$ for all $\mfp \in {\rm Spec}(R).$
Then, N is also a  f.g.p.  R-module.
\end{lem}

\begin{lem}\cite[Chapter I, Lemma 3.2 (b)]{demeyer1971separable}
\label{MotNisoMotN'implicaNisoN'} Let $R$ be a commutative ring, $M$  be  a f.g.p. R-module, and  $N,N'$ be $R$-modules  such that  $M\ot_RN\simeq M\ot_RN'$ as $R$-modules, then  $N\simeq N'$ as $R$-modules.
\end{lem}

\subsection{The group $\p(S/R)$}\label{sec:Group P}
\label{sec: pSR}

Let $R\subseteq S$ be an extension of non-necessarily commutative rings.  Following \cite{el2010invertible}, let $\mathcal{M}(S/R)$ be  the category with objects denoted by   $\xymatrix@C=1.2cm{ P\ar@{=>}[r]|-{[\phi]}& X}$, where $P$ is an  $R$-bimodule, $X$ is an  $S$-bimodule and $\phi:P\longrightarrow X$ is an $R$-bilinear map such that the maps  
\begin{equation*}
	\begin{tabular}{c c c}
		${\begin{array}{c c c l}
				\bar{\phi_r}: & P\otimes_RS & \longrightarrow & X ,\\
				& p\otimes_Rs & \longrightarrow & \phi(p)s \end{array}}$ & \ \ \mbox{and} \ \ & \ \ \ 	${\begin{array}{c c c l}
				\bar{\phi_l}:&	S\otimes_RP & \longrightarrow & X ,\\
				&	s\otimes_Rp & \longrightarrow & s\phi(p) \end{array}} $
	\end{tabular} \label{phirephil}
\end{equation*}
are isomorphisms of $R$-$S$-bimodules and $S$-$R$-bimodules, respectively.  {A  morphism   from $\xymatrix@C=1.2cm{ P\ar@{=>}[r]|{[\phi]}& X}$ to $\xymatrix@C=1.2cm{ Q\ar@{=>}[r]|{[\psi]}& Y}$ in $\mathcal{M}(S/R)$ is a pair $(\al, \beta),$} where $\al:P\longrightarrow Q$ is $R$-bilinear, $\beta:X\longrightarrow Y$ is $S$-bilinear and the  diagram 
$$\xymatrix{P\ar[rr]^{\phi}\ar[d]_{\al} & &  X\ar[d]^{\beta}\\
	Q\ar[rr]_{\psi} & &   Y}$$
is commutative. If a morphism $(\al, \beta)$ is such that  $\al$ and $\beta$ are isomorphisms of $R$-bimodules and $S$-bimodules, respectively, then  $(\al, \beta)$ is an  isomorphism between  $\xymatrix@C=1.2cm{ P\ar@{=>}[r]|{[\phi]}& X}$ and $\xymatrix@C=1.2cm{ Q\ar@{=>}[r]|{[\psi]}& Y.}$ We shall denote by $\xymatrix@C=1.2cm{ [P]\ar@{=>}[r]|{[\phi]}& [X]}$ the isomophism class of an object $\xymatrix@C=1.2cm{ P\ar@{=>}[r]|{[\phi]}& X}$  in $\mathcal{M}(S/R).$

Let  $\p(S/R)$  be the set of the isomorphism classes  $\xymatrix@C=1.2cm{ [P]\ar@{=>}[r]|{[\phi]}& [X]}$,  where $[P]\in \Pic(R)$ and $[X]\in \Pic(S)$. 


\begin{obs}\label{obsphirouphil} {By} \cite[Lemma 3.1]{yoichi1971galois}, { if}  $[P]\in \Pic(R)$ and $[X]\in \Pic(S)$, then $\xymatrix@C=1.2cm{ [P]\ar@{=>}[r]|{[\phi]}& [X]}\in \p(S/R),$ provided that  $\bar{\phi}_l$ (or $\bar{\phi}_r$) is an isomorphism.
\end{obs}
%

It follows from \cite[Theorem 1.3]{miyashita1973exact}, that $\p(S/R)$ is a group, in which the product of the isomorphism class of  $\xymatrix@C=1.2cm{ P\ar@{=>}[r]|{[\phi]}& X} \in \mathcal{M}(S/R)$  and that of   $\xymatrix@C=1.2cm{ Q\ar@{=>}[r]|{[\psi]}& Y} \in \mathcal{M}(S/R)$ is given by the isomorphism class of 
	$$\xymatrix@C=1.9cm{ P\ot_RQ\ar@{=>}[r]|-{[\phi\ot \psi]} & X\ot_SY}\in \mathcal{M}(S/R),$$
	 where 	$\phi\ot\psi: P\ot_RQ \longrightarrow  X\ot_S Y$ is defined by  	$(\phi\ot\psi)(p\ot q)=\phi(p)\ot\psi(q).$
	Moreover, the identity element is the isomorphism class   $[R]=[\iota]\R[S]$ 
of  $R=[\iota]\R S,$ and the inverse of the class of  $\xymatrix@C=1.2cm{ P\ar@{=>}[r]|{[\phi]}& X}$ is  
	$\xymatrix@C=1.2cm{ [P^*]\ar@{=>}[r]|{[\phi^*]}& [X^*]},$ where $\phi^*(f)(s\phi(p))=sf(p),$ with $s\in S,$  $p \in P$ and $f\in P^*$.

If $R$ is a $k$-algebra   over a commutative ring $k,$ we define   
\begin{equation}\label{pk}\p_k(S/R)=\{\xymatrix@C=1.2cm{ [P]\ar@{=>}[r]|{[\phi]} & [X]} \in \p(S/R):  [P]\in \Pic_k(R)\}.\end{equation}
It is easy to see that  $\p_k(S/R)$ is a subgroup of  $\p(S/R),$ see \cite[page 9]{DoRo}.

\begin{pro}\label{morfismoentrePicePSR} Let $k\subseteq R\subseteq S$ be a tower  of rings with the same unity. Suppose that   $R$ is commutative and  that $k \subseteq  \Z(S).$  Then the map
	$$\begin{array}{c c c l}
	\xi: & \Pic_k(k) & \longrightarrow & \p_R(S/R)\\
	& [P_0] & \longmapsto & ([P_0\ot_kR]=[\phi]\R[P_0\ot_kS]),
	\end{array}$$
where  $\phi:P_0\ot_kR\longrightarrow P_0\ot_kS$ is the inclusion map, is 	a well-defined group homomorphism.

\end{pro}
\begin{dem}
	Let  $[P_0]  \in \Pic_k(k).$ It follows from   Lemma \ref{PotSemPicS} that  $[P_0\ot_kR]\in \Pic(R)$ and $[P_0\ot_kS]\in \Pic(S)$. Consider the $R$-bimodule map
	$$
	\bar{\phi}_l:  S\ot_R(P_0\ot_kR) \ni   s\ot_Rp_0\ot_kr \mapsto p_0\ot_ksr \in  P_0\ot_kS.$$
We shall prove that $\bar{\phi}_l$ is invertible with inverse given by
$$ P_0\ot_kS\ni p_0\ot_k s\mapsto s\ot_R p_0\ot_k 1\in S\ot_R P_0\ot_k R.$$
	Indeed,  to check that  $\bar{\phi}_l^{-1}$ is well defined take  $\mu \in k$. Since  $P_0$ is a central $k$-bimodule and   $k\subseteq \Z(S)$ we have that 
	\begin{eqnarray*}
		\bar{\phi}_l^{-1}(p_0\mu, s)  =  s\ot_R p_0\mu\ot_k 1=s\mu\ot_R p_0\ot _k1
		=  \mu s\ot_R p_0\ot_k 1=\bar{\phi}_l^{-1}(p_0,\mu s).
	\end{eqnarray*}
	which shows that, $\bar{\phi}_l^{-1}$ is well-defined. On the other hand, notice that
	\begin{eqnarray*}
		(\bar{\phi}_l^{-1}\circ\bar{\phi}_l)(s\ot_Rp_0\ot_kr)& = & \bar{\phi}_l^{-1}(p_0\ot_k sr) = sr\ot_Rp_0\ot_k1 \\
		& = & s\ot_Rr(p_0\ot_k1)=s\ot_Rp_0\ot_kr,
	\end{eqnarray*}
	and
$
		(\bar{\phi}_l\circ\bar{\phi}_l^{-1})(p_0\ot_ks) = \bar{\phi}_l(s\ot_R p_0 \ot_k 1)=p_0\ot_ks. 
$
	Then it follows by   Remark \ref{obsphirouphil}  that $([P_0\ot_kR]=[\phi]\R[P_0\ot_kS])\in \p_R(S/R)$.  To ensure that  $\xi$ is well-defined we take a central $k$-bimodule $P'_0$ such that $[P_0]=[P_0']$ in $\Pic_k(k)$, hence there exists a   $k$-bimodule isomorphism  $f:P_0\longrightarrow P_0',$ and it is easy to see that the diagram below commutes 
	$$\xymatrix{ P_0\ot_kR\ar[dd]_{f\ot R}\ar[rr]^{\phi} & & P_0\ot_kS\ar[dd]^{f\ot S}\\
		& & \\
		P_0'\ot_kR\ar[rr]_{\phi'} & & P_0'\ot_kS   }$$
	where  $\phi$ and $\phi'$  are inclusions, this shows that $([P_0\ot_kR]=[\phi]\R[P_0\ot_kS])=([P'_0\ot_kR]=[\phi']\R[P'_0\ot_kS])$ in  $\p_R(S/R),$ and $\xi$ is well-defined.  To check that it is a group homomorphism take  $[P_0],[P'_0]\in \Pic_k(k),$  then 
	$$\xi([P_0\ot_kP_0'])=([P_0\ot_kP_0'\ot_kR]=[\psi]\R[P_0\ot_kP'_0\ot_kS]),$$
	where  $\psi$ if the inclusion,  moreover
	$$\xi([P_0])\xi([P_0'])=([P_0\ot_kR\ot_RP_0'\ot_kR]=  [\phi \otimes \phi']\R[P_0\ot_kS\ot_SP_0'\ot_kS]),$$
	and the following diagram commutes
	$$\xymatrix{ P_0\ot_kR\ot_RP_0'\ot_kR\ar[rr]^{  {\phi \otimes \phi'}   }\ar[dd]_{\simeq}  & &   P_0\ot_kS\ot_SP_0'\ot_kS\ar[dd]^{\simeq}\\
		& & \\
		P_0\ot_kP_0'\ot_kR\ar[rr]_{\psi} & &  P_0\ot_kP'_0\ot_kS,}$$
	 where the vertical isomorphisms are obtained using \eqref{eq:easyBimodMap}. This yields that  $\xi([P_0\ot_kP_0'])=\xi([P_0])\xi([P_0'])$ in $\p_R(S/R),$ proving that   $\xi$ is a homomorphism of groups. \end{dem}


\section{The  partial generalized crossed product}

\subsection{Partial actions and partial representations}\label{subsec:ParAcParRep}

\begin{defi}
	Let $G$ be a group and $S$  be a semigroup (respectively, a ring). A partial action $\alpha $ of $G$ on $S$ is a family of subsemigroups (respectively, two-sided ideals) $S_g, g\in G$, and  semigroup (respectively, ring) isomorphisms  $\al_g:S_{g^\m}\to S_g,$ which satisfy the following  conditions, for all $g,h\in G$:
	\begin{enumerate}
		\item [$(i)$] $S_1=S$ and $\al_1=Id_S$,
		\item [$(ii)$] $\al_h^\m(S_h\cap S_{g^\m})\subseteq S_{(gh)^\m}$,
		\item [$(iii)$] $\al_g\circ \al_h(s)=\al_{gh}(s)$, for every $s \in \al_h^\m(S_h\cap S_{g^\m}).$
	\end{enumerate} 
\end{defi}
We shall write  $\al=(S_g,\al_g)_{g\in G}$ for a partial action of  $G$ on a semigroup (or a ring) $S.$ 
As seen in   \cite{dokuchaev2005associativity}, conditions   (ii) and (iii) of the above definition imply   that $\al_g^\m=\al_{g^\m}$ and  
\begin{equation}
	\al_g(S_{g^\m}\cap S_{h})=S_g\cap S_{gh}, \ \ \mbox{for all} \ g,h \in G. \label{alxemSxinversoSy}
\end{equation}
{We say that    $\al=(S_g,\al_g)_{g\in G}$ is {\it unital} if each   $S_g$ is an ideal in $S$ generated by an   idempotent which is central in $S$, that is, $S_g=S1_g$,  with $1_g\in \Z (S),$ for all $g \in G.$ In this case, it is clear that   $S_g\cap S_h=S1_g1_h$ and (\ref{alxemSxinversoSy}) implies that  $\al_g(1_h1_{g^\m})=1_g1_{gh},$ for all   $g,h \in G.$ As a consequence, 

$
\al_{gh}(s1_{h^\m g^\m})1_g=\al_g(\al_h(s1_{h^\m})1_{g^\m}), \ \ \mbox{for each } \ g,h \in G \ \mbox{and any} \ s \in S.
$

Let $R$ be a ring. We recall from  \cite[Definition 1.2]{dokuchaev2005associativity}  that
the {\it  partial skew group ring}  $R \star_{\alpha}G$ for the unital partial action  $\alpha$ of $G$ on $R$ is the direct sum $\bigoplus\limits_{g\in G}R_g\delta_g$,
in which the $\delta_g$'s are symbols, with the multiplication defined by the $R$-bilinear extension of the rule:
$$(r_g\delta_g)  (r'_h\delta_h) = r_g\alpha_g(r'_h1_{g\m})\delta_{gh},$$ for all $g,h\in G$, $r_g\in R_g$ and $r_h\in R_h$.  It follows from  \cite[Corollary 3.2]{dokuchaev2005associativity} that $R \star_{\alpha}G$ is an associative ring with  identity $1_R\delta_1,$ moreover one can view $R$  as a subring of $R \star_{\alpha}G$ via the ring monomorphism $R\ni s\mapsto s\delta_e\in R \star_{\alpha}G.$

According to
\cite{dokuchaev2007partial},  the subring $R^{\alpha}:=\{r \in R :\alpha_g(r1_{g^{-1}})=r1_g,\,\text{for all}\,\, g\in G\}$ of a ring $R$ is called the {\it ring  of invariants} of$R$ under the unital partial action $\al,$ moreover
$R$ is an  $\al$-\textit{partial Galois extension} of $R^\alpha$ if  there exist
 $m\in \mathbb{N}$ and elements $x_i,y_i\in R, 1\leq i\leq m$, such that
\begin{equation*}\label{G2}
\sum_{i=1}^mx_i\alpha_{g}(y_i1_{g^{-1}})=\delta_{1, g},\, \text{for each}\, g \in G.
\end{equation*}
The elements $x_i,y_i$  are called \textit{partial Galois coordinates} of $R$ over $R^\alpha$. For a unital partial  action  on a semigroup  the subsemigroup of invariants is defined similarly. \\

\begin{obs}\label{jjota}
It is shown in \cite[Theorem 4.1]{dokuchaev2007partial} that  if $R\supseteq R^\al$ is a partial Galois extension such that $ R^\al \subseteq \Z (R),$  then  $R$ is a f.g.p. $R^\al$-module. Moreover, by  the same theorem it  follows   that for every left $R\star_\al G$-module $M$ the map 
\begin{equation}\label{isop}\omega\colon R\ot_{R^\al} M^G\ni x\ot_{R^\al} m\mapsto  xm\in M
\end{equation}
is a left $R\star_\al G$-module isomorphism, where  
\begin{equation}\label{mg} M^G=\{m\in M\,|\, (1_g\delta_g)m=1_gm, \,\forall  g\in G\},\end{equation} 
 in particular  $R\ot_{R^\al} M^G$ and $M$ are isomorphic as left $R$-modules.

\end{obs}

		\begin{defi}\label{defipartialrepresentation} A partial  representation of $G$ in a monoid $S$ is a map $\et_G\ni g\mapsto \et_g\in S$
	which satisfies the following   properties for all $g,h\in G$:
	\begin{itemize}
		\item [$(i)$] $\et_{1_G}=1_S$,
		\item[$(ii)$]  $\et_g\et_h\et_{h^{-1}}=\et_{gh}\et_{h^{-1}},$ 
		\item [$(iii)$]  $\et_{g^{-1}}\et_g\et_h=\et_{g^{-1}}\et_{gh}$.
	\end{itemize}
\end{defi}
It follows from Definition \ref{defipartialrepresentation} that 
\begin{equation}\label{simeq}\et_g\et_{g^\m}\et_g=\et_g, \,\,g\in G.\end{equation} 
For $g \in G$ denote $\e_g=\et_g\et_{g^{-1}}$.  It follows by  \cite{dokuchaev2000partial} that the $\e_g$'s  are   idempotents such that  $ \e_g \e_h = \e_h \e_g \ \  \mbox{and} \ \
\et_g\e_h=\e_{gh}\et_g,$ for all $g,h\in G$. 
Furthermore,
\begin{equation}\label{prop1}\et_g\et_h=\et_g\et_{g^\m}\et_g\et_h=\et_g\et_{g^\m}\et_{gh}=\e_g\et_{gh}.\end{equation}

\label{sec:FactSet GenParCrossProd}
\label{sec: pcgp}

Let $R$ be a ring, we say that a partial representation   $$\begin{array}{c c c c}
\G: & G& \to & \Pics(R),\\
& g & \to & [\G_g],
\end{array}$$ is {\it unital},  if  there exists a central  idempotent $1_g\in R$ such that $
[\G_g][\G_{g^{-1}}]=[R1_g].$ Thus there is a family of $R$-bimodule isomorphisms
\begin{equation*}
f^\G=\{f_{g,h}^\G:\G_g\ot_R \G_h\to 1_g\G_{gh}, \ g,h\in G\}. 
\end{equation*}

\noindent Following \cite{DoPaPi2}, we say that $f^\T$ is a \textit{factor set} for $\G,$ if $f^\G$ satisfies the  following commutative diagram: 
\begin{equation*}
\xymatrix{ \G_g\ot_R \G_h\ot_R\G_l\ar[r]^{\G_g\ot f^\G_{h,l}}\ar[d]_{f^\G_{g,h}\ot \G_l} & \G_g\ot_R 1_h\G_{hl}\ar@{=}[r] & 1_{gh}\G_g\ot_R \G_{hl}\ar[d]^{f^\G_{g,hl}}\\
	1_g\G_{gh}\ot_R \G_l\ar[rr]_{f^\G_{gh,l}} & & 1_g1_{hl}\G_{ghl}  } \label{diagramaassociatividadeprodutocruzado}
\end{equation*}

Let $f^\G=\{f_{g,h}^{\G}:\G_g\ot_R\G_h\to 1_g\G_{gh},\ g,h \in G\}$ be a factor set for $\G .$ The set $\D(\G)=\bigoplus_{g\in G}\G_g$ with multiplication defined by
\begin{equation}\label{prog}
u_g\stackrel{\G}{\circ} u_h = f_{g,h}^\G(u_g\ot_R u_h)\in 1_g\G_{gh}, \ \ u_g\in \G_g,u_h\in \G_h,
\end{equation}
is called a \textit{partial generalized crossed product associated to the factor set $f^\G$} or simply a \textit{partial generalized crossed product}.

\begin{pro}\cite[Proposition 3.13]{DoRo}\label{pcgpeanelcomutativocom1} Let $\G:G\to \Pics(R)$ be a unital partial representation with $\G_g\ot_R\G_{g^\m}\simeq R1_g$, for all $g\in G$, and let $f^\G=\{f_{g,h}^\G:\G_g\ot_R\G_h\to 1_g\G_{gh}, \ g,h \in G\}$ be   a factor set for  $\G$. Then, the partial generalized crossed product  $\D(\G)$ is an associative ring with unity and $R\simeq \G_1$ is a subring of $\D(\G)$.
\end{pro}


Let $\al=(D_g,\al_g)_{g \in G}$ be a unital partial action of  $G$ on a commutative ring $R$, where $D_g=R1_g$, and  $1_g$ is an idempotent for any  $g \in G$. 
Consider the map
\begin{equation}
\T: G \ni g \mapsto [ (D_g)_{g^\m}]  \in  \Pics_{R^\al}(R)\\
 \label{definicaodeT0}
\end{equation}
where   $ (D_g)_{g^\m}=D_g$ as sets, and  $R$-actions defined by:
\begin{equation*}
r*d=rd \ \ \mbox{and} \ \ d*r=d\al_g(r1_{g^\m}), \ \ \mbox{for all}  \ d \in D_g \ \mbox{and} \ r \in R.
\end{equation*}

By   \cite[Remark 6.4]{DoPaPi2}     and \cite[Proposition 6.2]{DoPaPi}    the map  $\T$ is  a unital partial representation with
\begin{equation}\label{isot} (D_g)_{g^\m}\ot_R(D_{g^\m})_g\simeq D_g,\end{equation} as $R$-bimodules.
Observe that
\begin{equation}
rd=r1_{g}d=\al_{g}(\al_{g^\m}(r1_{g}))d=d*\al_{g^\m}(r1_{g}),\label{reux}
\end{equation}
for all $g\in G, d\in D_g$ and $r\in R.$

The partial representation $\T$ will be very important for us to deal with partial generalized crossed products. Moreover, $\T $ is used to describe the 
unital partial action of $G$ on  $\Pics _R(R),$ as given in the following:

\begin{exe}\label{pacal}(A partial action on  $\Pics _R(R)$) Let   $g\in G$ and $P$ be a central $R$-bimodule with $1_{g^\m}p=p$,  for each  $p\in P$.  It follows from
  \cite[Lemma 3.6]{DoPaPi} that there is a central $R$-bimodule structure $\bullet$ on $P_g$, where  $P=P_g$ as sets, and  
	\begin{equation}\label{cena}
	r\bullet p=\al_{g^\m}(r1_{g})p, \ \ \mbox{for every} \ p\in P, r\in R.
	\end{equation}
	By  \cite[Theorem 3.8]{DoPaPi}, the family  $\al^*=(\mathcal{X}_g,\al_g^*)$,  where $\mathcal{X}_g=\Pics _R (R)[D_g]=\Pics _R (D_g)$ and 
	\begin{equation*}
	\begin{array}{c c c l}
	\al_g^*: & \mathcal{X}_{g^\m} & \longrightarrow & \mathcal{X}_{g}\\
	& [P] & \longmapsto & [P_g]
	\end{array}
	\end{equation*}
	for any $g\in G,$ is a partial action of  $G$ on $\Pics _R (R)$. Clearly, $\U (\mathcal{X} _g ) = \Pic (D_g),$ $g\in G.$ Moreover,  by equality (6.4) in \cite{DoPaPi}  and \cite[Remark 6.4]{DoPaPi2}  we obtain an $R$-bimodule isomorphism
	\begin{equation}\label{cong}P_g\simeq (D_g)_{g^\m}\ot_R P\ot_R(D_{g^\m})_g,\end{equation}
	for all $[P]\in  \mathcal{X}_{g^\m},$ and $g \in G,$ where $[(D_g)_{g^\m}]=\T (g)$.  Explicitly, the isomorphism is obtained by
\begin{equation}\label{isocon}
P_g\ni p\mapsto 1_g\otimes_R p\otimes_R 1_{g^{\m}}\in (D_g)_{g^\m}\ot_R P\ot_R(D_{g^\m})_g  
\end{equation}
with inverse
\begin{equation}\label{isoconi}
 (D_g)_{g^\m}\ot_R P\ot_R(D_{g^\m})_g\ni d_g\otimes_R p\otimes_R d'_{g^{\m}}   \mapsto \alpha _{g^\m}(d_g) p d'_{g^{\m}} \in P_g.
\end{equation}

	Consequently, the partial action $\alpha ^*$ coincides with the partial action constructed in \cite[p.19]{DoRo}. 
	\end{exe}
	
	We proceed with the next:
	
\begin{lem}\label{fs} Let $\al=(D_g,\al_g)_{g \in G}$ be a unital partial action of  $G$ on a commutative ring $R$, as above. Then the family of isomorphisms  $f^\T=\{ f^\T_{g,h}:(D_g)_{g^\m}\ot_R(D_h)_{h^\m}\longrightarrow 1_g(D_{gh})_{(gh)^\m}\}_{(g,h) \in G\times G  }$ defined by 
	\begin{equation}\label{fte}
	f^\T_{g,h}: (D_g)_{g^\m}\ot_R(D_h)_{h^\m}\ni u_g \ot_R u_h   \mapsto u_g\al_g(u_h1_{g^\m}) \in  1_g(D_{gh})_{(gh)^\m}\end{equation}
	is a factor set for the unital partial representation $\T$ defined by (\ref{definicaodeT0}). 
\end{lem}
\begin{dem} It is not difficult to see that $f_{g,h}$ is well-defined  and left $R$-linear. To show that it is right $R$-linear take  $r \in R.$ Then
	\begin{eqnarray*}
		f^\T_{g,h}(u_g\ot u_h*r) & = & u_g\al_g(u_h\al_h(r1_{h^\m})1_{g\m})\\
		& = & u_g\al_g(u_h1_{g^\m})\al_g(\al_h(r1_{h^\m}) 1_{g\m}   )\\
		& = & u_g\al_g(u_h1_{g^\m})\al_{gh}(r1_{(gh)^\m})1_g\\
		& = & u_g\al_g(u_h1_{g^\m})*r\\
		& = & f_{g,h}(u_g\ot_R u_h)*r.
	\end{eqnarray*}

	Therefore, $f^\T_{g,h}$ is an  $R$-bilinear map. We show that the inverse of  $f^\T_{g,h}$ is given by  
	$$
	(f^\T_{g,h})^\m : 1_g(D_{gh})_{(gh)^\m} \ni u_{gh} \mapsto1_g \ot_R \al_{g^\m}(u_{gh}) \in  (D_g)_{g^\m}\ot_R(D_h)_{h^\m}\\
	$$

	Indeed, for   $u_{gh} \in 1_g(D_{gh})_{(gh)^\m}$ we have
	\begin{eqnarray*}
		(f^\T_{g,h}\circ (f^\T_{g,h})^\m)(u_{gh}) & = & f_{g,h}(1_g\ot_R \al_{g^\m}(u_{gh})) = \al_g(\al_{g^\m}(u_{gh}))=  u_{gh}. 
	\end{eqnarray*}

	On the other hand,  given $u_g \in (D_g)_{g^\m}$ and $u_h \in (D_h)_{h^\m}$ we see that 
	\begin{eqnarray*}
		((f^\T_{g,h})^\m\circ f^\T_{g,h})(u_g \ot_R u_h) & = & (f^\T_{g,h})^\m(u_g\al_g(u_h1_{g^\m})) = 1_g\ot_R \al_{g^\m}(u_g\al_g(u_h1_{g^\m}))\\
		& = & 1_g\ot_R \al_{g^\m}(u_g)u_h1_{g^\m} = 1_g*\al_{g^\m}(u_g)\ot_R u_h\\
		& = & 1_g\al_g(\al_{g^\m}(u_g))\ot_R u_h=1_gu_g\ot_R u_h=u_g\ot_Ru_h.
	\end{eqnarray*}
	Thus $f^\T_{g,h}$ is an $R$-bimodule isomorphism, for every $g,h \in G$.  Now we check that the following diagram is commutative.
	
	$$\xymatrix@C=3cm{ (D_g)_{g^\m}\ot_R(D_h)_{h^\m}\ot_R(D_l)_{l^\m}\ar[r]^{(D_g)_{g^\m}\ot_R f^\T_{h,l}} \ar[dd]_{f^\T_{g,h}\ot_R (D_l)_{l^\m}}& (D_g)_{g^\m}\ot_R1_h(D_{hl})_{(hl)^\m}\ar@{=}[d]\\
		&  1_{gh}(D_g)_{g^\m}\ot_R(D_{hl})_{(hl)^\m}\ar[d]^{f^\T_{g,hl}}\ar[d]\\
		1_g(D_{gh})_{(gh)^\m}\ot_R(D_l)_{l^\m}\ar[r]_{f^\T_{gh,l}} & 1_g1_{gh}(D_{ghl})_{(ghl)^\m}.   }$$
Let $u_g \in (D_g)_{g^\m}, u_h \in (D_h)_{h^\m}$ and $u_l \in (D_l)_{l^\m}$. Then
	\begin{eqnarray*}
		(f^\T_{gh,l}\circ (f^\T_{g,h}\ot_R (D_l)_{l^\m}))(u_g\ot_R u_h \ot_R u_l) & = & f^\T_{gh,l}(u_g\al_g(u_h1_{g^\m})\ot_R u_l)\\
		& = & u_g\al_g(u_h1_{g^\m})\al_{gh}(u_l1_{(gh)^\m}),
	\end{eqnarray*}
	and
	\begin{eqnarray*}
		(f^\T_{g,hl}\circ((D_g)_{g^\m}\ot_R f^\T_{h,l}))(u_g\ot_R u_h \ot_R u_l) & = & f^\T_{g,hl}(u_g\ot_R u_h\al_h(u_l1_{h^\m}))\\
		& = & u_g\al_g(u_h\al_h(u_l1_{h^\m})1_{g^\m})\\
		& = & u_g\al_g(u_h1_{g^\m})\al_{gh}(u_l1_{(gh)^\m})1_g\\
		& = & u_g\al_g(u_h1_{g^\m})\al_{gh}(u_l1_{(gh)^\m}).
	\end{eqnarray*}
Hence the diagram  above is commutative, and this shows that  $f^\T$ is a  factor set for $\T$. \end{dem}

 {\it In all what follows} the partial generalized crossed product $\D(\T)$ will be considered with multiplication given by the factor set $f^\T$ described above.

By  Proposition \ref{pcgpeanelcomutativocom1}  and   Lemma \ref{fs} one obtains the partial generalized crossed product 
$\D(\T)=\bigoplus_{g\in G} {(D_g)_{g^\m}},$  with multiplication given by 
\begin{equation*}
u_g\stackrel{\T}{\circ} u_h=u_g\al_g(u_h1_{g^\m}), \ \ \mbox{for } u_g \in (D_g)_{g^\m}, u_h\in (D_h)_{h^\m}.
\end{equation*}

Moreover, we have the next.

\begin{lem}\label{pcgpisoskew} Let $\al=(D_g,\al_g)_{g \in G}$ be a unital partial action of  $G$ on a commutative ring $R$, as above. Then the partial generalized crossed product  $\D(\T)$ is isomorphic to   $R\star_\al G$  as $R$-bimodules and  as $R^\al$-algebras.
\end{lem}
\begin{dem} 
First of all notice that
	$R^\al\subseteq \Z(R\star_\al G)$. Indeed, for  $r \in R^\al, g\in G$ and  $a_g\delta_g \in R\star_\al G$, we have 
$r(a_g\delta_g)=ra_g\delta_g=   r  1_ga_g\delta_g=  a_g \al_g(r1_{g^\m})\delta_g    =(a_g\delta_g)r,$ and thus  $R\star_\al G$ is an $R^\al$-algebra.  Consider the $R$-bimodule isomorphism
$$
\kappa:  \D(\T) \ni u_g \mapsto   u_g\delta_g\in R\star_\al G.$$
Then, given $g,h\in G, u_g \in (D_g)_{g^\m}$ and $u_h\in (D_h)_{h^\m}$  we have
		\begin{eqnarray*}
			\kappa(u_g\stackrel{\T}{\circ}u_h) =  \kappa(u_g\al_g(u_h1_{g^\m}))=u_g\al_g(u_h1_{g^\m})\delta_{gh}
			 =  (u_g\delta_g)(u_h\delta_h)=\kappa(u_g)\kappa(u_h),
		\end{eqnarray*}
this shows that  $\kappa$ is an  $R^\al$-algebra isomorphism. 
		\end{dem}

Let $R\subseteq S$ be an extension of (non-necessarily commutative) rings. Denote by  $\mathcal{S}_R(S)$ the set of the   $R$-subbimodules of $S,$ equipped with  the multiplication
\begin{equation}\label{eq:MonoidS}
MN=\left\lbrace  \dsum _{i=1}^l m_in_i; \ m_i \in M, n_i \in N, l\in \mathbb Z^+\right\rbrace ,
\end{equation} for $M,N\in \mathcal{S}_R(S),$  then  $\mathcal{S}_R(S)$ is  a monoid with neutral element $R.$ 

Let   $\Gamma:  G  \to \mathcal{S}_R(S)$ be  a unital  partial representation,  that is, a partial representation such that  for  $g\in G,$     $\e_g:=
\G_g\G_{g^{-1}}=R1_g$  for some  central idempotent $1_g\in R.$ By \cite[Proposition 3.22]{DoRo}, $[ \G _g ] \in \Pics (R)$ for all $g \in G.$   
Denote
\begin{equation}\label{pz}
	\p_\Z(S/R)^{(G)}=\{\xymatrix@C=1.2cm{ [P]\ar@{=>}[r]|{[\phi]} & [X]}\in \p_\Z(S/R): \ \G_g\phi(P)=\phi(P)\G_g, \ \mbox{for all } \ g\in G\},
\end{equation}
where $\Z$  stands for the center of $R,$ and $\p_\Z(S/R)$ is given by \eqref{pk}. It follows from  \cite[Remark  5.1]{DoRo} that $ \xymatrix@C=1.2cm{ [P]\ar@{=>}[r]|{[\phi]} & [X]}\in \p_\Z(S/R)^{(G)},$ if and only if, $$\G (g) \phi(P)\G ({g^\m})=\phi(P)1_g,$$ for any $g\in G.$ Moreover, we  have the next.
\begin{lem}\cite[Lemma 5.2]{DoRo} \label{lemma:SubgrP^G}  The set $\p_\Z(S/R)^{(G)}$ is a  subgroup of $\p_\Z(S/R)$.
\end{lem}

 Let $\al=(D_g,\al_g)_{g \in G}$ be a unital partial action of  $G$ on $R.$  We recall from   \cite[ Example 3.25]{DoRo}  that the map 
\begin{equation}\label{t0}
\Theta _0 : G\ni g \mapsto D_g \delta _g\in  \mathcal{S}_R( R\star_{\al}G ) 
\end{equation}
is a unital  partial representation of $G$ into the semigroup  
$\mathcal{S}_R( R\star_{\al}G )$ of the $R$-subbimodules of  $ R\star_{\al}G .$ In addition, $\D(\T_0) = R\star_{\al}G.$

{\it In what follows in this work $R$ will denote a commutative ring, and   $\al=(D_g,\al_g)_{g\in G}$  a unital partial action of $G$ on  $R$, with $D_g = 1_g R,$ $g\in G.$ }

\begin{teo}\label{isopic} Suppose  that $ R^\alpha \subseteq R$ is a partial Galois   extension. Let, furthermore,  $S= R\star_{\al}G$ and $\T _0$ the partial representation given  in \eqref{t0}. Then
	\begin{equation}\label{isogr}
	\begin{array}{c c c l}
	\xi: & \Pic_{R^\al}(R^\al) & \longrightarrow & \p_R(S/R)^{(G)}\\
	&  [P_0] & \longmapsto & \xymatrix@C=1.2cm{  [P_0\ot_{R^\al} R]\ar@{=>}[r]|-{[\phi]} & [P_0\ot_{R^\al}S]},
	\end{array}
	\end{equation}
is an isomorphism of groups,	where  $\phi:P_0\ot_{R^\al}R\longrightarrow P_0\ot_{R^\al}S$ is the inclusion and the  $S$-bimodule structure of  
$P_0\ot_{R^\al}S$ is defined by \eqref{eq:bimodStructureOnTensorPr}.
	
	\end{teo}
\begin{dem} Since
$ (P_0\ot_{R^\al} R) D_g \delta _g = D_g \delta_g (P_0\ot_{R^\al} R), $
we have $([P_0\ot_{R^\al}R]=[\phi]\R[P_0\ot_{R^\al}S]) \in \p_R(S/R)^{(G)}.$
	But   $R^\al\subseteq \Z(R\star_{\al}G)$, thus one gets  from  Proposition \ref{morfismoentrePicePSR} that  $\xi$ is a homomorphism of groups.
	Let  $[P_0]\in \Pic _{R^\al} (R^\al)$ be such that  $([P_0\ot_{R^\al}R]=[\phi]\R[P_0\ot_{R^\al}S])=([R]=[\iota]\R[S])$ in $\p_R(S/R)^{(G)}$. Then, $P_0\ot_{R^\al}R\simeq R\simeq R^\al\ot_{R^\al}R$ as $R$-bimodules. Since $R$ is a f.g.p. $R^\al$-module we have from  Lemma \ref{MotNisoMotN'implicaNisoN'} that $P_0\simeq R^\al$ as $R^\al$-modules. Then  $[P_0]=[R^\al]$ in  $\Pic _{R^\al}(R^\al)$ and  $\xi$ is a monomorphism. To show that it is surjective   take $([P]=[\phi]\R[X])\in \p_R(S/R)^{(G)}$. Then $ \T_0 (g) \phi(P)\T_0 ({g^\m})=\phi(P1_g)$, for all $g \in G$. Now we show that the equality
	\begin{equation}\label{triangle}
	(a_g\delta_g)\cdot p=\phi^{\m}((a_g\delta_g)\phi(p)(1_{g^\m}\delta_{g^\m})) \in P1_g, 
	\end{equation}
	$p \in P,$ endows  $P$ with a left  $S$-module structure.

Indeed, take $a_g\delta_g, a_h\delta_h \in S$ and  $p \in P$ then
	\begin{eqnarray*}
		(a_g\delta_g)\cdot ((a_h\delta_h)\cdot p) & = & (a_g\delta_g) \cdot (\phi^{-1}((a_h\delta_h)\phi(p)(1_{h^\m} \delta_{h^\m})))\\
		& = & \phi^{\m}((a_g\delta_g)\phi(\phi^{-1}((a_h\delta_h)\phi(p)(1_{h^\m} \delta_{h^\m}))(1_{g^\m}\delta_{g^\m}))\\
		& = &  \phi^{\m}((a_g\delta_g)(a_h\delta_h)\phi(p)(1_{h^\m}\delta_{h^\m})(1_{g^\m}\delta_{g^\m}))\\
		& = & \phi^{-1}(a_g\al_g(a_h 1_{g^\m})\delta_{gh}\phi(p)1_{h^\m}\al_{h^\m}(1_{g^\m}1_h)\delta_{h^\m g^\m})\\
		& = & \phi^{-1}(a_g\al_g(a_h1_{g^\m})\delta_{gh}\phi(p)1_{h^\m}1_{h^\m g^\m}\delta_{h^\m g^\m})\\
		& = & \phi^{-1}(a_g\al_g(a_h1_{g^\m})\delta_{gh} { 1_{h^\m} \delta_1} \phi(p)1_{h^\m g^\m}\delta_{h^\m g^\m})\\
		& = & \phi^{-1}(a_g\al_g(a_h1_{g^\m})\al_{gh}(1_{h^\m}1_{h^\m g^\m})\delta_{gh}\phi(p)1_{h^\m g^\m}\delta_{h^\m g^\m})\\
		& = & \phi^{-1}(a_g\al_g(a_h1_{g^\m})1_g1_{gh}\delta_{gh}\phi(p)1_{h^\m g^\m}\delta_{h^\m g^\m})\\
		& = & \phi^{-1}(a_g\al_g(a_h1_{g^\m})\delta_{gh}\phi(p)1_{h^\m g^\m}\delta_{h^\m g^\m})\\
		& = & a_g\al_g(a_h1_{g^\m})\cdot p\\
		& = & [(a_g\delta_g)(a_h\delta_h)]\cdot p.
	\end{eqnarray*}
	Then  there is a left   $R$-module isomorphism $ \omega:  R\ot_{R^\al}P^G  \to P$ given by \eqref{isop}. Since $P$ and $R$ are f.g.p.  $R^\al$-modules, it follows from  Lemma \ref{MMotNpfgimplicaNpfg} that  $P^G$ is also a f.g.p.  $R^\al$-module. We show now that  $[P^G]\in \Pic_{R^\al}(R^\al)$. For this purpose, localizing at a prime ideal  ${\mathfrak p}$ of $R^\al ,$ we may assume that $R^\al $ is a local ring. Then 
		$R$ is semilocal,   and $\Pic(R)$ is trivial (see  \cite[p. 38 (D)]{L}). Consequently, $P \simeq R,$ and the isomorphism  $ \omega$  implies that
	\begin{equation*}
	{\bf rk}_{R^\al}(R)={\bf rk}_{R^\al}(P)={\bf rk}_{R^\al}(R){\bf rk}_{R^\al}(P^G). \label{postokdePG}
	\end{equation*}
	This yields that  $\mbox{rank}_{R^\al}(P^G)=1$. Since ${\mathfrak p}$ was chosen arbitrarily, it follows that  $[P^G]\in \Pic_{R^\al}(R^\al)$.
	Then, 
	$$\xi([P^G])=([P^G\ot_{R^\al}R]=[\psi]\R[P^G\ot_{R^\al}S]),$$
	where $\psi$ is the  inclusion.
	
	Consider the   $S$-$R$-bimodule  isomorphism  given by 
	$$\nu :S\ot_{R^\al}P^G\longrightarrow S\ot_RR\ot_{R^\al}P^G\stackrel{S\ot \omega}{\longrightarrow} S\ot_RP\stackrel{\overline{\phi_l}}{\longrightarrow} X,$$
	where $\omega$ is given by (\ref{isop}). Then, 
	$$\nu ((a_g\delta_g)\ot_{R^\al} p)=(a_g\delta_g)\phi(p), \ \mbox{where}\  a_g \in D_g,\,\,p\in P^G.$$
	
	\begin{afr} If $p \in P^G$,  the equality $(a_g\delta_g)\phi(p)=\phi(p)(a_g\delta_g)$, holds for all $g \in G$ and $a_g \in D_g.$
	\end{afr}
	Indeed, since $p \in P^G$, we have
	\begin{equation*}
	{ p1_g = (1_g\delta_g) \cdot p = \phi^{-1}((1_g\delta_g)\phi(p)(1_{g^\m}\delta_{g^\m}))},  \ \ \mbox{for all } \ g\in G.
	\end{equation*}
	Then for any $g\in G$ one gets
	\begin{eqnarray*}
		& \Rightarrow & (1_g\delta_g)\phi(p)(1_{g^\m}\delta_{g^\m})=\phi(p1_g),\\
		& \R & (1_g\delta_g)\phi(p)(1_{g^\m}\delta_{g^\m})(1_g\delta_g)=\phi(p1_g)(1_g\delta_g),\\
		& \R & (1_g\delta_g)\phi(p)(1_{g^\m}\delta_1)=\phi(p)(1_g\delta_1)(1_g\delta_g),\\
		& \R & (1_g\delta_g)(1_{g^\m}\delta_1)\phi(p)=\phi(p)(1_g\delta_g),\\
		& \R & (1_g\delta_g)\phi(p)=\phi(p)(1_g\delta_g).
	\end{eqnarray*}
	
	Which gives, 
	\begin{eqnarray*}
		(a_g\delta_g)\phi(p) & = & (a_g\delta_1)(1_g\delta_g)\phi(p)=(a_g\delta_1)\phi(p)(1_g\delta_g)
		 =  \phi(p)(a_g\delta_1)(1_g\delta_g)=\phi(p)(a_g\delta_g),
	\end{eqnarray*}
	for any  $g \in G$. 	Now we check that $\nu $ is right  $S$-linear.
	\begin{eqnarray*}
		 \nu (((a_g\delta_{g})\ot_{R^\al} p)\cdot (a_h\delta_h))  & = & \nu ((a_g\delta_{g})(a_h\delta_h)\ot_{R^\al} p) = (a_g\delta_{g})(a_h\delta_h)\phi(p)\\
		&  = & (a_g\delta_{g})\phi(p)(a_h\delta_h) = \nu ((a_g\delta_g)\ot _{R^\al}p)(a_h\delta_h),
	\end{eqnarray*} keeping in mind  \eqref{eq:bimodStructureOnTensorPr2} for the right $S$-action on $P.$
	Thus, $\nu $ is an isomorphism of $S$-bimodules. 
	
	Finally, observe that
	$\nu (\psi(r\ot_{R^\al} p))=\nu (r\delta_1\ot_{R^\al} p)=(r\delta_1)\phi(p)=\phi(rp)=\phi(\om(r\ot_{R^\al} p)),$
	para todo $r \in R$ e $p\in P^G$. Then, the diagram
	$$\xymatrix{  R\ot_{R^\al}P^G\ar[dd]_{\omega}\ar[rr]^{\psi} & & S\ot_{R^\al}P^G\ar[dd]^{\nu}\\
		& & \\
		P\ar[rr]_{\phi} & & X}$$
	is commutative. 
	Thus, 
	$$\xi([P^G])=([P^G\ot_{R^\al}R]=[\psi]\R[P^G\ot_{R^\al}S])=([P]=[\phi]\R[X]).$$
	then  $\xi$ is an  epimorphism, and thus a group isomorphism.
	\end{dem}

\subsection {Partial cohomology of groups}\label{sec:ParCohomol}
The notion of partial group cohomology was introduced and studied in \cite{dokuchaev2015partial}. 
Let $R$ be a commutative ring (or a commutative monoid) and $\al=(D_g,\al_g)_{g\in G}$ be a unital partial action of $G$ on $R$, where each $D_g$ is generated by the central idempotent  $1_g$. An $n$-\textit{cochain}, with $n\in \N$, of $G$ with values in $R$ is a function $f:G^n\to R$ such that $f(g_1,...,g_n)\in \U(R1_{g_1}1_{g_1g_2}...1_{g_1g_2...g_n})$. A $0$-cochain as  an element in $\U(R)$. 
The  set $C^n(G,\al,R)$   of all $n$-cochains of $G$ with values in  $R$ is an abelian group with the multiplication defined point-wise,  whose identity element is $I(g_1,...,g_n)=1_{g_1}1_{g_1g_2}...1_{g_1g_2...g_n}$ and $f^\m(g_1,...,g_n)=f(g_1,...,g_n)^\m \in \U(R1_{g_1}1_{g_1g_2}...1_{g_1g_2...g_n})$, for $g_1,...,g_n \in G,$ is the inverse of $f\in C^n(G,\al, R).$

We recall the next.

\begin{pro}\label{prodeltan} \cite[Proposition 1.5]{dokuchaev2015partial} The map  
$\delta^n:C^n(G,\al,R)\to C^{n+1}(G,\al,R)$ defined by 
\begin{eqnarray*}
(\delta^nf)(g_1,...,g_{n+1}) & = & \al_{g_1}(f(g_2,...,g_{n+1})1_{g_1^\m})\prod_{i=1}^nf(g_1,...,g_ig_{i+1},...,g_{n+1})^{(-1)^i}\nonumber\\
& & f(g_1,...,g_n)^{(-1)^{n+1}}, \label{definicaodedeltapequeno}
\end{eqnarray*} for any $f\in  C^n(G,\al,R)$ and $n>0,$ and for $n=0$  by $$
(\delta^0x)(g)=\al_{g}(x1_{g^\m})x^\m, $$ where $ x \in \U(R), $ is a   group morphism such that
$$
	(\delta^{n+1}\circ \delta^n)(f)(g_1,...,g_{n+2})=1_{g_1}1_{g_1g_2}...1_{g_1g_2...g_{n+2}},
$$
	for every $f \in C^n(G,\al,R)$ and $g_1,...,g_{n+1}\in G$.
\end{pro}

We consider the groups
$
Z^n(G,\al,R)=\ker(\delta^n) \ \ \mbox{and} \ \ B^n(G,\al,R)=\mbox{Im}(\delta^{n-1})
$
that are the group of  \textit{  the partial $n$-cocycles} and that of the \textit{partial $n$-coboundaries}, respectively.  
By Proposition~\ref{prodeltan}, we have $B^n(G,\al,R)\subseteq Z^n(G,\al,R)$. Thus, for $n>0$
we define the group of the partial $n$-cohomologies by 
\begin{equation*}
H^n(G,\al,k)=\dfrac{Z^n(G,\al,R)}{B^n(G,\al,R)}.
\end{equation*}
For $n=0$ we set $H^0(G,\al,R)=Z^0(G,\al,R)=\ker(\delta^0)$.

\begin{exe}
	For $n=0$  we have $$H^0(G,\al, R)=Z^0(G,\al, R)=\{x\in \U(R), \al_g(x1_{g^\m})=x1_g, \ \forall \ g\in G\},$$
	\begin{equation}\label{b1}B^1(G,\al,R)=\mbox{Im}(\delta^0)=\{f \in C^1(G,\al,R); \exists \ x \in \U(R) \ \ f(g)=\al_g(x1_{g^\m})x^\m, \ \forall \ g\in G\}.\end{equation}
	
	If $n=1$, then  $(\delta^1f)(g,h)=\al_g(f(h)1_{g^\m})f(gh)^\m f(g)$, for every  $f \in C^1(G,\al,R)$ which implies,
	\begin{equation}\label{z1}Z^1(G,\al,R)=\{ f\in C^1(G,\al,R); \al_g(f(h)1_{g^\m})f(g)=f(gh)1_g, \ \forall \ g,h \in G \},\end{equation}
	$$B^2(G,\al,R)=\{ f\in C^2(G,\al,R); \exists \ \s\in C^1(G,\al,R), \ \mbox{with} \ f(g,h)\s(gh)=\al_g(\s(h)1_{g^\m})\s(g), \ \forall \ g,h \in G \}.$$
	For $n=2$, $$(\delta^2f)(g,h,l)=\al_g(f(h,l)1_{g^\m})f(gh,l)^\m f(g,hl) f(g,h)^\m,$$  for each   $f \in C^2(G,\al,R)$ and $g,h,l \in G.$ Then,
	\begin{equation}\label{z2}Z^2(G,\al, R)=\{f \in C^2(G,\al,R); \al_g(f(h,l)1_{g^\m})f(g,hl)=f(gh,l)f(g,h), \ \forall \ g,h,l \in G \}.\end{equation}
	
\end{exe}

Let  $f,f'\in Z^n(G,\al,R)$, we say that $f$ and $f'$ are \textit{cohomologous} if there is $g \in C^{n-1}(G,\al,R)$ such that $f=f'(\delta^ng)$. In this case,  $[f]=[f']$ in  $H^n(G,\al,R)$.  

An $n$-cocycle $f$ is called \textit{normalized}, if   for any  $g_1,...,g_{n-1}\in G, n>1$
\begin{equation}\label{eq:normalized}
f(1,g_1,...,g_{n-1})=f(g_1,1,...,g_{n-1})=...=f(g_1,...,g_{n-1},1)=1_{g_1}1_{g_1g_2}...1_{g_1...g_{n-1}}.
\end{equation}

\begin{obs}\label{coco}By \cite[Remark 3.6]{DoRo} every $1$-cocycle  is   normalized. 
Moreover, it follows by  \cite[Remark 2.6]{dokuchaev2015partial}  that  if  $f\in Z^2(G,\al, R)$ then there is  
a normalized $\widetilde{f}\in Z^2(G,\al,R)$  such that $\s=\widetilde{\s}(\delta^1\epsilon)$, 
for some $\epsilon \in C^1(G,\al, R)$. Therefore, $\s$ is cohomologous to a normalized partial  $2$-cocycle.	\end{obs}

 Let $f\in Z^2(G,\al,R).$ The {\it partial crossed product} is the  abelian group  $R\star_{\al, f}G=\bigoplus\limits_{g\in G}D_g\delta_g$  with the multiplication given by 
$$(a_g\delta_g)(b_h\delta_h)=a_g\alpha_g(b_h1_{g^\m})f(g,h)\delta_{gh},$$ for all $g,h\in G, a_g\in D_g$ and $b_h\in D_h.$ Then  $R\star_{\al, f}G$ is an associative ring with identity $1_{R\star_{\al, f}G} =f(1,1)^\m\delta_1.$ Indeed, for $g\in G$ and $a_g\in D_g$ one has that
\begin{align*} (a_g\delta_g)(f(1,1)^\m\delta_1)=a_g\alpha_g(f(1,1)1_g)^\m f(g,1)\delta_g=a_g\delta_g,
\end{align*} where the last equality follows by taking $h=l=1$ in \eqref{z2}. In a similar way, one shows $ (a_g\delta_g)=(f(1,1)^\m\delta_1)(a_g\delta_g).$
The associativity of $R\star_{\al, f}G$ can be explained in various ways referring to known results for twisted partial crossed product with normalized $2$-cocycle $f$. One way is to refer the reader to a general result on associativity established in \cite[Theorem 2.4]{dokuchaev2008crossed}, observing that the proof does not use the normalizing condition \eqref{eq:normalized}. Another one is to show that the above defined $R\star_{\al, f}G$ is isomorphic (as a non-necessarily associative ring) to $R\star_{\al, f'}G,$ where $f' $ is a normalized $2$-cocycle. The latter follows from the next fact, which will be also used later.

\begin{pro} \label{coiso} Let $f, f'\in Z^2(G,\al,R)$ be  two cohomologous cocycles, then  the (non-necessarily associative)   $R^\alpha$-algebras  $R\star_{\al, f}G$ and $R\star_{\al, f'}G$ are isomorphic.
\end{pro}
\begin{dem}  Take $\varepsilon \in C^1(G,\al, R)$ such that $f=f'(\delta^1\varepsilon).$ Then the  map $R\star_{\al, f}G\ni a_g\delta_g\mapsto a_g\varepsilon_g\delta_g\in R\star_{\al, f'}G,$ is readily seen to be a 
$R^\alpha$-algebra homomorphism with inverse $R\star_{\al, f'}G\ni a_g\delta_g\mapsto a_g\varepsilon^\m_g\delta_g\in R\star_{\al, f}G.$
\end{dem}

\begin{cor}\label{cor:assoc}  Let $\alpha $ be as above. Then for any  $f \in Z^2(G,\al,R)$ the $R^\alpha$-algebra  $R\star_{\al, f}G$ is associative.
\end{cor}
\begin{dem} By Remark~\ref{coco} $f$ is cohomologous to a normalized $2$-cocycle  $f' \in Z^2(G,\al,R).$ By  \cite[Theorem 2.4]{dokuchaev2008crossed} the ring $R\star_{\al, f}G$ is associative and the associativity of  $R\star_{\al, f}G$ follows from Proposition~\ref{coiso}.
\end{dem}

\begin{obs}\label{coigual}
Consider the unital partial   representation  $\T$ defined in  (\ref{definicaodeT0}).  Notice that the map 
$$ \zeta_g:  (D_g)_{g^\m}\ot_R (D_{g^\m})_g  \ni a_g\ot_R b_{g^\m} \mapsto  a_g\al_g(b_{g^\m}) \in  D_g$$
is a  $R$-bimodule isomorphism with inverse $D_g\ni a_g  \mapsto  a_g \ot _R  1_{g^\m} \in (D_g)_{g^\m}\ot _R(D_{g^\m})_g .$
On the other hand, the partial action $\overline{\al}$ induced by    $\T$ as in  \cite[Proposition 3.11]{DoRo}, coincides with  $\al$.  Indeed,  since   $\zeta_g(1_g\ot_R 1_{g^\m})=1_g\al_{g}(1_{g^\m})=1_g$,  then 
	\begin{equation*}
\overline{\al_g}(r1_{g^\m})= \zeta_g(1_g\ot r1_{g^\m})=1_g\al_g(r1_{g^\m})=\al_g(r1_{g^\m}),
\end{equation*} for any $r\in R$ and $g\in G.$  From this one concludes that
	$H_\T^n (G,\overline{\al},R) =H^n(G,\al,R),  $ for all $n\in \N.$
\end{obs}

\subsection{Partial generalized crossed products and the group $\C(\T/R)$}\label{sec:groupC}

 We shall work with the group  $\C(\T/R)$ from \cite{DoRo}, defined using the partial representation $\T$  from \eqref{definicaodeT0}, and for which we first recall a couple of notions.

Let $M$ and $N$ be $R$-bimodules. We  write  $M|N$ if $M$ is isomorphic, as an $R$-bimodule, to a direct summand 
of some direct power  of $N$, that is, if there exists  an $R$-bimodule $M'$ and an $R$-bimodule isomorphism  $N^{(n)}\simeq M\oplus M'$, for some   $n \in \N$.  It is clear that the relation $| $   is  reflexive, transitive and compatible with the tensor product, in the sense that,
if    $M|N$ and $Q$ is an  $R$-bimodule, then 
\begin{equation*}
	(M\ot_RQ)|(N\ot_RQ) \ \ \mbox{and} \ \ (Q\ot_RM)|(Q\ot_RN). \label{combatibilidadecomopt}
\end{equation*}

We proceed with  the next.
\begin{defi}\cite{DoRo} Let $\D(\Omega)$ and $\D(\Gm)$ be  partial generalized crossed products with $R$-bimodule and ring isomorphism   $\iota:R\longrightarrow \Omega _1$ and ${\iota}':R\longrightarrow \Gm_1$, respectively,  given by Proposition \ref{pcgpeanelcomutativocom1}.  A morphism of partial generalized crossed products  $F:\D(\Omega)\longrightarrow \D(\Gm)$ is a set of $R$-bimodule morphisms  $\{F_g:\Omega_g\longrightarrow \Gm_g\}_{ g \in G }$ such that $F_e\circ \iota={\iota}' ,$ and  the following  commutative diagram is satisfied:
	\begin{equation}
	\xymatrix{ \Omega_g\ot_R\Omega_h\ar[rr]^{f_{g,h}^\Omega}\ar[d]_{F_g\ot F_h} & & 1_g\Omega_{gh}\ar[d]^{F_{gh}}\\
		\Gm_g\ot_R\Gm_h\ar[rr]_{f_{g,h}^\Gm} & & 1_g\Gm_{gh} } \label{morfismodepcgp0}
	\end{equation}
	A morphism $F$ of partial generalized crossed products is called an isomorphism if each  morphism $F_g:\Omega_g\longrightarrow \Gm_g$ is an $R$-bimodule isomorphism.  We denote by  $[\D(\G)]$ the isomorphism class of $\D(\G).$
\end{defi}

\begin{obs} \label{sod}It is a consequence of  \cite[Remark 3.20]{DoRo}  that if  $F=\{F_g:\Om_g\to \Gm_g\}_ {g\in G }$  is a family of $R$-bimodule isomorphism satisfying the commutative diagram (\ref{morfismodepcgp0}), then $F$ is an isomorphism of partial generalized crossed products.
\end{obs}
The following result shall be needed in the sequel.

\begin{pro}\label{isomorfismoT}\cite[Corollary 3]{miyashita1973exact},\cite[Proposition 1.3]{el2012invertible} Let $S$ be a ring and  $M$ and $N$ be $S$-bimodules. Suppose that $M|S$ and $N|S$,  and let $f_i:M\rightarrow S$, $g_i:S\rightarrow M$,  $1\leq i\leq n$, two  $S$-bilinear maps,     such that  $\sum\limits_{i=1}^ng_if_i={\rm id}_M.$ Then the function,
$$T_{M,N}:  M\ot_SN \ni m\ot_S n  \mapsto  \dsum_{i=1}^nf_i(m)n\ot_S g_i(1)\in  N\ot_SM$$   is an $S$-bimodule isomorphism. If $N$ is $S$-central, then  $T_{M,N}:  M\ot_SN \ni m\ot_S n  \mapsto n\ot_S m\in  N\ot_SM.$  \end{pro}

Let $\C(\T/R)$ be  set of isomorphism classes $[\D(\Gamma)]$ of the generalized crossed products $\D(\Gamma),$ where $\Gamma :G \to  \Pics (R)$ are partial representations 
such that
\begin{equation}\label{eq:Gamma}   
 \Gamma_g|(D_g)_{g^\m} \,\,\,\,  \mbox{and} \,\,\,\,  \Gamma_g\ot_R\Gamma_{g^\m}\simeq D_g,  \;\;\;  \mbox{as} \; R\mbox{-bimodules},
\end{equation}
for all $g  \in G.$ Clearly, $\Gamma$ is a unital partial representation.
We recall from   \cite[Theorem 4.1]{DoRo} that
$\C(\T/R) $  is a group with product.
\begin{equation}\label{prodc}
[\D(\Gamma)] [\D(\Omega)]=\left[\bigoplus_{g\in G}\Gamma_g\ot _R(D_{g^\m})_g \ot_R\Omega_g\right],
\end{equation}
where the factor set for the unital partial  representation $\Lambda: G\ni g\mapsto [\G_g\ot_R(D_{g^\m})_g\ot_R\Om_g]\in \Pics(R),$
is determined in \cite[Lemma 3.18]{DoRo} as follows:
Let
	\begin{equation*}
		f^\Gm=\{f^\Gm_{g,h}:\Gm_g\ot_R \Gm_h\to 1_g\Gm_{gh}, g, h\in G\}, \ f^\Om=\{f^\Om_{g,h}:\Om_g\ot_R \Om_h\to 1_g\Om_{gh}, \ g, h\in G \}
	\end{equation*}
	be factor sets for  $\Gm$ and $\Om$, respectively,  then a factor set for  $\Lambda$ is given by  the family of $R$-bimodule isomorphisms  
	$f^\Lambda=\{f_{g,h}^\Lambda:\Lambda_g\ot_R \Lambda_h\to 1_g\Lambda_{gh}, \ g, h\in G \}$ defined by 
	
	\begin{equation}\label{fsl}
		\xymatrix@C=1.0cm{  \Gamma_g\ot_R \T_{g^\m}\ot_R \Omega_g\ot_R \Gamma_h\ot_R \T_{h^\m}\ot_R \Omega_h \ar[rd]^{\ \ \ \ \ \  \ \ \ \ \Gamma_g\ot T_{\T_{g^\m}\ot\Omega_g,\Gamma_h\ot \T_{h^\m}}\ot \Om_h}\ar@/_1.7cm/@{-->}[rddd]_{f^{\Lambda}_{g,h}} & & \\
			&  \Gamma_g\ot_R\Gamma_h\ot_R \T_{h^\m} \ot_R \T_{g^\m}\ot_R \Omega_g \ot_R\Omega_h\ar[d]^{f_{g,h}^\Gamma\ot f_{h^{\m},g^\m}^\T\ot f_{g,h}^\Omega} \\
			& 1_g\Gamma_{gh}\ot_R 1_{y^\m}\T_{(gh)^\m}\ot_R 1_g\Omega_{gh}\ar@{=}[d]  \\
			&  1_g\Gamma_{gh}\ot_R \T_{(gh)^\m}\ot_R \Omega_{gh},  }
\end{equation}
where $T_{-,-}$ is the   isomorphism from Proposition~\ref{isomorfismoT}.  

Let $f\in Z^1(G,\al^*,\Pics _R (R))$.  It follows by   \cite[Lemma 6.3]{DoPaPi}  that  $\Gamma=f\T$, is a partial  representation with an $R$-bimodule isomorphism   
\begin{equation}\label{prop4}\Gm_g\ot_R\Gm_{g^\m}\simeq D_g, \,\,  \mbox{for every} \ g\in G. \end{equation} 

 However,  the map   $\Gamma$ may not have an associated  factor set, because in general the diagram (6.9) in \cite{DoPaPi} is not commutative. The following result characterizes the group  $\C(\T/R)$ in terms of the cocylces  $f\in Z^1(G,\al^*,\Pics _R (R)),$ for which $\Gamma=f\T$ is endowed with a factor set.

\begin{teo}\label{prop:grupoC} $\C(\T/R)=\{[\D(\Gamma)]: \exists  f \in Z^1(G,\al^*,\Pics _R (R))\; \text{such that} \;  \Gamma=f\T \}$.	
\end{teo} 
\begin{dem}
	Let  $\D(\Gamma)$ be a partial generalized crossed  product, and let   $f \in Z^1(G,\al^*,\Pics _R (R))$ such that  $\Gamma=f\T.$ Write $\Gamma(g)=[\Gamma_g]$ and $f(g)=[M_g]$, for any $g\in G$. 
	Since $f(g)\in \Pics _R (R),$ the $R$-module $M_g$ is projective and finitely generated and, consequently, $M_g|R$ as $R$-bimodules. 
	Then, 
	$$\Gamma_g= (M_g\ot _R  (D_g)_{g^\m})  | (R\ot _R (D_g)_{g^\m})  \simeq (D_{g})_{g^\m}, \ \mbox{for every} \ g\in G.$$
	By \eqref{prop4} we have that  $\Gamma_g\ot _R\Gamma_{g^\m}\simeq D_g$, for each $g\in G$. We conclude that $[\D(\Gamma)] \in \C(\T/R).$

Conversely, take  $[\D(\Gamma)] \in \C(\T/R)$. Since  $\Gamma_g|(D_g)_{g^\m},$ it follows by \cite[Remark 2.8]{DoRo} that there are $R$-bilinear  maps  $f_i:\Gamma_g\to(D_g)_{g^\m}$ and  $f'_i:(D_g)_{g^\m}\to \Gamma_g, \,i=1,2,...,n$,  such that $\dsum_{i=1}^n f'_if_i=Id_{\Gamma_g}$. This implies,
	\begin{equation}
	rv_g=v_g\al_{g^\m}(r1_g), \ \ \mbox{for all} \ v_g \in \Gamma_g \ \mbox{and}\ r \in R. \label{vxeremCcasogaloiscomutativo}
	\end{equation}
	Indeed, 
	\begin{eqnarray*}
		rv_g & = & r\dsum_{i=1}^n f'_i(f_i(v_g)) = \dsum_{i=1}^n f'_i(rf_i(v_g))\nonumber \\
		& \stackrel{(\ref{reux})}{=} & \dsum_{i=1}^n f'_i(f_i(v_g)*\al_{g^\m}(r1_{g}))\nonumber \\
		& = & \dsum_{i=1}^n f'_i(f_i(v_g))\al_{g^\m}(r1_{g})\nonumber\\
		& = & v_g\al_{g^\m}(r1_{g}). \label{revx}
	\end{eqnarray*}
	Set $M_g=\Gamma_g\ot_R(D_{g^\m})_{g}$, then for  $v_g \in \Gamma_g$ and $u_{g^\m}\in (D_{g^\m})_g$ we see that 
	\begin{equation*}
	(v_g\ot _R u_{g^\m})*r=v_g\ot_R u_{g^\m}\al_{g^\m}(r1_g)=v_g\al_{g^\m}(r1_g)\ot_R u_{g^\m}\stackrel{(\ref{vxeremCcasogaloiscomutativo})}{=} rv_g\ot _Ru_{g^\m},
	\end{equation*}
	for any $r\in R$. Hence,  $M_g$ is a central $R$-bimodule. Analogously, it is easy to see using  	\eqref{vxeremCcasogaloiscomutativo} that $M'_g=(D_{g})_{g^\m}\ot_R \Gamma_{g^\m}$ is also a  central $R$-bimodule.

	For $g \in G$, we shall show that $[M_g]\in   \Pic_R(D_g).$  Due to \eqref{isot} and \eqref{prop4} the following  $R$-bimodule isomorphisms hold:
	\begin{eqnarray*}
		M_g\ot_R M'_g & = & \Gamma_g\ot_R (D_{g^\m})_g\ot _R(D_{g})_{g^\m}\ot_R \Gamma_{g^\m}\\
		&\simeq  & \Gamma_g\ot_R D_{g^\m}\ot _R\Gamma_{g^\m}\\
		& \simeq & \Gamma_g\ot _R\Gamma_{g^\m}\simeq D_g. 
	\end{eqnarray*}
	Analogously, 
	\begin{eqnarray*}
		M'_g\ot_R M_g & = & (D_{g})_{g^\m}\ot_R \Gamma_{g^\m}\ot_R \Gamma_g\ot_R (D_{g^\m})_g\\
		& \simeq & (D_{g})_{g^\m}\ot_R D_{g^\m}\ot_R (D_{g^\m})_g\\
		& \simeq & (D_{g})_{g^\m}\ot_R (D_{g^\m})_g\simeq D_g.\\
	\end{eqnarray*}
	Then $[M_g]\in \Pic_R(D_g),$ and thus $[M_g]\in \Pics_R(R),$ thanks to \eqref{upic}.  Define,
	$$
	f: G \ni g\mapsto [M_g] \in \Pics _R (R),$$
then $f(g)\in \U(\mathcal{X}_g),$ for all $g\in G,$ (see Remark \ref{pacal}). We shall show that   $f \in Z^1(G,\al^*,\Pics_R(R)).$ First of all, observe that in $\Pics(R)$ we have
\begin{equation}\label{prop2}f(g)=[M_g]=[M_g][D_g]\stackrel{\eqref{isot}}=[M_g]\T(g)\T(g^\m)=[\Gamma_g]\T(g^\m)= \Gamma(g)\T(g^\m) .\end{equation}
 Hence,  for  $g,h \in G$ we get
	\begin{eqnarray*}
		f(g)\al_g^*(f(h)[D_{g^\m}]) & \stackrel{\eqref{cong},\eqref{prop2}}= & \Gamma(g)\T(g^\m)\T(g)f(h)\T(g^\m)\\
		&\stackrel{\eqref{prop2}} = & \Gamma(g)\T(g^\m)\T(g)\Gamma(h)\T(h^\m)\T(g^\m)\\
		& \stackrel{\eqref{prop1}}= & \Gamma(g)[D_{g^\m}]\Gamma(h)[D_{h^\m}]\T((gh)^\m)\\
		&\stackrel{\eqref{eq:Gamma}} = & \Gamma(g)\Gamma(h)\T((gh)^\m)\\
		& = & [D_g]\Gamma(gh)\T((gh)^\m)\\
		& \stackrel{\eqref{prop2}} = & [D_g]f(gh).
	\end{eqnarray*}
	Hence by \eqref{z1} we have $f \in Z^1(G,\al^*,\Pics_R(R))$.  Finally, take  $g\in G,$ then $$f(g)\T(g)=\Gamma(g)\T(g^\m)\T(g)=\Gamma(g) [D_{g^\m}]=\Gamma(g).$$ Therefore, $\Gamma=f\T$, for some $f \in Z^1(G,\al^*,\Pics_R(R)),$ as desired.
	\end{dem}
	
Let $\Gamma:G\to \Pics(R)$ be a unital partial representation such that $[\D(\Gamma)]\in \C(\T/R).$ Then, by Theorem~\ref{prop:grupoC}, the partial representation  $\Gamma $ is of the form $  f \T$ for some $1$-cocycle 
 $f \in Z^1(G,\al^*,\Pics _R (R)),$ and since $f(h) \in \Pic (D_h),$ we have that  
$$\Gamma_h\ot_R \T_{h^\m} = f(h) \otimes _R \T _h  \ot_R \T_{h^\m} 
\simeq f(h) \ot _R D_h \simeq f(h),$$ keeping in mind the unital condition \eqref{isot} on $\T .$ This implies that  $\Gamma_h\ot_R \T_{h^\m}$ is a central $R$-bimodule, and the map $ T_{\T_{g^\m}\ot\Omega_g,\Gamma_h\ot \T_{h^\m}}$ takes the simple form given in Proposition~\ref{isomorfismoT}. Then by \eqref{fsl} we have that the factor set for the unital partial  representation $\Lambda: G\ni g\mapsto [\G_g\ot_R(D_{g^\m})_g\ot_R\Om_g]\in \Pics(R),$ is given by the family
\begin{equation}\label{fla}f_{g,h}^\Lambda(d_g\ot_R d'_{g^{\m}}\ot_R u_g\ot_R e_h\ot_R e'_{h^{\m}}\ot_R v_h)= f_{g,h}^\G(d_g\ot_Re_h)\ot_Re'_{h^{\m}}\al_{h^{\m}}(d'_{g^{\m}}1_h)\ot_Rf_{g,h}^\Om(u_g\ot_Rv_h),
\end{equation}
for every $d_g\ot_R d'_{g^{\m}}\ot_R u_g\ot_R e_h\ot_R e'_{h^{\m}}\ot_R v_h\in  \Gamma_g\ot_R \T_{g^\m}\ot_R \Omega_g\ot_R \Gamma_h\ot_R \T_{h^\m}\ot_R \Omega_h$ and $g,h\in G.$ Moreover, observe that 
$[(f \T )_g]   \in \Pics_{R^\alpha}(R),$ as $[f(g)] \in \Pics _R (R)$ and $[ \T _g] \in \Pics_{R^\alpha}(R).$

\subsection{The  Brauer group}\label{sec:Brauer}

We recall the notion of the  Brauer group $B(R)$ of a commutative ring $R,$ which was introduced  in   \cite{auslander1960brauer},   based on the concept  of a  separable $R$-algebra.  An   $R$-algebra $A$ is said to be {\it separable} over $R$ if $A$ is a projective module over its enveloping algebra 
$A^e = A\otimes _R A^{\op},$ where $ A^{\op}$ stands for the opposite algebra of $A.$ Moreover, $A$ is called an {\it Azumaya} $R$-algebra if $A$ is  separable  over $R$ and  $\Z(A)=R1_A.$ Two Azumaya $R$-algebras $A$ and $B$ are said to be equivalent  if there exist faithful f.g.p. $R$-modules $P$ and $Q$ such that
 the $R$-algebras $ A \otimes _R  {\rm End}_R(P) $ and $  B \otimes   _R   {\rm End}_R\,(P) $ are isomorphic.  We denote by $\llbracket  A \rrbracket $ the equivalence class of $A.$
Then the Brauer group $\B(R)$ of $R$ consists of the equivalence classes $\llbracket  A \rrbracket $ of Azumaya $R$-algebras $A,$ with the group operation induced  by $\otimes _R.$   Given a commutative $R$-algebra $S$ the map   
$\B(R) \ni   \llbracket A\rrbracket  \mapsto  \llbracket A \otimes_R S\rrbracket \in \B(S),$ is a  group homomorphism, whose kernel, denoted by  $B(S/R)$ is  called the {\it relative Brauer group}.

{\it In all what follows we suppose that $G$ is a finite group, and $R\supseteq R^\alpha$ is a partial Galois extension}.

\begin{obs}\label{azuc} Let $f\in Z^2(G,\alpha, R).$ Then it follows by Remark~\ref{coco} and Proposition~\ref{coiso} that there exists a normalized cocycle $f'\in Z^2(G,\alpha, R)$ cohomologous to $f,$ and there is a  $R^\alpha$-algebra isomorphism $ R\star_{\al,f}G\simeq  R\star_{\al,f'}G.$ Since $R\supseteq R^\alpha$ is a partial Galois extension, it follows by \cite[Theorem 2.4]{PS} that $\alpha$ is $\omega$-outer in $R$ (see \cite[p. 1096]{PS}) which in turn by \cite[Proposition 3.3]{PS} and \cite[(vi) Lemma 2.1]{PS} implies that  $R\star_{\al,f'}G$ is Azumaya $R^\alpha$-algebra. Therefore  $R\star_{\al,f}G$ is  also Azumaya $R^\alpha$-algebra. 
\end{obs}

 The following facts about the  relative Brauer group $\B(R/R^\al)$ will be of interest for us.
\begin{obs}\label{brr}
Take $\llbracket A\rrbracket \in \B(R/R^\al).$ Since $R$ is a f.g.p. $R^\alpha$-module, it follows   by  \cite[Theorem 5.7]{auslander1960brauer} that  we may assume that $R$ is a maximal commutative $R^\alpha$-subalgebra of  $A.$  Moreover, if $\llbracket A\rrbracket \in \B(R^\al)$ and $A$ contains $R$ as a maximal commutative $R^\alpha$-subalgebra. Since by \cite[Theorem 4.2]{dokuchaev2007partial}  that $R\subseteq R^\alpha$ is separable, then by   \cite[Theorem 5.6]{auslander1960brauer}  we get that 
 $\llbracket A\rrbracket  \in  \B(R/ R^\alpha ).$ 
\end{obs}

 Consider the $(R,A)$-bimodule
$_g(1_{g^\m}A)$ whose underlying set is $1_{g^\m}A$ with the actions defined by
\begin{equation}\label{aopb}r \cdot a = \alpha_{g^\m} (r1_{g}) a \;\;\; \mbox{and} \;\;\;  a \cdot a' = a a',  \end{equation}
where $r\in R, a \in 1_{g^\m}A,$ and $a' \in A.$ Consider also the following $(R,A)$-subbimodule of  $_g(1_{g^\m}A):$
\begin{equation}\label{defC}
C_g=\{a \in 1_{g^\m}A; \ \al_{g^\m}(r1_g)a=ar, \ \forall \ r \in R\}.
\end{equation}
notice that, $C_g$ is a central $R$-bimodule. Suppose that $\llbracket A\rrbracket\in \B(R/R^\al).$ By  \cite[Theorem 5.9]{DoPaPi} there is a group homomorphism  
\begin{equation}\label{fi5}\phi_5:B(R/R^\al)\ni \llbracket A\rrbracket\mapsto {\rm cls}(f_A)\in H^1(G,\al^*,\Pics_R(R)),\end{equation} where by
  \cite[Remark 5.3]{DoPaPi2} the map
 $f_A\in Z^1(G,\al^*,\Pics_R(R))$  can be chosen such that  $f_A(g)=[C_g],$ for all $g\in G.$  In particular, $[C_g]\in \Pic_R(D_g)$
and thus belongs  to $\Pics_R(R).$ Let $\Gm: G\ni g\mapsto f_A(g)\T(g)\in \Pics_{R^\alpha}(R).$ Write $\Gm(g)=[\Gm_g], $ $g\in G.$

Then  by \cite[Lemma 6.3]{DoPaPi} we have that  $\Gm$ is a unital partial representation with  $\Gm_g\ot _R \Gm_{g^\m}\simeq D_g.$  Moreover, since  $C_g|R ,$ the compatibility of $|$ with   tensor products implies  $\Gm_g|(D_g)_{g^\m}$,  for each  $g \in G.$  On the other hand, by the proof 
of \cite[Proposition 6.5]{DoPaPi2} we may assume without loss of generality that   $\Gamma _g  = C_g $ as sets, for any $g\in G,$ and the $R$-actions on $\Gamma _g$ are given by 
\begin{equation}\label{star}r \star a = \alpha_{g^\m} (r1_{g}) a =ar\;\;\; \mbox{and} \;\;\;  
a \star r = a \alpha_{g} (r1_{g^\m})=ra,  \end{equation}
where $a \in \Gamma _g=C_g, r \in R.$

In addition, 	
 $$
f_{g,h}: \Gm_g\ot _R \Gm_h\ni u_g\ot_R u_h  \mapsto  1_g\ot_R u_hu_g \in  D_g\ot _R \Gm_{gh} $$
is a factor set for  $\Gm  $ and  $\D(\Gm)=\bigoplus_{g\in G}\Gm_g$ is a partial generalized crossed product.
 Since $\Gamma = f_A \Theta $ and $f_A\in Z^1(G,\al^*,\Pics_R(R)),$ we have  by Theorem~\ref{prop:grupoC} 
that 
$[\D(\Gm)]\in C(\T/R)$. Furthermore,  recall from \cite[p. 218]{DoPaPi2} that  the multiplication in $\Delta (\Gamma )$ is given by 
\begin{equation} \label{prodg}u_g\stackrel{\Gm}{\circ} u_h    =  1_g \star (u_hu_g)=1_{h\m}1_{h\m g\m}u_hu_g=u_hu_g.\end{equation}
where $u_g \in \Gamma _g, u_h \in \Gamma _h.$

Given a ring  $A$   and a  nonempty subset $B$ of $A$, we denote by
 $$C_A(B)=\{a \in A: ax=xa \ \mbox{for all } x\in B\}$$ the {\it centralizer} of $B$ in $A$.  If $B$ is a commutative $R$-subalgebra of an $R$-algebra $A,$  it is well known that $B$ is a maximal commutative subalgebra if and only if $C_A(B)=B.$

\begin{teo}\label{Aisopcgp} Let $A$ be an Azumaya  $R^\al$-algebra, containing  $R$ as a maximal  commutative  subalgebra and  $\phi_5(\llbracket A\rrbracket )=[f_A],$  defined in \eqref{fi5}.
Then with the  above notation   we have that  $\D(\Gm)\simeq A^{\op}$ as $R^\al$-algebras.
\end{teo}
\begin{dem} As above, we have that $\Gamma_g=C_g$ as sets. Then $\Gamma_g\subseteq A,$ for all $g\in G ,$ and we may consider the $R^\al$-map
\begin{equation}\label{fiazu} \varphi: \bigoplus_{g \in G}\Gm_g\ni u_g  \mapsto u_g \in A^{\op}, \end{equation}
	induced  by the  inclusion of  $\Gm_g$ in  $A$.
	Let  $u_g \in \Gm_g$ e $u_h \in \Gm_h$, then $\varphi(u_g\stackrel{\Gm}{\circ} u_h)  \stackrel{\eqref{prodg}}= u_h u_g,$
	which shows that  $\varphi$  is a homomorphism of $R^{\alpha}$-algebras.
	Localizing at  $\mathfrak p \in {\bf Spec}(R^\al)$ one may assume that  $R^\al$ is local. Then   $R$ is semilocal and $\Pic(R)$ is trivial. The fact that $D_g$ is a direct summand of  $R$, implies $\Pic_R(D_g)=\{1\}.$ Moreover, since   $f \in Z^1(G,\al^*, \Pics_R(R))$, we have  $f_A(g)=[C_g]\in \Pic_R(D_g)=\{1\}$ and $C_g\simeq D_g$, as $D_g$-modules. Let  $\gamma_g:D_g\to C_g$ be   a $D_g$-module  isomorphism 
	and denote $\gamma_g(1_g)=w_g \in C_g$. Given  $c_g \in C_g$, there exists $d \in D_g$ with
	\begin{equation*}
	c_g=\gamma_g(d)=d\star \gamma_g(1_g)=\gamma_g(1_g)d.
	\end{equation*}
	Thus 
	\begin{equation}\label{eq:locC_g}
	C_g=D_g\star w_g=w_gD_g.
	\end{equation}
	
	Furthermore, by \cite[p. 223]{DoPaPi2}

	there exists $a_g \in A$ such that
	\begin{equation}\label{agwg}
	w_ga_g=1_{g^\m} \;\;\;\; \mbox{and} \;\;\;\; a_gw_g=1_g.
	\end{equation}
	Given  $r \in R$ since $w_g\in C_g,$ we have that
	$
	w_g\al_g(r1_{g^\m})a_g  =  rw_ga_g=r1_{g^\m}
	 =  1_{g^\m}r = w_g(a_gr).
$
	Then the fact that  $w_g$ is a free basis for  $C_g$, implies
	\begin{equation}\label{agr}
	\al_g(r1_{g^\m})a_g=a_gr, \ \forall \ r \in R.
	\end{equation}
	Furthermore, observe that:
	\begin{eqnarray}
	a_ga_h1_l=1_g1_{gh}1_{ghl}a_ga_h. \label{agah1l}
	\end{eqnarray}
	Consider $\tau_{g,h}=a_ga_hw_{gh},$ we shall prove that $\tau_{g,h}\in R$ . For $r\in R$,  we have
	\begin{eqnarray*}
	\tau_{g,h}r & = & a_ga_hw_{gh}r\\
	&\stackrel{\eqref{defC}} = & a_ga_h\al_{h^\m g^\m}(r1_{gh})w_{gh}\\
	& \stackrel{\eqref{agr}}= & a_g\al_h(  1 _{h^\m } \al_{h^\m g^\m}(r1_{gh}))a_hw_{gh}\\ 
	& = & a_g\al_{g^\m}(r1_g)1_ha_hw_{gh}\\
	& \stackrel{\eqref{agr}}= & r1_ga_g1_ha_hw_{gh}\\
	& \stackrel{\eqref{agr}}= & ra_g1_{g^\m}1_ha_hw_{gh}\\
	& = & ra_g\al_{h}(1_{h^\m g^\m}1_{h^\m})a_hw_{gh}\\
	& \stackrel{\eqref{agr}}= & ra_ga_h1_{h^\m g^\m}w_{gh}\\
	& = & ra_ga_hw_{gh}.
	\end{eqnarray*}

	Therefore, $\tau_{g,h} \in C_A(R)=R$. Moreover,
	\begin{eqnarray*}
	1_g1_{gh}a_ga_hw_{gh}& = & \al_g(1_h1_{g^\m})a_ga_hw_{gh}\stackrel{\eqref{agr}}=a_g1_h1_{g^\m}a_hw_{gh}
	 =  a_g\al_h(1_{h^\m g^\m}1_{h^\m})a_hw_{gh}\\
	& \stackrel{\eqref{agr}}= & a_ga_h1_{h^\m g^\m}w_{gh}\stackrel{\eqref{defC}}=a_ga_hw_{gh}.
	\end{eqnarray*}
Hence, $\tau_{g,h}=a_ga_hw_{gh} \in D_gD_{gh}$. Now we check that $\tau_{g,h}$ has  an inverse in $ D_gD_{gh}$. Observe that
\begin{eqnarray*}
\tau_{g,h}(a_{gh}w_hw_g) & = & a_ga_hw_{gh}a_{gh}w_hw_g\\
& \stackrel{\eqref{agwg}}= & a_ga_h1_{h^\m g^\m}w_hw_g\\
& = & a_ga_hw_h1_{g^\m}w_g\\
& \stackrel{\eqref{agwg}}= & a_g1_h1_{g^\m}w_g\\
& = & a_gw_g1_{gh}\\
& = & 1_g1_{gh}.
\end{eqnarray*}
Similarly, $(a_{gh}w_hw_g)\tau_{g,h} = 1_g1_{gh}.$ Thus, $\tau_{g,h}^\m=a_{gh}w_hw_g$ and $\tau_{g,h}\in \mathcal{U}(D_gD_{gh}),$ and we have a map  $\tau_{-,-}\in  C^2(G,\al,R)$ defined by  $\tau_{-,-}: G\times G\ni (g,h)\mapsto \tau_{g,h}\in R.$ We shall prove that $\tau_{-,-}\in  Z^2(G,\al,R).$

Take $g,h,l \in G$. We see that
\begin{eqnarray*}
\tau_{g,h}\tau_{gh,l} & = & a_ga_hw_{gh}a_{gh}a_lw_{ghl}=a_ga_hw_{gh}1_{gh}1_{ghl}a_{gh}a_lw_{ghl}
 =  a_ga_h1_lw_{gh}a_{gh}a_lw_{ghl}\\&\stackrel{\eqref{agah1l}} = & 1_g1_{gh}1_{ghl}a_ga_hw_{gh}a_{gh}a_lw_{ghl}
 =  1_g1_{gh}1_{ghl}a_ga_h1_{h^\m g^\m}a_lw_{ghl}= 1_g1_{gh}1_{ghl}a_ga_ha_lw_{ghl}.
\end{eqnarray*}
On the other hand,
\begin{eqnarray*}
\al_g(\tau_{h,l}1_{g^\m})\tau_{g,hl} & = & \al_g(\tau_{h,l}1_{g^\m})a_ga_{hl}w_{ghl}= a_g\tau_{h,l}a_{hl}w_{ghl}
 =  a_ga_ha_lw_{hl}a_{hl}w_{ghl}\\
& = & a_ga_ha_l1_{l^\m h^\m}w_{ghl}
 \stackrel{\eqref{agah1l}}=  a_g1_h1_{hl}a_ha_lw_{ghl}=1_g1_{gh}1_{ghl}a_ga_{h}a_{l}w_{ghl}.
\end{eqnarray*}
Then, $\al_g(\tau_{h,l}1_{g^\m})\tau_{g,hl}=\tau_{g,h}\tau_{gh,l}$ and $\tau_{-,-} \in Z^2(G,\al,R),$ keeping in mind \eqref{z2}.
Moreover,
\begin{equation}\label{eq:ProductInB}
\tau_{g,h}a_{gh}=a_ga_hw_{gh}a_{gh}=a_ga_h1_{h^\m g^\m} \stackrel{\eqref{agah1l}}=1_g1_{gh}a_ga_h.
\end{equation}

The map $$F: R\star_{\al,\tau}G\ni  d_s\delta_g\mapsto d_ga_g\in A$$

is a homomorphism  of algebras, as
\begin{eqnarray*}
F((d_g\delta_g)(d_h\delta_h)) & = & F(d_g\al_g(d_h1_{g^\m})\tau_{g,h}\delta_{gh})\\
& = & d_g\al_g(d_h1_{g^\m})\tau_{g,h}a_{gh}\\
& = & d_g\al_g(d_h1_{g^\m})1_g1_{gh}a_ga_h,
\end{eqnarray*}
and
\begin{equation*}
F(d_g\delta_g)F(d_h\delta_h)=(d_ga_g)(d_ha_h)=d_g\al_g(d_h1_{g^\m})a_ga_h. 
\end{equation*}
Let $\mathcal{B}=\dsum_{g \in G}\mathcal{B}_g\subseteq A$, where $\mathcal{B}_g=D_ga_g$, with $g \in G$. Then, 
$$ (d_ga_g)(d_ha_h) = d_g\al_g(d_h1_{g^\m})a_ga_h
 =  d_g\al_g(d_h1_{g^\m})1_g1_{gh}a_ga_h = d_g\al_g(d_h1_{g^\m})\tau_{g,h}a_{gh} \in {\mathcal B}_{gh},$$ thanks to  \eqref{eq:ProductInB}.
Hence, the image of $F$ is $\mathcal B$ and $F$ can be considered as an epimorphism of $R^{\alpha}$-algebras 
$ R\star_{\al,\tau}G  \to \mathcal{B}.$  Furthermore, it  follows from  Remark \ref{azuc}  that $ R\star_{\al,\tau}G$ is Azumaya $R^\alpha$-algebra.   Then by \cite[Corollary 2.3.7]{demeyer1971separable}
we have that $\ker F = (\ker F \cap R^{\alpha }) R\star_{\al,\tau}G.$ In view of \eqref{agwg}, $a_1$ is an invertible element of $A$ and therefore $F$ is injective on $R^{\alpha},$ which implies that $\ker F =0.$ By the argument given at the end of the proof of  \cite[Lemma 5.5]{DoPaPi2} we obtain that  $A=B,$ so that   $A \cong R\star_{\al,\tau}G$.   

Note that $\tau _{1,1} = a_1  a_1 w_1 = a_1 \in R$ thanks to \eqref{agwg}. Moreover, observe that
\begin{equation}\label{wg}
\tau_{g^\m, g}^\m a_{g^\m} = a_1w_gw_{g^\m}a_{g^\m}= a_1 w_g1_g=  a_1 w_g.
\end{equation}
When localizing, we have by \eqref{eq:locC_g} that  $C_g=D_g\star w_g=D_{g\m} w_g,$ and it  follows by \eqref{wg} that  the localization of $\varphi$  defined in \eqref{fiazu} is given  by

$$\varphi_{loc}: \bigoplus_{g \in G}D_{g\m} w_g\ni d_{g\m} w_g  \mapsto  d_{g\m} w_g = d_{g^\m}  a^{\m}_1 \tau_{g^\m,g}^\m a_{g^\m}\in A^{\op}.
$$
We shall check that
$
\varphi':  A^{\op} \ni d_ga_g \to d_g \tau_{g,g^\m} a_1 w_{g^\m} \in  \bigoplus_{g \in G}D_{g^\m} w_g
$
is the inverse of   $\varphi_{loc}.$  Indeed, since $d_{g^\m} a^{\m}_1 {\tau_{g^\m,g}^\m} \in D_{g^\m} $ and $R$ is commutative,  we have that
\begin{eqnarray*}
(\varphi'\circ\varphi_{loc} )(d_{g^\m}w_g)  =
   \varphi'((d_{g^\m} a^{\m}_ 1{\tau_{g^\m,g}^\m}) a_{g^\m})
 =  d_{g^\m} a^{\m}_1 \tau_{g^\m,g}^\m \tau_{g^\m,g} a_1 w_g
 =  d_{g^\m}1_{g^\m}w_g=d_{g^\m}w_g.
\end{eqnarray*}
On the other hand,
\begin{eqnarray*}
	(\varphi_{loc}  \circ \varphi')(d_{g}a_g) & = 
	  \varphi_{loc}(d_{g} {\tau_{g,g^\m}} a_ 1 w_{g^\m}) 
	 =  d_{g} a^{\m}_1 \tau_{g,g^\m}\tau_{g,g^\m}^\m a_1 a_g
	 =  d_{g}1_{g}a_g=d_{g}a_g,
\end{eqnarray*}
and we conclude that $\D(\Gm)\simeq A^{\op}$ as $R^\al$-algebras.
\end{dem}

The following result is a consequence of  Theorem~\ref{prop:grupoC} and \cite[Proposition 6.3]{DoPaPi2}.
\begin{pro}\label{CdentrodeB} Let $\Gm: G\to\Pics(R)$  be a unital partial representation endowed with a factor set such that $[\D(\Gm)]\in \C(\T/R).$ Then $\D(\Gm)$ is an Azumaya $R^\al$-algebra containing $R$  as a maximal  commutative subalgebra. 
\end{pro}

It follows from Proposition \ref{CdentrodeB} and Remark \ref{brr} that $\llbracket \Delta(\Gamma) \rrbracket\in \B(R/R^\alpha)$. Thus  $\varphi_5(\llbracket \Delta(\Gamma) \rrbracket)=[f_{\D(\Gm)}],$  where  $f_{\D(\Gm)}(g)=[C_g],$ and $C_g$ is defined in \eqref{defC} (taking $A=\D(\Gm)$),  and the bimodule structure of $C_g$ is given by  \eqref{aopb}, for all $g\in G.$  Furthermore, by \cite[Proposition 6.1]{DoPaPi2}, we have that $\Gamma_1\simeq R$ as $R$-algebras.  

\noindent We proceed with the next.

\begin{lem}\label{CgeJgm}
	Let $[\D(\Gm)] \in \C(\T/R)$ and $g\in G.$  Then there is a  $R$-bimodule isomorphism  $\Gm_g\simeq (D_g)_{g^\m}\ot _R C_{g^\m}$, where   $R$-bimodule structure on $\Gamma_g$ is given via multiplication in $\D(\Gm)$.
\end{lem}
\begin{dem} First of all, it follows from  the equality (31) in \cite{DoPaPi2} that $\Gamma_g=C_{g^\m},$ as sets, for any $g\in G.$  We shall show that   
$\lambda:  \Gm_g \ni u_g \mapsto  1_g\ot_R u_g\in  (D_g)_{g^\m} \ot _R C_{g^\m}$ is a  $R$-bimodule isomorphism. It is clear that  $\lambda$ is right $R$-linear. To show that it is left $R$-linear, take $r \in R$ and $u_g\in \Gm_g ,$ then 
	\begin{eqnarray*}
		\lambda(ru_g) & = & 1_g\ot_R ru_g = 
		1_g\ot_R \al_{g^\m}(r1_g)\cdot u_g
		 =  1_g*\al_{g^\m}(r1_g)\ot_R u_g= r1_g\ot_R u_g
		 =  r\lambda(u_g).
	\end{eqnarray*} In order to construct the inverse of $\lambda$  let
	\begin{equation}\label{invi}\lambda':  (D_g)_{g^\m}\ot_R C_{g^\m} \ni d_g \ot_R c_{g^\m}  \mapsto d_gc_{g^\m}\in \Gamma_g. \end{equation}
	Let us see that  $\lambda'$ is well-defined. Take $r \in R $, then 
	\begin{equation*}
	\lambda'(d_g* r,c_{g^\m}) = \lambda'(d_g\al_g(r1_{g\m}),c_{g^\m})= d_g\al_g(r1_{g^\m})c_{g\m}.
	\end{equation*}
	On the other hand,
	\begin{equation*}
	\lambda'(d_g,r \cdot c_{g^\m})= \lambda'(d_g,\al_g(r1_{g^\m})c_g)=d_g\al_g(r1_{g^\m})c_{g^\m},
	\end{equation*}
 and   $\lambda'$ is well-defined.  Finally,
	\begin{eqnarray*}
		(\lambda \circ \lambda')(d_g\ot_R c_{g\m}) & = & \lambda(d_gc_{g^\m}) = 1_g\ot_R  d_gc_{g^\m}=1_g \ot_R \al_{g^\m} (d_g) \cdot c_{g^\m}  \\
		& = &  
		1_g \cdot \al_{g^\m} (d_g)  \ot_R c_{g^\m} 
		 =  d_g\ot_R c_{g^\m},
	\end{eqnarray*}
	and
$
	(\lambda'\circ \lambda)(u_g)=\lambda'(1_g\ot_R u_g) = 1_gu_g=u_g.
$
	Hence  $\lambda^\m=\lambda',$ which finishes the proof.
	\end{dem}

\begin{pro}\label{proparaL} Let $[\D(\Gm)]\in \C(\T/R)$ and $P$ be an $R^\al$-module. Suppose that $P$ is a progenerator such that $\End_{R^\al}(P)\simeq \D(\Gm)$, as $R^\al$-algebras. Then,
	\begin{itemize}
		\item[$(i)$] There is a $R$-bimodule structure on $P$ such that $[P]\in \Pic_R(R)$;
		\item[$(ii)$] $\Gamma_g\simeq P\ot_R(D_g)_{g^\m}\ot_RP^*$, for all  $g \in G$. 
	\end{itemize}
\end{pro}

\begin{dem} 
	$(i)$ Let  $\varphi:\D(\Gm)\to \End_{R^\al}(P)$  be a $ R^\al$-algebra  isomorphism. We endow  $P$ with a  $R$-module structure via
$r\cdot p=\varphi(r)(p),$ for all $r\in R$ and $p\in P.$
Since by Proposition \ref{CdentrodeB}  $R$  is  a maximal  commutative subalgebra of  $\D(\Gm),$ we have:
\begin{equation}
R=C_{\D(\Gm)}(R)\simeq C_{\End_{R^\al}(P)}(R)=\End_R(P). \label{RisoHom}
\end{equation}
In particular, $P$ is a  faithful $R$-module. Since the ring extension $R^\al\subseteq R$ is  separable and by Proposition~\ref{CdentrodeB} the $R^{\alpha}$-algebra  $\End_{R^{\alpha}}(P)$ is  Azumaya,     it follows by  
\cite[Theorem 5.6]{auslander1960brauer} that  $\End_{R^\al}(P)$ is a f.g.p. 
$R$-module.  Thus, the fact that $P$ is an  $R^\alpha$-progenerator implies by   \cite[Theorem 2]{azumaya1966completely}  that $P$ is a  f.g.p.  $\End_{R^{\alpha}}(P)$-module. This implies that $P$ is a f.g.p. $R$-module, thanks to \cite[Proposition 1.1.4]{demeyer1971separable}. Hence, the equality (\ref{RisoHom}) implies ${\bf rk}_R(P)=1,$ that is $[P]\in \Pic_R(R)$ (see, for example,  \cite[Lemma 1.5.1]{demeyer1971separable}).

\noindent $(ii)$ Let $A=\End_{R^\al}(P)$ we know that $\llbracket A\rrbracket \in \B(R/R^\al),$ then by 
 formula (31) from  \cite{DoPaPi2}  we have
\begin{equation}\label{send}
\Gm_g=\{\xi \in A: \al_g(r1_{g^{-1}})\xi=\xi r, \ \forall \ r \in  R\}.
\end{equation}
Note that the map
$
\lambda:  \Gm_g\ot _R P\ni \xi \ot_R p  \mapsto \xi(p)\ot _R 1_g \in P\ot _R (D_g)_{g^\m},
$
is well-defined.  Indeed, for  $r \in R, p \in P$ and $\xi \in \Gm_g$,
\begin{equation*}
\lambda(\xi\cdot  r, p)=(\xi\cdot  r)(p)\ot _R 1_g\stackrel{\eqref{estrelaMebimodulo}}=\xi(rp)\ot _R 1_g=\lambda(\xi, rp).
\end{equation*}
It is easy to see that $\lambda$ is left $R$-linear.  Moreover,
\begin{eqnarray*}
\lambda(\xi \ot _R pr)  &=& \xi(rp)\ot _R 1_g\stackrel{\eqref{send}}= \al_g(r1_{g^\m})\xi(p)\ot _R 1_g
 =  \xi(p)\ot _R \al_g(r1_{g^\m}) \\
&=& (\xi(p)\ot _R 1_g)* r
 =  \lambda(\xi \ot _R p)* r,
\end{eqnarray*} and $\lambda$ is a morphism of $R$-bimodules. We prove that $\lambda$ is an isomorphism. Localizing at an ideal in  ${\bf Spec}(R^\al)$ we may assume, as in the proof 
of Theorem~\ref{Aisopcgp},  that 
 $C_g\simeq D_g,$ as $R$-modules.  Let 
$\gamma_g:D_g\to C_g$ be an $R$-isomorphism and denote $\gamma_g(1_g)=f_g \in C_g .$  Then  $C_g=D_g\cdot f_g=f_gD_g$ and  by \eqref{defC} $\al_{g^\m}(r1_g)f_g(p)=f_g(rp)$, for every  $r \in R$ and  $p \in P$.
By Lemma \ref{CgeJgm},  one has an isomorphism $\widetilde{\gamma}_g:(D_g)_{g^\m}\to \Gm_{g}$ defined  via the chain of  isomorphisms 
$$
(D_g)_{g^\m} \simeq (D_g)_{g^\m}\ot _R D_{g^\m}\simeq (D_g)_{g^\m}\ot _R C_{g^\m}\simeq \Gm_g.
$$
The image of $d_g \in (D_g)_{g^\m} $  by $\widetilde{\gamma}_g$ is 
$$ d_g \mapsto d_g \ot_R 1_{g^\m} \mapsto d_g \ot_R  \gamma_{g^\m}(1_{g^\m}) 
\stackrel{\eqref{invi}}\mapsto d_g \gamma_{g^\m}(1_{g^\m}) = d_g f_{g^\m} ,$$
so that $\widetilde{\gamma}_g(1_g)= 1_gf_{g^\m} = f_{g^\m}$ and  $\Gm_g=D_g f_{g^\m} =  f_{g^\m}  D_{g^\m},$ where the last equality follows from \eqref{send}.   Since $A$ is an  Azumaya $R^\alpha$-algebra containing $R$ as a maximal commutative subalgebra, one obtains by (34) and (35) in \cite{DoPaPi2}  that there is  $l_g \in A$  such that
\begin{equation*}
f_{g} \circ l_g=1_{g^\m} \quad \mbox{and} \quad l_g\circ f_{g}=1_g.
\end{equation*}
Then by formula (36) from \cite{DoPaPi2} we have that
$
\al_g(r1_{g^\m})l_g(p)=  l_g(r  p), \ \mbox{for all } r \in R, p \in P.
$

\noindent We shall show that 
$
\lambda:  f_{g^\m}D_{g^{-1}}\ot _R P \ni f_{g^\m}d_{g^\m} \ot _R p  \mapsto( f_{g^\m}d_{g^\m} ) (p)\ot _R 1_g\in  P\ot _R (D_g)_{g^\m},$ is an isomorphism with inverse
$$
\lambda' : P\ot_R (D_g)_{g^\m}\ni p \ot_R d_g  \mapsto f_{g^\m}\alpha_{g^\m}(d_g)\ot_R  l_{g^\m}(p) \in f_{g^\m}D_{g^{-1}}\ot _R P.$$
Firstly,  to check that  $
\lambda'$ is well-defined, take $p \in P, d_g \in D_g$ and $r \in R$,  we have
\begin{eqnarray*}
\lambda'(pr,d_g) & = &f_{g^\m}\alpha_{g^\m}(d_g)\ot_R l_{g^\m}(r  p)\\
& = &f_{g^\m}\alpha_{g^\m}(d_g)\ot _R\al_{g^\m}(r1_g)l_{g^\m}(p)\\
& = & f_{g^\m}\alpha_{g^\m}(d_g)\al_{g^\m}(r1_g)\ot_R l_{g^\m}(p)\\
& = &   f_{g^\m}\alpha_{g^\m}(rd_g)\ot_R l_{g^\m}(p)\\
& = & \lambda'(p,rd_g),
\end{eqnarray*} as desired.  Now, given  $p \in P$ and $d_{g^\m}\in D_{g^\m}$
\begin{eqnarray*}
(\lambda'\circ \lambda)(f_{g^{-1}}d_{g^\m}\ot_R p) & = & \lambda'((f_{g^{-1}}d_{g^\m})(p)\ot _R1_g)\\
& = &f_{g^\m}1_{g\m} \ot_R l_{g^\m}((f_{g^\m}d_{g^{-1}})(p))\\
& = & f_{g^\m}\ot_R d_{g^{-1}} \cdot p\\
& = & f_{g^\m}d_{g^{-1}}\ot_R  p.
\end{eqnarray*}
Furthermore,
\begin{eqnarray*}
(\lambda\circ \lambda')(p\ot _R d_g) & = & \lambda(f_{g^\m}\alpha_{g^\m}(d_g)\ot_R l_{g^\m}(p))\\
& = & f_{g^\m}(\alpha_{g^\m}(d_g)l_{g^\m}(p))\ot_R 1_g\\
& = & d_g \cdot p\ot_R 1_g\\
& = &p\ot_R d_g.
\end{eqnarray*}
Thus $\lambda ^\m=\lambda'$.
Consequently,  there exists an $R$-bimodule isomorphism $\varphi_g: \Gm_g\to P\ot_R (D_g)_{g^\m}\ot_R P^*$, defined by the chain of  isomorphisms
$$\Gm_g\simeq \Gm_g\ot_R R\simeq \Gm_g\ot P\ot_R P^{*}\simeq P\ot_R (D_g)_{g^\m}\ot P^{*},$$
which finishes the proof.\end{dem}

Observe that using the notation in the proof of  Proposition \ref{proparaL}, we have  by  the  definition of  $\lambda$  
that the isomorphism $\varphi_g$ is given by:
\begin{equation}
\varphi_g:\Gm_g  \ni \xi_g\mapsto \dsum_{i=1}^n\xi_g(p_i)\ot_R 1_g\ot f_i\in P\ot_R (D_g)_{g^\m}\ot_R P^*,\label{defidevarphig}
\end{equation} 
where $\{p_i,f_i\}_{1\leq i\leq n}$ is a dual basis of $P$ as an $R$-module.

\noindent  Let  $[P] \in \Pic_R(R),$ then $[P]$ is an invertible element in  $\Pics_{R^\alpha}(R),$ it follows that 
the map  $$\T^P: G\to \Pics_{R^\alpha}(R),$$ defined by   $\T^P_g\simeq P\ot_R (D_g)_{g^\m}\ot_R P^*$,  is  a partial representation, which is easily seen to be unital with $\T_g^P\ot_R \T_{g^\m}^P\simeq D_g$ and  $\T_g^P|(D_g)_{g^\m}$, for all  $g \in G$. Moreover, let
	\begin{equation}\label{omp}
		\D(\T^P)=\bigoplus_{g \in G}P\ot_R (D_g)_{g^\m}\ot_R P^*,
	\end{equation} and consider the $R$-bimodule isomorphisms: 
	$$\{f_{g,h}^P:  P\ot_R (D_g)_{g^\m}\ot_R P^*\ot_R P\ot_R (D_h)_{h^\m}\ot_R P^*  \to 1_gP\ot_R (D_{gh})_{(gh)^\m}\ot _RP^*\}_{g,h\in G}$$
defined by the composition:
\begin{align*}
p_1 \ot_R d_g \ot_R f_1\ot_R p_2\ot _R d'_h \ot_R f_2 &\mapsto p_1 \ot_R d_g \ot_R f_1( p_2)\ot _R d_h \ot_R f_2\\
&\mapsto p_1 \ot_R d_g \alpha_g( f_1( p_2)1_{g^\m})\ot _R d'_h \ot_R f_2\\
&\stackrel{\eqref{fte}}\mapsto  p_1 \ot_R d_g \alpha_g( f_1( p_2)d'_h1_{g^\m}) \ot_R f_2\\
\end{align*}
that is, for $g,h\in G$
	\begin{equation}
		\begin{array}{c c c l}
			f_{g,h}^P: & P\ot_R (D_g)_{g^\m}\ot_R P^*\ot_R P\ot_R (D_h)_{h^\m}\ot_R P^* & \longrightarrow & 1_g P\ot_R (D_{gh})_{(gh)^\m}\ot _RP^*\\
			& p_1 \ot_R d_g \ot_R f_1\ot_R p_2\ot _R d_h \ot_R f_2 & \longmapsto & p_1 \ot_Rd_g  \alpha_g( f_1( p_2)d_h1_{g^\m}) \ot_R f_2
		\end{array} \label{conjuntodefatoresparaP}
\end{equation} %
Direct computations, similar to those made to prove that \eqref{fte} determines a factor set for $\T ,$ show that \eqref{conjuntodefatoresparaP} gives a factor set for $ \T^P ,$ so that $\D(\T^P)$ is a partial  generalized crossed product and $ [\D(\T^P)]\in \mathcal{C}(\T/R).$

   \cite[Lemma 5.6]{DoRo}  Let $[P]\in \Pic_\Z(R)^{(G)},$  and  denote $\G_g^{P}=P\ot_R\G_g\ot_R P^{*}$,  for all $g \in G$. Then  the map 
$ \G^P: G \ni g\mapsto [\G_g^P] \in  \Pics(R),$ is a unital partial representation

We proceed with the next. 

\begin{pro}\label{teoc}Let $\Gamma$ be as in  Proposition \ref{proparaL}.  Then the  following assertions hold.
\begin{enumerate}
\item   The map
$\mathcal{L}_0:  \Pic_R(R) \ni 		 [P]  \mapsto [\D(\T^P)] \in  \mathcal{C}(\T/R)$
is a group homomorphism.
\item If $P$ is as  in  Proposition \ref{proparaL}, then  $[\Delta (\Gamma)]=[\D(\T^P)] $ in $\C(\T/R)$.
\end{enumerate}
\end{pro}

\begin{dem} (1)   Note that the map $\mathcal{L}_0 $ is a particular case of the group morphism 
$\mathcal{L} : \Pic_\Z(R')^{(G)} \to  \mathcal{C}(\T/R'),$  given \cite[Theorem 5.10]{DoRo} for a  non-necesarilly commutative ring $R'.$ Indeed, 
the subgroup  $\Pic_\Z(R')^{(G)}$  of  $\Pic_\Z(R)$ is defined in \cite{DoRo} as follows:
	$$\Pic_\Z(R')^{(G)}=\{[P]\in \Pic_\Z(R'): P\ot_{R'}\Gamma'_g\ot_{R'} P^*|\Gamma'_g, \ \mbox{for all} \ g \in G\},$$
	where $\Z=\Z(R')$ and $\Gamma' : G \to \Pics (R)$ is a unital partial representation. Taking $R'=R$ and   $\Gamma' = \T,$ it is readily seen that   $\Pic_\Z(R)^{(G)} =  \Pic_R (R)$ and $\mathcal{L}_0 $ coincides with $\mathcal{L}. $

(2) By Proposition \ref{proparaL}, we know that  $\Gm_g\simeq P\ot_R (D_g)_{g^\m}\ot_R P^*= \T^P_g$,  for every   $g \in G$. By Remark \ref{sod} it remains to show that  the following diagram is commutative:
$$\xymatrix{\Gm_g\ot_R \Gm_h\ar[rr]^{f^\Gm_{g,h}}\ar[dd]_{\varphi_g \ot_R \varphi_h} & & 1_g \Gm_{gh}\ar[dd]^{ \varphi_{gh}}\\
& & \\
P\ot_R (D_g)_{g^\m}\ot_R P^*\ot_R P\ot_R (D_h)_{h^\m}\ot_R P^*\ar[rr]_{f_{g,h}^P} & & 1_g P\ot_R (D_{gh})_{(gh)^\m}\ot_R P^*,}$$
where, keeping in mind the $R^\alpha$-algebra isomorphism $\Delta(\G)\simeq \End_{R^\al}(P),$ the map   $f_{g,h}^\Gm$ is induced by the product in $\End_{R^\al}(P)$, that is, $f_{g,h}^\Gm(\xi_g\ot_R \xi_h)=1_g (\xi_g\circ \xi_h),$ the map $\varphi_g$  is given in  (\ref{defidevarphig}), and  $f_{g,h}^P$  is defined in (\ref{conjuntodefatoresparaP}). Let $\xi_g \in \Gm_g, \xi_h \in \Gm_h$ and    $\{p_i,f_i\}_{1\leq i\leq n}$  be a dual basis of the $R$-module  $P$. Then 
\begin{eqnarray*}
[f_{g,h}^P\circ (\varphi_g \ot_R \varphi_h)](\xi_g \ot \xi_h) & = & f_{g,h}^P \left( \dsum_{i,j}\xi_g(p_i)\ot_R 1_g\ot _Rf_i\ot_R \xi_h(p_j)\ot_R 1_h\ot_R f_j\right) \\
&\stackrel{\eqref{conjuntodefatoresparaP}} = & \dsum_{i,j} 1_g \xi_g(p_i)\ot _R\alpha_g(f_i(\xi_h(p_j))1_{g^\m})1_{gh}\ot_R f_j\\
&= & \dsum_{i,j} 1_g   \alpha_g(f_i(\xi_h(p_j))1_{g^\m})\xi_g(p_i) \ot _R 1_{gh}\ot_R f_j\\
& \stackrel{\eqref{send}}= &  \dsum_{i,j}  \xi_g(f_i(\xi_h(p_j))1_{g^\m}p_i) \ot _R 1_{gh}\ot_R f_j\\
& = &  \dsum_{i,j}  \xi_g(f_i(\xi_h(p_j)1_{g^\m})p_i) \ot _R 1_{gh}\ot_R f_j\\
& = &  \dsum_{j} \xi_g(\xi_h(p_j)1_{g^\m})\ot_R 1_{gh}\ot_R f_j\\
& = &  \dsum_{j} 1_{g}\xi_g(\xi_h(p_j))\ot_R 1_{gh}\ot_R f_j\\
& = &\varphi_{gh}(1_g ( \xi_g\circ \xi_h))\\
& = & (\varphi_{gh}\circ f_{g,h}^\Gm)(\xi_g \ot_R \xi_h),
\end{eqnarray*}
as desired. 
\end{dem}

Take $[P] \in \Pic_R(R).$  By  \cite[Proposition I.1.6]{demeyer1971separable} we have that  $P$ is an $R^\al$-progenerator. By, \cite[Proposition II. 4.1]{demeyer1971separable} we get that  $\End_{R^\al}(P)$ is an Azumaya $R^\al$-algebra. Moreover, 
$$C_{\End_{R^\al}(P)}(R)=\End_{R}(P)\simeq R.$$
Thus,  $R$ is a maximal commutative subalgebra of  $\End_{R^\alpha}(P)$. Consequently,   $(\End_{R^\alpha}(P))^{\op}$ is an Azumaya $R^\al$-algebra containing $R$ as a maximal commutative subalgebra.  Applying  Theorem \ref{Aisopcgp}  with $A= (\End_{R^\alpha}(P))^{\op}$ we conclude that there   exists a partial representation $\Gamma_P:G\to {\Pics}(R)$ endowed with a factor set, such that  $\End_{R^\al}(P)\simeq \D(\Gamma_P)$  as $R^\al$-algebras, and  $[\D(\Gamma_P)]\in \C(\T/R)$ thanks to Theorem~\ref{prop:grupoC}. Then Proposition \ref{teoc} gives us the group homomorphism
\begin{equation*} \label{Lcasogaloiscomutativo}
\mathcal{L}_0:  \Pic_R(R)\ni [P]  \mapsto [\D(\Gamma_P)]\in  \C(\T/R),
\end{equation*}
where $\D(\Gm_P)\simeq \End_{R^\al}(P)$ as $R^\al$-algebras.

The following result gives a relation between the groups $\C(\T/R)$ and  $\B(R/R^\al).$
\begin{teo}\label{exac}
	The sequence of group homomorphisms $$\xymatrix{\Pic_R(R)\ar[r]^{\mathcal{L}_0} &  \C(\T/R) \ar[r]^{\eta} &  \B(R/R^\al)\ar[r] &  1}$$ is exact,
where 
\begin{equation}\label{etta}\eta:\C(\T/R)\ni [\Delta(\Gamma)]\mapsto \llbracket \Delta(\Gamma) \rrbracket \in \B(R/R^\al).\end{equation} 
\end{teo} 
\begin{dem} Firstly we check that $\eta$ is  a surjective homomorphisms of groups. By Theorem \ref{CdentrodeB},  the map  $\eta$ is well defined.  Take  $[\Delta(\Gamma)], [\Delta(\Omega)]\in \C(\T/R).$ By Theorem~\ref{prop:grupoC}  there are  $ f^\Gamma, f^\Omega \in Z^1(G,\al^*,\Pics _R (R))$  such that $ \Gamma=f^\Gamma\T$ and  $ \Omega=f^\Omega\T.$  
  Since $\Delta(\Gamma)$ and $\Delta(\Omega)$  are partial generalized crossed products, it follows by \cite[Lemma 6.6]{DoPaPi2} that there are 
 $\llbracket A \rrbracket, \llbracket B \rrbracket \in \B(R/R^\al)$ such that $ \phi_5(\llbracket A \rrbracket)={\rm cls}(f^\Gamma)$ and  $ \phi_5(\llbracket B \rrbracket)={\rm cls}(f^\Omega),$ respectively, where $\phi_5$  is the group homomorphism defined in \eqref{fi5}.  Then by  the proof of \cite[Proposition 6.5]{DoPaPi2}  we may assume that each $(f^\Gamma \T)_g$  is a subset of $A$ and each $( f^\Omega\T)_g$ is a subset of $B,$ and 
  the families
$$\mathcal F^{f^\Gamma }=\{\mathfrak f^{\Gamma }_{g,h}:(f^\Gamma \T)_g\ot_R(f^\Gamma\T)_h\ni u_g\ot_R u_h\mapsto u_hu_g\in 1_g(f^\Gamma \T)_{gh}\}_{g,h\in G},$$
and
$$\mathcal F^{ f^\Omega}=\{\mathfrak f^{\Om}_{g,h}:( f^\Omega\T)_g\ot_R(f^\Omega\T)_h\ni u'_g\ot_R u'_h\mapsto  u'_hu'_g\in 1_g( f^\Omega\T)_{gh}\}_{g,h\in G},$$ are factor sets 
of $\D(f^\Gamma \T) $ and  $\D(f^\Omega\T),$ respectively, where the product $u_hu_g$ is given in $A$ and the product $ u'_hu'_g$ is given in $B.$
 Furthermore, since 
\begin{equation}\label{fimor}\phi_5(\llbracket A \ot_{R^\alpha} B \rrbracket)={\rm cls}(f^\Gamma f^\Omega),\end{equation}  using again    \cite[Proposition 6.5]{DoPaPi2} we have that
$[\D(f^\Gamma f^\Omega\T) ]$ is a partial generalized crossed product with factor set given by:
$$\mathcal F^{f^\Gamma f^\Omega\T}=\{f^\Gamma f^\Omega\T_{g,h}:(f^\Gamma f^\Omega\T)_g\ot_R(f^\Gamma f^\Omega\T)_h\ni u''_g\ot_R u''_h\mapsto u''_hu''_g\in 1_g(f^\Gamma f^\Omega\T)_{gh}\}_{g,h\in G}.$$   
We shall show that
$[\Delta(\Gamma )][\Delta(\Omega)] = [\Delta(f^\Gamma \T)][\Delta(f^\Omega\T)]=[\Delta(f^\Gamma f^\Omega\T)].$ By \eqref{prodc}  we have
$$[\D(f^\Gamma\T)] [\D(f^\Omega\T)]=\left[\bigoplus\limits_{g\in G}f^\Gamma\T_g\ot _R(D_{g^\m})_g \ot_Rf^\Omega\T_g\right].$$ Consider the unital partial representation, $\Lambda
:G\ni g\mapsto[f^\Gamma\T_g\ot _R(D_{g^\m})_g \ot_Rf^\Omega\T_g] \in \Pics (R),$  then  using \eqref{fla} we get that  the family $f^\Lambda$ given by 
\begin{align*}f_{g,h}^\Lambda(u^\G_g\ot d_{g^{\m}}\ot_R v^\Omega_g\ot_R w^\G_h\ot_R d'_{h^{\m}}\ot_R x^\Omega_h)&= f_{g,h}^\G(u^\G_g\ot w^\G_h)\ot_Rd'_{h^{\m}}\al_{h^{\m}}(d_{g^{\m}}1_h)\ot_Rf_{g,h}^\Om( v^\Omega_g\ot_Rx^\Omega_h)\\
&= w^\G_hu^\G_g\ot_Rd'_{h^{\m}}\al_{h^{\m}}(d_{g^{\m}}1_h)\ot_R x^\Omega_hv^\Omega_g.
\end{align*} is a factor set  for $\Lambda.$
We need to construct a  family of $R$-bimodule homomorphisms $$F_g=\{F_g:f^\Gamma\T(g)\ot_R (D_{g^\m})_g\ot_R f^\Omega\T(g)\to f^\Gamma f^\Omega\T(g)\}_{g\in G}$$ such that the diagram \eqref{morfismodepcgp0} is commutative. \\
Take $g\in G,$  and let   $F_g:f^\Gamma\T(g)\ot_R (D_{g^\m})_g\ot_R f^\Omega\T(g)\to f^\Gamma f^\Omega\T(g)$ be  given by the chain of $R$-bimodule isomorphisms
\begin{align*}
f^\Gamma\T(g)\ot_R (D_{g^\m})_g\ot_R f^\Omega\T(g)&  \simeq  f^\Gamma(g)\ot_R (D_g)_{g^\m} \ot_R (D_{g^\m})_g\ot_R f^\Omega\T(g)\\
&\simeq f^\Gamma(g)\ot_R D_g \ot_R  f^\Omega\T(g)\\
&\simeq f^\Gamma(g)\ot_R f^\Omega\T(g)
\\& \simeq (f^\Gamma f^\Omega\T)(g) ,
\end{align*} where for the first  isomomorphism we consider that  induced by the map $\lambda$ given in \cite[p. 222]{DoPaPi2} and the second isomorphism comes from  \eqref{fte}.
Then  $F_g (u_g\ot_R  d_{g^{\m}}\ot_R  u'_g)$ is given by
\begin{align*} u_g\ot_R  d_{g^{\m}}\ot_R  u'_g \mapsto u_g\ot_R 1_g \ot_R   d_{g^{\m}}\ot_R  u'_g\mapsto u_g\ot_R \al_g(   d_{g^{\m}})\ot_R  u'_g\mapsto u_g \al_g(   d_{g^{\m}})\ot_R  u'_g.
\end{align*}

In order to show the commutativity of diagram \eqref{morfismodepcgp0} in our case denote 
$$\mathfrak c_{g,h}=F_{gh}\circ f_{g,h}^\Lambda(u^\G_g\ot_R d_{g^{\m}}\ot_R v^\Omega_g\ot_R w^\G_h\ot_R d'_{h^{\m}}\ot_R x^\Omega_h),$$ and 
$$\mathfrak c'_{g,h}=f^\Gamma f^\Omega\T_{g,h}\circ (F_g\ot_R F_h)(u^\G_g\ot_R d_{g^{\m}}\ot_R v^\Omega_g\ot_R w^\G_h\ot_R d'_{h^{\m}}\ot_R x^\Omega_h). $$
 Then 
\begin{align*}
\mathfrak c_{g,h}&=F_{gh}( w^\G_hu^\G_g\ot_Rd'_{h^{\m}}\al_{h^{\m}}(d_{g^{\m}}1_h)\ot_R x^\Omega_hv^\Omega_g)=w^\G_hu^\G_g\al_{gh}(d'_{h^{\m}}\al_{h^{\m}}(d_{g^{\m}}1_h))\ot_R x^\Omega_hv^\Omega_g
\\
&=w^\G_hu^\G_g\al_{g}(d_{g^{\m}})\al_{gh}(d'_{h^{\m}}1_{(gh)^{\m}})\ot_R x^\Omega_hv^\Omega_g,
.\end{align*}
while
\begin{align*}\mathfrak c'_{g,h}&= f^\Gamma f^\Omega\T_{g,h}(u^\G_g\al_g(d_{g^{\m}})\ot_R v^\Omega_g\ot_R w^\G_h\al_h( d'_{h^{\m}})\ot_R x^\Omega_h) =(w^\G_h\al_h( d'_{h^{\m}})\ot_R x^\Omega_h)(u^\G_g\al_g(d_{g^{\m}})\ot_R v^\Omega_g)
\\&\stackrel{(\ast)}=w^\G_h\al_h( d'_{h^{\m}})u^\G_g\al_g(d_{g^{\m}})\ot_R  x^\Omega_hv^\Omega_g
=w^\G_hu^\G_g\al_g(d_{g^{\m}})\al_{gh}( d'_{h^{\m}}1_{(gh)^{\m}})\ot_R  x^\Omega_hv^\Omega_g.
\end{align*}
where the  equality $(\ast)$ comes from the product in the $R^\al$-algebra $A\otimes_{R^\al} B.$  Thus $\mathfrak c_{g,h}=\mathfrak c'_{g,h},$ and
 $[\Delta(f^\Gamma \T)][\Delta(f^\Omega\T)])=[\Delta(f^\Gamma f^\Omega\T)].$ By Theorem \ref{Aisopcgp}  we get $\eta([\Delta(f^\Gamma\T)])=\llbracket A^{\rm op} \rrbracket$ and $\eta([\Delta(f^\Omega\T)])=\llbracket B^{\rm op} \rrbracket .$ Then by \eqref{fimor}, using again  Theorem \ref{Aisopcgp}, we obtain   that
$$\eta([\Delta(\Gamma)][\Delta(\Omega)])=\eta([\Delta(f^\Gamma f^\Omega\T)]=\llbracket (A \otimes_{R^\alpha} B)^{\rm op} \rrbracket=\llbracket A ^{\rm op} \rrbracket\llbracket B^{\rm op} \rrbracket=\eta([\Delta(\Gamma)])\eta([\Delta(\Omega)]),$$  concluding that $\eta$ is a homomorphism. To show that $\eta$ is an epimorphism, take $\llbracket A \rrbracket \in \B(R/R^\al).$ By Theorem \ref{Aisopcgp} there exists $[\D(\Gm)] \in \C(\T/R)$ with $\eta([\D(\Gm)])=\llbracket A^{\rm op} \rrbracket.$ Thus  $\llbracket A \rrbracket^{-1}=\llbracket A^{\rm op} \rrbracket \in {\rm im}(\eta),$ and  $\llbracket A\rrbracket \in {\rm im}(\eta),$ showing that  $\eta$ is surjective. 
It remains to prove that  $\ker(\eta)= {\rm im}(\mathcal{L}_0).$ Take  $[\D(\Gm)] \in {\rm Im}(\mathcal{L}_0)$,  there exists $[P]\in \Pic_R(R)$ such that $\mathcal{L}_0([P])=[\D(\Gm)]$, where  $\D(\Gm)\simeq \End_{R^\al}(P)$, as  $R^\al$-algebras.
Hence, 
$$\eta(\mathcal{L}_0([P]))=\llbracket\D(\Gm)\rrbracket=\llbracket\End_{R^\al}(P)\rrbracket=\llbracket R^\al\rrbracket$$
 in  $\B(R/R^\al),$ which implies ${\rm Im}(\mathcal{L}_0)\subseteq \ker(\eta)$.   Conversely, if $[\D(\Gm)] \in \ker(\eta),$ then there exists a  faithfull f.g.p  $R^\al$-module $P$ such that $\D(\Gm)\simeq \End_{R^\al}(P)$ as $R^\al$-algebras. Since, by \cite[Corollary 1.1.10]{demeyer1971separable} $P$  is a $R^\alpha$-progenerator, then   by  Propositions \ref{proparaL} and \ref{teoc}  there is a $R$-module structure on $P$ such that $[P]\in \Pic_R(R)$ and 
$\D(\Gm)\simeq \D(\T^P)$ as partial generalized crossed products. Then  $\mathcal{L}_0([P])=[\D(\T^P)]=[\D(\Gm)].$
	Therefore, $\ker(\eta)\subseteq {\rm Im}(\mathcal{L}_0),$ completing our proof. \end{dem}

\section{Comparing the seven-term exact sequences}\label{sec:Comparing}
In this section we shall show that the seven-term exact sequence for partial Galois extensions of commutative rings (see Theorem \ref{galoise} below), can be obtained from the main result of \cite{DoRo}. 
First for the reader's convenience, we   recall  some notions and notation.  We fix an extension of  non-necessarily commutative rings  $R'\subseteq S,$ and   a unital  partial representation $G\ni g\mapsto \G_g\in  \mathcal{S}_{R'}(S),$ in the semigroup of $R'$-subbimodules of $S.$ Write $\G :G \ni g\mapsto [\G_g]\in  \Pics(R')$ and 
  \begin{equation}\label{PicsZero}
	\Pics_0(R')=\{[P]\in \Pics(R'): \ P|R' \ \mbox{as bimodules} \}.
\end{equation}
Then by \cite[Proposition 3.11]{DoRo} $\Gamma$ induces a unital partial action   $\overline{\al}$ of  $G$ on $\Z=\Z(R').$ Moreover, let $ \overline\al^*$ be the partial action of $G$ on $\Pics_0(R')$ presented in \cite[p. 56]{DoRo}.  Consider the group  $\Pic_\Z(R')^{(G)} $ given in the proof of  Proposition \ref{teoc},  and define  $\mathcal{B}(\Gamma/R')$ by the exact sequence  
\begin{equation}\label{GroupB}
	\xymatrix{ \Pic_\Z(R')^{(G)}\ar[r]^-{\mathcal{L}} & \C(\Gamma/R')\ar[r]& \mathcal{B}(\Gamma/R')\ar[r] & 1 },
\end{equation}
 that is $	\mathcal{B}(\Gamma /R')=\dfrac{\C(\Gamma /R')}{\mbox{Im}(\mathcal{L})},$ where
$\mathcal{L}: \Pic_\Z(R')^{(G)}\ni [P]\mapsto [\Delta(\Omega^P)]\in  \C(\Gamma/R') $ is the group homomorphism defined in  \cite[Theorem 5.10]{DoRo}. 

We also need  the group  $\overline{H^1}(G,\al^*,\Pics_0(R'))$  defined  by the exact sequence 
$$\xymatrix{  \Pic_\Z(R')^{(G)}\ar@{-->}[rr]\ar[rd]_{\mathcal{L}} & & Z^1(G,\overline\al^*,\Pics_0(R'))\ar[r] &   \overline{H^1}(G,\al^*,\Pics_0(R'))\ar[r] & 1\\
	& \C(\Gamma/R')\ar[ru]_{\zeta}& & &    }$$
that is,
\begin{equation}\label{GroupOverlineH}
	\overline{H^1}(G,\overline\al^*,\Pics_0(R'))=\dfrac{Z^1(G,\overline\al^*,\Pics_0(R'))}{(\zeta\circ\mathcal{L})(\Pic_\Z(R')^{(G)})},
\end{equation}
being $\zeta: \C(\Gamma/R')\ni [\Delta(\Omega)]\mapsto f^\Omega \in Z^1(G,\overline\al^*,\Pics_0(R'))$ the group homomorphism defined in \cite[Lemma 5.21]{DoRo}, and 
\begin{equation}\label{fom}f^\Omega: G\ni g\mapsto [\Omega_g\otimes_{R'} \Gamma_{g^{\m}} ] \in \Pics_0(R'),\end{equation} 
as defined in \cite[p. 56]{DoRo}.

 Suppose that  the partial representation $\G$ is endowed with a factor set, and consider the ring extension $R'\subseteq \D(\G).$
Let us denote by $\Aut_{R'\mbox{-rings}}(\D(\Gamma))$ the unit group of the monoid ${\rm End}_{R'\mbox{-rings}}(\D(\Gamma))$ of all ring endomorphisms which act by identity on $R'.$ Write 
\begin{equation*}
	\Aut_{R'\mbox{-rings}}(\D(\Gamma))^{(G)}=\{f\in \Aut_{R'\mbox{-rings}}(\D(\Gamma)); \ f(\G_g)=\G_g, \ \forall \ g\in G\}.
\end{equation*}
Obviously, $\Aut_{R'\mbox{-rings}}(\D(\Gamma))^{(G)}$ is a subgroup of $\Aut_{R'\mbox{-rings}}(\D(\Gamma))$.  Let  $\p_\Z(\D(\Gamma)/R')^{(G)}$ be the group  defined in \eqref{pz}, and consider the group homomorphism $\mathcal{E}: \Aut_{R\mbox{-rings}}(\D(\Gamma)) \to \p(\D(\Gamma)/R')$ which appears in the first row of the commutative diagram given by \cite[Lemma 5.4]{DoRo}. It is proven in that Lemma that 
$$\mathcal{E}(f)=\xymatrix@C=1.2cm{ [R]\ar@{=>}[r]|{[\iota_f]} & [\Delta(\Gamma)_f]} \in \p_\Z(\D(\Gamma)/R)^{(G)},$$ for all $f\in \Aut_{R'\mbox{-rings}}(\D(\Gamma))^{(G)},$ where  the map $\iota_f$ is the canonical inclusion of $R$-bimodules 
and $\Delta(\Gamma)_f$ is the left  $\Delta(\Gamma)$-module  $\Delta(\Gamma)$ considered as a  right $\Delta(\Gamma)$ -module via $s_fs'=sf(s'),$ for all  $s_f\in \Delta(\Gamma)_f$ and $s'\in \Delta(\Gamma).$

We shall recall the morphisms of the seven-term exact sequence constructed in \cite{DoRo}.

\noindent {\bf The morphism $\varphi_1.$} The map $\varphi_1$ is constructed  using the commutative diagram in \cite[Lemma 5.4]{DoRo}, and  is defined as follows:
$$ \begin{array}{c c c l}\varphi_1: &   H_\Gamma^{1}(G, \overline\al,\Z) & \longrightarrow &  \p_\Z(\D(\Gamma)/R')^{(G)}\\ 
&{\rm cls}(\widetilde f) & \longmapsto & \mathcal{E}(f)=\xymatrix@C=1.2cm{ [R']\ar@{=>}[r]|{[\iota_f]} & [\Delta(\Gamma)_f]} 
		\end{array},$$
where the relation between the $1$-cocycle $ \widetilde f$ and $f$ presented in the proof of \cite[Theorem 5.5]{DoRo}. It is shown in \cite[p. 40-41]{DoRo} that $f$   is the   unique element in $ \Aut_{R'\mbox{-rings}}(\D(\Gamma))^{(G)}$ such that 

\begin{equation}\label{forf}f:\Gamma_g\ni u_g\mapsto \widetilde f(g)u_g,\in \Gamma_g,\end{equation} for all $g\in G. $ 

\noindent {\bf The morphism $\varphi_2.$} By \cite[Proposition 5.3]{DoRo} we have
$$ \begin{array}{c c c l}\varphi_2: &  \p_\Z(\D(\Gamma)/R')^{(G)}& \longrightarrow &  \Pic_\Z(R')\cap \Pics_\Z(R')^{\al^*}\\ &\xymatrix@C=1.2cm{ [P]\ar@{=>}[r]|{[\phi]} & [X]}  & \longmapsto & [P]
		\end{array}.$$

\noindent {\bf The morphism $\varphi_3.$} It is proved in \cite[Proposition 4.2]{DoRo} that the set $$\C_0(\G/R')=\{[\D(\Omega)]\in C(\G/R')\mid \Omega_g\simeq \G_g, \;\text{as} \; R'\text{-bimodules} \; \forall g\in G \}$$ is a subgroup of  $C(\G/R').$ The homomorphism $\varphi _3$  is defined in \cite[p. 45]{DoRo} by 
$$\varphi_3:  \Pic_\Z(R')\cap \Pics_\Z(R')^{\al^*}\ni  [P]\mapsto [\Delta(\Gamma^P)]\in C_0(\G/R'),$$
 where $\Delta(\Gamma^P)=\bigoplus_{g\in G}\Gamma^P_g,$ and $\Gamma^P_g=P\otimes_{R'}\Gamma_g\otimes_{R'} P^*.$

 We shall work with the map
\begin{equation*}\label{fi3}
		\begin{array}{c c c l}
			\zeta_0 \circ \varphi_3: & \Pic_\Z(R')\cap \Pics_\Z(R')^{\al^*} & \longrightarrow &  H_\Gamma^2(G,\overline\al,\Z)\\
			& [P] & \longmapsto &{\rm cls} (\widetilde\tau_{-,-}),
		\end{array}
\end{equation*}
where  $\zeta_0: C_0(\G/R')\to H_\Gamma^2(G,\overline\al,\Z) $ is the group isomorphism  of  \cite[Theorem 4.4]{DoRo}.

\noindent {\bf The morphism $\varphi_4.$} We have
\begin{equation*}
\begin{array}{c c c l}
			\varphi_4: &  H_\Gamma^2(G,\overline\al,\Z)& \longrightarrow & \mathcal{B}(\Gamma /R') \\
			& {\rm cls}( \omega) & \longmapsto & [\D(\Si)]{\rm Im}\,\mathcal L,
		\end{array}
\end{equation*}
where the relation between $\omega$ and the partial generalized crossed product $\D(\Si)$ is given in the proof of \cite[Theorem 4.4]{DoRo} as follows.  Set $\Si_g:=\G_g $ (see \cite[p. 35]{DoRo}). Then the family
$$\{f^\Si_{g,h}:\Si_g\ot_{R'}\Si_h\ni u_g\ot_{R'}u_h\mapsto  \omega_{g,h}f_{g,h}^\G(u_g\ot_{R'}u_h)\}_{g,h\in G}$$ is a factor set for $\Si $ and $[\D(\Si)]\in C_0(\G/R')$ is such that $\zeta([\D(\Si)])={\rm cls}( \omega).$

In  order to recall the last two morphisms we give the following.

\noindent {\bf Notation:} For a  morphism $f:G_1\to G_2$ of abelian groups,  the co-restriction of $f$ is the canonical map $f^c:G_2\to G_2/{\rm im}f.$ We also denote by $f^*$ the morphism  $f^*(x)=f(x^\m),$ for all $x\in G_1.$

\noindent {\bf The morphism $\varphi_5.$} The map  $\varphi_5:\mathcal{B}(\G/R)\to \overline{H^1}(G,\al^*,\Pics_0(R))$ is introduced in \cite[p. 59]{DoRo} by means of the following commutative  diagram:

\begin{equation*}
	\xymatrix{ & \C_0(\G/R)\ar[rd]^{\varphi_4}\ar@{_(->}[d] & & \\
		\Pic_\Z(R')^{(G)} \ar[r]^{\mathcal{L}} & \C(\G/R')\ar[r]^{\mathcal{L}^c}\ar[d]_{\zeta} & \mathcal{B}(\G/R) \ar[r]\ar@{-->}[d]^{\varphi_5} & 1\\
		& Z^1(G,\al^*,\Pics_0(R'))\ar[r]^{(\zeta \mathcal{L})^c} & \overline{H^1}(G,\al^*,\Pics_0(R'))\ar[r] & 1.   } \label{definicaodevarphi5}.
\end{equation*}

\noindent {\bf The morphism $\varphi_6.$} Let $
			\delta:  Z^1(G,\al^*,\Pics_0(R'))\ni f \mapsto  {\rm cls}\left(\widetilde{\beta_{-,-,-}^f}\right)\in H_\G^3(G,\al, \Z)$ be the morphism defined in \cite[Lemma 5.25]{DoRo}, and the cocycle $\widetilde{\beta_{-,-,-}^f}$ is defined  using
\cite[Proposition 3.17 ]{DoRo}  and  \cite[Lemma 3.10]{DoRo}. Then the homomorphism  $\varphi_6:\overline{H^1}(G,\al^*,\Pics_0(R'))\longrightarrow H^3_\G (G,\al,\Z)$ is defined  via the commutative diagram:
\begin{equation*}
	\xymatrix{ \Pic_\Z(R')^{(G)}\ar[r]^{\zeta\mathcal{L}\ \ \ \ \ \ \ } & Z^1(G,\al^*,\Pics_0(R'))\ar[r]^{(\zeta\mathcal{L})^c}\ar[d]_{\delta} & \overline{H^1}(G,\al^*,\Pics_0(R'))\ar[r]\ar@{-->}[ld]^{\varphi_6}& 1\\
		& H^3_\G (G,\al,\Z). & &   }
\end{equation*}

Using the above morphisms $\varphi_i, 1\leq i \leq 6,$  we have the next.

\begin{teo}\label{md}\cite[Theorem 5.27]{DoRo} \label{theo:main}  The sequence of group homomorphisms

	\begin{equation*}
		\xymatrix{1\ar[r] &  H_\Gamma^{1}(G, \overline\al,\Z)\ar[r]^-{\varphi_1} &  \p_\Z(\D(\Gamma)/R')^{(G)}\ar[r]^-{\varphi_2} & \Pic_\Z(R')\cap\Pics_\Z(R')^{\al^*}\ar[r]^-{\varphi_3} & H_\Gamma^2(G,\overline\al,\Z)\\ 
			\ar[r]^-{\varphi_4}& \mathcal{B}(\Gamma/R')\ar[r]^-{\varphi_5} & 	 \overline{H}^1(G,\overline\al^*,\Pics_0(R'))\ar[r]^-{\varphi_6} & H^3_\Gamma(G,\overline\al,\Z)  }
	\end{equation*}
	is exact.

\end{teo}
Now we are ready to show that if $R\supseteq R^\alpha$ is  a partial Galois extension of commutative rings,  then Theorem \ref{md} recovers the main result of \cite{DoPaPi2}. That is we have the next.
\begin{teo}\label{galoise} Let $R\supseteq R^\alpha$ be a partial Galois extension of commutative rings, then there exists an exact sequence of group homomorphisms
\begin{align}\label{exact}
&0\longrightarrow H^1(G,\alpha , R){\stackrel{\phi_1}\longrightarrow} {\bf Pic}_{R^\alpha}(R^\alpha ){\stackrel{\phi_2}\longrightarrow}{\bf PicS}_R(R)^{\alpha^*}\cap {\bf Pic}_R(R) {\stackrel{\phi_3}\longrightarrow }H^2(G,\alpha, R) {\stackrel{\phi_4}\longrightarrow} B(R/R^\alpha){\stackrel{\phi_5}\longrightarrow}\\
& H^1(G,\alpha^*,{\bf PicS}_R(R)){\stackrel{\phi_6}\longrightarrow}   
H^3 (G,\alpha , R).\nonumber
\end{align} 
\end{teo}
\begin{dem} We shall obtain the sequence \eqref{exact} by taking in the sequence of Theorem~\ref{theo:main}   the partial representation $\Gamma $ to be equal to   the unital partial representation $\T $ defined in \eqref{definicaodeT0}. Firstly, we check that in this case   the groups in  \eqref{exact} are either  equal, or isomorphic to the groups of the sequence of Theorem \ref{md}.  Indeed, since $\Z=\Z(R)=R,$ it follows by Remark \ref{coigual} that the groups $H^n(G,\alpha , R)$ and $H_\Gamma^{n}(G, \overline\al,\Z)$ coincide for $n\in \{1,2,3\}.$ Moreover,  it is clear that  $\Pics_0(R)\supseteq \Pics_R(R),$ and thus the equality  $\Pics_0(R)=\Pics_R(R)$ follows  by \cite[Lemma 5.20 (ii)]{DoRo}.  Furthermore, by \eqref{cong} the partial actions $\alpha^*$ and $\overline\al^*$ from \cite[p. 56]{DoRo} coincide, which implies  $ Z^1(G,\overline\al^*,\Pics_0(R))=Z^1(G,\alpha^*,{\bf PicS}_R(R))$ and   $ B^1(G,\overline\al^*,\Pics_0(R))=B^1(G,\alpha^*,{\bf PicS}_R(R)).$ Moreover,  as observed in the proof of item  (1)  of Proposition  \ref{teoc} we have that 
$\Pic_\Z(R)^{(G)} =  \Pic_R(R)$ and $\mathcal L=
\mathcal L_0.$  Then
\begin{align*}f\in B^1(G,\al^*,\Pics_R(R))&\stackrel{\eqref{b1}}\Longleftrightarrow  f(g)=[P]\alpha^*_g([P^*][D_{g^{\m}}]),\,\,\,{\rm for\, some}\,\,\, [P]\in \Pic_R(R) \\
&\stackrel{\eqref{cong}}{\Longleftrightarrow }f(g)=[P][(D_g)_{g^{\m}}][P^*][(D_{g^{\m}})_g] , \,\,\,{\rm for\, some}\,\,\, [P]\in \Pic_R(R) \\
&\stackrel{\eqref{omp}}{\Longleftrightarrow }f(g)=[(\T^ P)_g\ot_R(D_{g^{\m}})_g],  \,\,\,{\rm for\, some}\,\,\, [P]\in \Pic_R(R) \\
&\stackrel{}{\Longleftrightarrow }f=(\zeta\circ\mathcal{L})([P]),  \,\,\,{\rm for\, some}\,\,\, [P]\in \Pic_R(R). \\
\end{align*}
Hence $(\zeta\circ\mathcal{L})(\Pic_\Z(R)^{(G)}) =B^1(G,\al^*,\Pics_R(R))$ and 
$$ \overline{H}^1(G,\overline\al^*,\Pics_0(R))=\dfrac{Z^1(G,\overline\al^*,\Pics_0(R))}{(\zeta\circ\mathcal{L})(\Pic_\Z(R)^{(G)})}=H^1(G,\alpha^*,{\bf PicS}_R(R)).$$

\noindent In addition,  we have that $\T(g)=[\T_0(g)],$ where $\T_0:G\ni g\mapsto D_g\delta_g\in R\star_\alpha G$ is the partial representation defined in \eqref{t0}. Since $\Delta(\T)=\Delta(\T_0)=R\star_\alpha G,$ then  Theorem \ref{isopic}  implies that  there is a group isomorphism
 $$\Pic_{R^\al}(R^\al) \simeq  \p_R(R\star_\alpha G/R)^{(G)}=\p_R(\Delta(\T)/R)^{(G)}.$$ 
Further,  since $\mathcal L=
\mathcal L_0$ it follows from Theorem  \ref{exac} that
$\B(R/R^\al)\stackrel{\eqref{etta}}\simeq \dfrac{\mathcal{C}(\T/R)}{{\rm Im}(\mathcal L_0)}=\mathcal{B}(\T/R),$ as desired.

\noindent Now we show that    the morphisms in \cite[p. 227]{DoPaPi2} can be obtained from the morphisms in Theorem \ref{md}. Notice that :
$$\xi^{\m} \circ\varphi_1 : H^1(G,\alpha , R)\ni {\rm cls}( \widetilde f)\mapsto \left[R^G\right]\in  \Pic_{R^\al}(R^\al) , $$ where $\xi$ is the group isomorphism defined in \eqref{isogr},
$R^G$ is given by  \eqref{mg} and the left $R\star_\alpha G$-module structure of $R$ is defined by  \eqref{triangle}, that is, for $r\in R,$ $g\in G$ and $a_g\in D_g$ 
\begin{align*}(a_g\delta_g)\cdot r&=\iota_f^{\m}((a_g\delta_g)(\iota_f(r))_f(1_{g^\m}\delta_{g^\m}))
= \iota_f^{\m} (a_g\alpha_g(r1_{g^\m})\delta_gf(1_{g^\m}\delta_{g^\m}))
\\&\stackrel{\eqref{forf}}= \iota_f^{\m} (a_g\alpha_g(r1_{g^\m})\delta_g  \widetilde f(g^\m) (1_{g^\m}\delta_{g^\m})) =a_g\alpha_g(r1_{g^\m})\alpha_g(\widetilde f(g^\m))
\stackrel{\eqref{z1}}=a_g\alpha_g(r1_{g^\m})\widetilde f(g)^\m.
\end{align*} 
 This implies that  $(\xi^{\m} \circ\varphi_1)({\rm cls}(\widetilde f))=\phi_1({\rm cls}(\widetilde f^{\m})),$ 
for all ${\rm cls}(\widetilde f)\in  H^1(G,\alpha , R),$ where $\phi_1$ is the group homomorphism defined in \cite[p. 202]{DoPaPi2}. 
Hence  $(\xi^{\m})^* \circ\varphi_1=\phi_1.$  

\noindent We pass to the second morphism. Notice that
$$
\begin{array}{c c c l}
	 \varphi_2 \circ \xi : & \Pic_{R^\al}(R^\al) & \longrightarrow & {\bf PicS}_R(R)^{\alpha^*}\cap {\bf Pic}_R(R)    \\
	&  [P_0] & \longmapsto & [P_0\ot_{R^\al} R]\end{array},$$ which coincides with the morphism $ \phi_2$ defined in \cite[p. 203]{DoPaPi2}.  

\noindent With respect to the third morphism,
we shall show that $\zeta_0\circ \varphi_3([P])=\phi_3([P^*]),$ where $\phi_3$ is the map  given in \cite[p. 205]{DoPaPi2}.  Take  $[P]\in {\bf PicS}_R(R)^{\alpha^*}\cap {\bf Pic}_R(R) $ and write $\phi_3([P]) = {\rm cls}( \omega_{-,-}),$ where the $2$-cocycle $ \omega_{-,-}$ is defined in \cite[p. 205]{DoPaPi2}.
Then  $\phi_3([P^*]) = {\rm cls}( \omega^\m_{-,-}).$ Let $\mathfrak a\in C^{1}(G, \alpha, R)$ be given by $\mathfrak a_g=\om_{g, g^{\m}},$ for all $g\in G.$ Then 
\begin{equation}\label{coba}(\delta^1 \mathfrak a)(g,h)=\al_g(\om_{h, h^{\m}}1_{g^{\m}})\om_{g, g^{\m}}\om^{\m}_{gh, (gh)^{\m}}.
\end{equation}
 By the $2$-cocycle identity we have $$\al_g(\om_{h, h^{\m}}1_{g^{\m}})=\om_{gh, h^{\m}}\om_{g,h}.$$  Moreover,  
$\al_{gh}(\om_{h^{\m}, g^{\m}})\om_{gh, (gh)^{\m}}=\om_{g, g^{\m}}\om_{gh, h^{\m}}.$ That is, $$\om_{g, g^{\m}} \om^{\m}_{gh, (gh)^{\m}}=\al_{gh}(\om_{h^{\m}, g^{\m}})\om^{\m}_{gh, h^{\m}}.$$    

\noindent Thus it follows by  \eqref{coba}, that
$(\delta^1 \mathfrak a)(g,h)=\om_{gh, h^{\m}}\om_{g,h}\al_{gh}(\om_{h^{\m}, g^{\m}})\om^{\m}_{gh, h^{\m}}=\om_{g,h}\al_{gh}(\om_{h^{\m}, g^{\m}}).$ Hence the map $\mathfrak b_{-,-}\in C^2(G,\alpha, R),$ given by $\mathfrak b_{g, h}=\al_{gh}(\om_{h^{\m}, g^{\m}}),$ belongs to $Z^2(G,\alpha, R)$ and is cohomologous to  $\omega^\m_{-,-},$ which implies   $\phi_3([P^*]) = {\rm cls}(\mathfrak b_{-,-}).$

Furthermore, from the proof of the surjectivity of   $\zeta $ in \cite[Theorem 4.4]{DoRo}, we have that $\zeta([\D(\Sigma)])={\rm cls}(\mathfrak b_{-,-}),$ where $\Sigma_g=(D_g)_{g^\m}, g\in G,$ and $\D(\Sigma)$ is endowed with the factor set 
$$f_{g,h}^\Sigma=\{f^\Sigma_{g,h}:\Sigma_g\otimes_R \Sigma_h\ni u_g\otimes_R u'_h \mapsto\mathfrak b_{g,h}u_g\alpha_g(u'_h1_{g^\m})\in 1_g\Sigma_{gh}\}_{g,h\in G}.$$  
But the definition of $\varphi_3$ implies $\varphi_3([P])=[\D(\T^P)],$  as given  in \eqref{omp}. We shall show that $[\D(\T^P)]=[\D(\Sigma)]$ in $C(\T/R).$  First,  one needs to define a morphism $F:\D(\T^P)\to \D(\Sigma).$   Since $[P]\in {\bf PicS}_R(R)^{\alpha^*}\cap {\bf Pic}_R(R),$ then $[P]$ is fixed by the partial action $\alpha ^{\ast},$ so that there is a $R$-bimodule isomorphism $\psi_g:P1_g\to (P1_{g^\m})_g,$ for each $g \in G,$  where the $R$-bimodule structure on $(P1_{g^\m})_g$ is given in \eqref{cena}.  In particular
\begin{equation}\label{fig}
\psi_g(rx)=\alpha_{g^{\m}}(r 1_g)\psi_g(x)\,\,\,\,\,\,\,\,\,\,\,\,{\rm and}\,\,\,\,\,\,\,\, \,\,\,\,\psi^{\m}_g(rx')=\alpha_{g}(r1_{g^{\m}})\psi^{\m}_g(x'),
\end{equation}
for all $ r\in R, x\in P1_g$ and $x'\in P1_{g^{\m}}.$
Now let  $F_g$ be the  composition of the $R$-bimodule isomorphisms:
\begin{align*}
\T^P_g&=P\ot_R (D_g)_{g^\m}\ot_R P^*=1_gP\ot_R (D_g)_{g^\m}\ot_R P^*\stackrel{\psi_g}\simeq (P1_{g^\m})_g\ot_R (D_g)_{g^\m}\ot_R P^*
\\& \stackrel{\eqref{cong}}\simeq (D_g)_{g^\m}\ot_R P\ot_R(D_{g^\m})_g\ot_R (D_g)_{g^\m}\ot_R P^*\stackrel{\eqref{isot}}\simeq (D_g)_{g^\m}\ot_R P\ot_RD_{g^\m}\ot_R P^*
\\& \simeq (D_g)_{g^\m}\ot_R D_{g^\m}\ot_R P\ot_R P^*
\simeq (D_g)_{g^\m}\ot_RR \simeq (D_g)_{g^\m}=\Si_g.
\end{align*}
Thus the image of $F_g:\T^P_g\to \Sigma_g$ at $p\ot_Rd_g\ot_Rf\in \T^P_g$ is given by

\begin{align*}
p\ot_Rd_g\ot_Rf&=1_gp\ot_Rd_g\ot_Rf\mapsto\psi_g(1_gp)\ot_Rd_g\ot_Rf\stackrel{\eqref{isocon}}\mapsto1_g\ot_R\psi_g(1_gp)\ot_R1_{g^\m}\ot_Rd_g\ot_Rf
\\&\mapsto1_g\ot_R\psi_g(1_gp)\ot_R\alpha_{g^\m}(d_g)\ot_Rf\mapsto1_g\ot_R\alpha_{g^\m}(d_g)\ot_R\psi_g(1_gp)\ot_Rf
\\&\mapsto d_g\ot_Rf(\psi_g(1_gp))\mapsto d_g\alpha_g\left(f(\psi_g(1_gp))1_{g^\m}\right).
\end{align*}
To determine   $F_g^{\m}(d_g)$  for each $d_g\in \Si_g$  take  $\{p_i,f_i\}_{1\leq i\leq n}$  a dual basis of $P$ as an $R$-module. Then
\begin{align*}
d_g&\mapsto d_g\ot_R 1_R\mapsto \sum_{i} d_g\ot_R1_{g^{\m}}\ot_R p_i\ot_R f_i\mapsto \sum_{i} d_g\ot_R p_i\ot_R1_{g^{\m}}\ot_R f_i
\\&\mapsto \sum_{i} d_g\ot_R p_i\ot_R1_{g^{\m}}\ot_R1_g\ot_R f_i \mapsto \sum_{i} \alpha _{g^\m}(d_g) p_i\ot_R1_g\ot_R f_i\\
& \mapsto  \sum_{i} \psi^{\m}_g (\alpha _{g^\m}(d_g) p_i) \ot_R1_g\ot_R f_i =  \sum_{i} \psi^{\m}_g ( p_i) \ot_R d_g \ot_R f_i.
\end{align*}
By Remark \ref{sod} it is enough to show that   the following  diagram is commutative:
	\begin{equation*}
	\xymatrix{ \T^P_g\ot_R\T^P_h\ar[rr]^{f_{g,h}^P}\ar[d]_{F_g\ot F_h} & & 1_g\T^P_{gh}\ar[d]^{F_{gh}}\\
		\Sigma_g\ot_R\Sigma_h\ar[rr]_{f_{g,h}^\Sigma} & & 1_g\Sigma_{gh} } \label{morfismodepcgp}
	\end{equation*}

Indeed, let $\mathfrak{P}_{g,h}=F_{gh}\circ f_{g,h}^{P}\circ( F^{\m}_g\otimes F^{\m}_h)(d_g \ot_R d'_h).$ Then
\begin{align*}
\mathfrak{P}_{g,h}&= \dsum_{i,j} F_{gh}\circ f_{g,h}^{P}(\psi^{\m}_g(p_i1_{g^{\m}})\ot_Rd_g\ot_R f_i\ot_R \psi^{\m}_h(p_j1_{h^{\m}})\ot_Rd'_h\ot_R f_j)\\
& \stackrel{\eqref{conjuntodefatoresparaP}}=\dsum_{i,j}F_{gh}(\psi^{\m}_g(p_i1_{g^{\m}})\ot_Rd_g\al_g(f_i( \psi^{\m}_h(p_j1_{h^{\m}}))d'_h1_{g^{\m}})\ot_R f_j)
\\&=\dsum_{i,j}F_{gh}(\psi^{\m}_g(p_i1_{g^{\m}})\al_g(f_i( \psi^{\m}_h(p_j1_{h^{\m}}))d'_h1_{g^{\m}})\ot_R 1_{gh}d_g\ot_R f_j)\\
&\stackrel{\eqref{fig}}=\dsum_{i,j}F_{gh}(\psi^{\m}_g(p_if_i( \psi^{\m}_h(p_j1_{h^{\m}}))d'_h1_{g^{\m}})\ot_R 1_{gh} d_g\ot_R f_j)
\\&=\dsum_{j}F_{gh}(\psi^{\m}_g( \psi^{\m}_h(p_j1_{h^{\m}}))d'_h1_{g^{\m}})\ot_R 1_{gh} d_g\ot_R f_j)
\\&=\dsum_{j}d_g\al_{gh}(f_j(\psi_{gh}(1_{gh}\psi^{\m}_g( \psi^{\m}_h(p_j1_{h^{\m}})d'_h1_{g^{\m}})))1_{(gh)^{\m}})
\\&=\dsum_{j}d_g\al_{gh}(f_j(\psi_{gh}(\psi^{\m}_g( \psi^{\m}_h(p_j1_{h^{\m}})d'_h1_{g^{\m}})))1_{(gh)^{\m}})
\\&=\dsum_{j}d_g\al_{gh}(f_j(\psi_{gh}(\psi^{\m}_g( \psi^{\m}_h(p_j\al_{h^{\m}}(d'_h1_{g^{\m}})))))1_{(gh)^{\m}})=\diamond. 
\end{align*}
It follows by \cite[p. 748]{DoPaPi} that $\psi_{(ab)^{\m}}\psi^{\m}_{b^{\m}}\psi^{\m}_{a^{\m}}(x)=\omega_{a,b}x,$ for any $a,b\in G$ and  $x\in 1_a1_{ab}P.$ Thus, taking $a=h^{\m}$ and $b=g^{\m},$ we have that $\psi_{gh} \psi^{\m}_g\psi^{\m}_h(x)=\omega_{h^{\m}, g^{\m}}x.$  Then
\begin{align*}
\diamond&=\dsum_{j}d_g\al_{gh}(f_j(\omega_{h^{\m}, g^{\m}}p_j\al_{h^{\m}}(d'_h1_{g^{\m}}))1_{(gh)^{\m}})
=\dsum_{j}d_g\al_{gh}(\omega_{h^{\m}, g^{\m}}\al_{h^{\m}}(d'_h1_{g^{\m}})f_j(p_j))
\\&=d_g\al_{gh}(\omega_{h^{\m}, g^{\m}}\al_{h^{\m}}(d'_h1_{g^{\m}}))
=d_g\al_{g}(d'_h1_{g^{\m}})\al_{gh}(\omega_{h^{\m}, g^{\m}})= d_g\al_{g}(d'_h1_{g^{\m}})\mathfrak b_{g,h}=f^\Sigma_{g,h}(d_g\otimes_R d'_h).
\end{align*}
Hence $(\zeta\circ \varphi_3)([P])=\zeta([\D(\T^P)])=\zeta([\D(\Sigma)])={\rm cls}(\mathfrak b_{-,-})=\phi_3([P^*]),$ thus $\zeta^*\circ \varphi_3=\phi_3.$

Concerning the next morphism,  consider, 
$$
\begin{array}{c c c l}
\bar\eta \circ \varphi_4: & H^2(G,\alpha, R) & \longrightarrow & \B(R/R^\alpha)
 \\&  {\rm cls}(w) & \longmapsto & \llbracket\D(\Sigma)\rrbracket \end{array},
$$
where $\bar\eta:  \B(\T/R) \to  \B(R/R^\alpha) $  is the isomorphism induced by \eqref{etta} and $\Sigma $ is as defined as above, but  with factor set
$$f_{g,h}^\Sigma=\{f^\Sigma_{g,h}:\Sigma_g\otimes_R \Sigma_h\ni u_g\otimes_R u'_h \mapsto\om_{g,h}u_g\alpha_g(u'_h1_{g^\m})\in 1_g\Sigma_{gh}\}_{g,h\in G}.$$
Then it follows from \eqref{prog}  that the map $ R\star_{\alpha, \omega} G\ni u_g\delta_g \mapsto u_g\in \D(\Sigma),$ is an $R^\alpha$-algebra isomorphism. Thus $\bar\eta \circ \varphi_4=\phi_4,$ where $\phi_4$ is the group homomorphism defined in (5.1) of \cite{DoPaPi2}.

\noindent With respect to the morphism $\phi_5$  (see \eqref{fi5})   we shall show that the map $\varphi_5 \circ \bar\eta^{\m}:   \B(R/R^\alpha)\to H^1(G,\alpha^*,{\bf PicS}_R(R)),$ satisfies $\varphi_5 \circ \bar\eta^{\m}(\llbracket A \rrbracket)= \phi_5(\llbracket A \rrbracket^\m)$. Indeed, take $\llbracket A \rrbracket\in  \B(R/R^\alpha),$ let $C_g$ be defined by  \eqref{defC}, 
 $f_A\in Z^1(G,\al^*,\Pics_R(R))$  be such that   $f_A(g)=[C_g],$ for all $g\in G,$  and set  $\Gm: G\ni g\mapsto f(g)\T(g)\in \Pics_{R^\alpha}(R).$   By Theorem \ref{Aisopcgp}, we have that 
$\eta([\D(\G)]^\m)=\llbracket A \rrbracket,$ that is $\bar\eta([\D(\G)]^\m{\rm Im}(\mathcal L_0))=\llbracket A \rrbracket,$ and by the definition of 
$\varphi _5$ we obtain
$$\varphi_5 \circ \bar\eta^{\m}(\llbracket A \rrbracket)=\varphi_5([\D(\G)]^\m{\rm Im}(\mathcal L))=\varphi_5\circ \mathcal{L}^c([\D(\G)]^\m) =((\zeta \mathcal{L})^c\circ\zeta)([\D(\G)]^\m).$$ 
By the proof of \cite[Theorem 4.1]{DoRo} we get that  the inverse element of the class $[\D(\G)] \in \C(\T/R)$ is given by $[\D(\G)]^{\m}=\left[ \bigoplus_{g\in G}\T_g\ot_R
\G_{g^\m}\ot_R\T_g\right],$ then  by \eqref {fom} we see  that  $\zeta([\D(\G)]^\m)=\bar\kappa,$ where  
\begin{align*}\bar\kappa: G\ni g& \mapsto [\T_g\ot_R
\G_{g^\m}\ot_R\T_g\ot_R\T_{g^{\m}} ] =  
[\T_g\ot_R \G_{g^\m}\ot_R  D_g  ]  = [\T_g\ot_R \G_{g^\m}\ot_R  \G _g \otimes _R \G _{g^\m}  ] \\ &=  [\T_g\ot_R
\G_{g^\m}] =[\T_g\ot_R C_{g^\m}\ot_R  \T_{g^\m}]
\stackrel{\eqref{cong}}=\alpha^*_g([C_{g^\m}])\stackrel{\eqref{z1}}=[C_g]^{\m}=f(g)^\m \in \Pics_R(R).
\end{align*}
From this we conclude that $\varphi_5 \circ \bar\eta^{\m}(\llbracket A \rrbracket)=(\zeta \mathcal{L})^c(f^\m)={\rm cls}(f^\m)=\phi_5(\llbracket A \rrbracket^\m).$ Thus  $\varphi_5 \circ (\bar\eta^{\m})^*=\phi_5.$  

\noindent Finally, we move to the last morphism. We shall show that $\varphi_6=\phi_6,$ where  $\phi_6$ is given in \cite[p.217-218]{DoPaPi2}.  Indeed, for $ f\in Z^1(G,\alpha^*,{\bf PicS}_R(R)),$ we have that $\varphi_6({\rm cls}(f))=(\varphi_6\circ(\zeta\mathcal{L})^c) (f)=\delta(f)={\rm cls}\left(\widetilde{\beta_{-,-,-}^f}\right).$ Since $f\in Z^1(G,\alpha^*,{\bf PicS}_R(R)) $ we obtain a partial representation $\G:G\ni g\mapsto f(g)\T(g)\in {\bf PicS}_{R^\alpha}(R)$  that satisfy the conditions of  \cite[Proposition 3.17]{DoRo}. Hence, this partial representation gives rise to a map  $\beta_{-,-,-}^f$ as in \cite[Corollary 3.16]{DoRo} and by the construction of $\widetilde{\beta_{-,-,-}^f},$ we know that  for all $g,h,l\in G$ the element $\widetilde{\beta_{g,h,l}^f}$ belongs to  $\U(D_gD_{gh}D_{ghl})$ and $\beta_{g,h,l}^f(x)=\widetilde{\beta_{g,h,l}^f}x,$ for all $x\in 1_g1_{gh}\G_{ghl}.$ Moreover, the element $\widetilde{\beta_{g,h,l}^f}$ is unique with this property.   With respect to $\phi_6$, observe that by the construction of the diagram (28) in \cite{DoPaPi2}, the map $\beta_{-,-,-}^f$ above coincides with $\omega,$ obtained  from the diagram  (28) in \cite{DoPaPi2}. Now if, $\phi_6({\rm cls}(f))={\rm cls}(\omega)$ then $\omega_{g,h,l}\in \U(D_gD_{gh}D_{ghl})$ is such that $\omega_{g,h,l}x=\widetilde\omega(x)=\beta_{g,h,l}^f(x),$ and we have $\omega_{g,h,l}=\widetilde{\beta_{g,h,l}^f},$ for all $g,h,l\in G$ and $x\in1_g1_{gh}\G_{ghl}.$  This implies $\varphi_6=\phi_6,$  as desired.
\end{dem}

\subsection{Some final remarks}\label{sec:final}
\begin{obs}\label{rem:final}
Let $\G:G\ni g\mapsto [\G_g]\in \Pics(R)$ be a partial representation endowed with a factor set,  $\G'_g$ be an $R$-bimodule in the isomorphism class of   $[\G_g],$ and let $a_g:\G'_g\to \G,$ be an   $R$-bimodule isomorphism.   Then the family

$$\mathcal F'=\{f'_{g,h}=a^{\m}_{gh}\circ f^\G_{g,h}\circ(a_g\ot_Ra_h): \G'_g\ot_R \G'_h\to 1_g\G'_{g,h}\}_{g,h\in G}$$
is a factor set for $\G$ such that    $\D'= \bigoplus\limits_{g\in G}\G'_g$ is a partial generalized crossed product with $[\D(\G)]=[\D'].$  Indeed, for $g,h,l\in G$ we have 

\begin{align*}
f'_{gh, l}\circ (f'_{g,h}\ot_R \G'_l)&=a^{\m}_{ghl}\circ f^\G_{gh,l}\circ(a_{gh}\ot_Ra_l)\circ ([a^{\m}_{gh}\circ f^\G_{g,h}\circ(a_g\ot_Ra_h)]\ot_R\G'_l)\\
&=a^{\m}_{ghl}\circ f^\G_{gh,l}\circ (f^\G_{g,h}\circ(a_g\ot_Ra_h)\ot_Ra_l)
\\&=a^{\m}_{ghl}\circ f^\G_{gh,l}\circ (f^\G_{g,h}\ot_R \G_l)\circ((a_g\ot_Ra_h)\ot_Ra_l)
\\&=a^{\m}_{ghl}\circ f^\G_{g,hl}\circ (\G_g \ot_R f^\G_{h, l} )\circ(a_g\ot_R (a_h\ot_Ra_l)),
\end{align*}
and
\begin{align*}
f'_{g,h l}\circ ( \G'_g\ot_R f'_{h,l})&=a^{\m}_{ghl}\circ f^\G_{g,hl}\circ(a_{g}\ot_Ra_{hl})\circ[ \G'_g \ot_R a^{\m}_{hl}\circ f^\G_{h,l}\circ(a_h\ot_Ra_l)] \\
&=a^{\m}_{ghl}\circ f^\G_{g,hl}\circ( a_{g} \ot_R (f^\G_{h,l}\circ(a_h\ot_Ra_l)))
\\&=a^{\m}_{ghl}\circ f^\G_{g,hl}\circ (\G_g \ot_R f^\G_{h, l} )\circ (a_g\ot_R(a_h\ot_Ra_l)),\end{align*} showing that $\D'$ is a partial generalized crossed product. Finally, 
 by the definition of $\mathcal F'$ the diagram \eqref{morfismodepcgp0} commutes and hence $[\D(\G)]=[\D'].$
\end{obs}

\begin{pro}
Let $Z^\C=\{f\in Z^1(G, \alpha^*, {\Pics(R)})\mid [\D(f\T)]\in \C(\T/R)\},$ then the following assertions hold:
\begin{enumerate}
\item $Z^\C$ is a subgroup of $Z^1(G, \alpha^*, {\Pics(R)})$ and contains $B^1(G, \alpha^*, {\Pics(R)}).$
\item  The map $\mathfrak c: Z^\C \ni f\mapsto [\D(f\T)]\in \C(\T/R),$ where $\D(f\T)$ is being considered with the factor set 
$ {\mathcal F}^{f^{\Gamma}}$ from the proof of Theorem~ \ref{exac}, is a group monomorphism,.
\item  The diagram
	$$\xymatrix{{\bf Pic}_R(R)  \ar[dd]_{(\delta^0)^*}\ar[rr]^{\mathcal L_0} & &\C(\T/R)\ar[dd]_{\zeta}\\
		& & \\
		B^1(G, \alpha^*, {\Pics(R)})\ar@{^(->}[rr]_{\iota} & & Z^\C(G, \alpha^*, {\Pics(R)})}$$
	is commutative. 

\end{enumerate}

\end{pro}
\begin{proof} (1) Take $f\in B^1(G, \alpha^*, {\Pics(R)}).$ Then by \eqref{b1}, there exists $[P]\in {\bf Pic}_R(R)$ such that $f(g)\simeq P^*\ot_R \T_g \ot_R P\ot_R \T_{g^{\m}}.$ Then  
$$f(g)\ot_R \T_g\simeq P^*\ot_R \T_g \ot_R P\ot_R \T_{g^{\m}}\ot_R\T_g\simeq P^*\ot_R \T_g \ot_R P\ot_R D_{g^{\m}}\simeq  P^*\ot_R \T_g \ot_R P\simeq \T^{P^*}.$$ It follows by Remark~\ref{rem:final}  that $\Delta(f\T)$ is a partial generalized crossed product and $[\Delta(f\T)]=\mathcal L_0([P^*])\in \C(\T/R),$ which shows that  $B^1(G, \alpha^*, {\Pics(R)})\subseteq Z^\C$. To show that $Z^\C$ is a group, take $f_1,f_2\in Z^\C,$ then it follows by the proof of Theorem  \ref{exac} that $f_1f_2\in  Z^\C.$ Moreover,  by \cite[Lemma 6.6]{DoPaPi2} there is 
 $\llbracket A \rrbracket \in \B(R/R^\al)$ such that $ \phi_5(\llbracket A \rrbracket)={\rm cls}(f^\m_1),$ which by \cite[Proposition 6.5]{DoPaPi2}   implies that $\Delta(f^\m_1\T)$ is a partial generalized crossed product. Thus $f^\m\in Z^\C,$ and $Z^\C$ is a subgroup of $Z^1(G, \alpha^*, {\Pics(R)}).$

(2)  The fact that we have a group  homomorphism follows  from Theorem~\ref{prop:grupoC}  and  the proof of  Theorem  \ref{exac}.  To see that $\mathfrak c$ is a monomorphism, consider  the morphism   $\zeta: C(\T/R)\ni [\D(\Om)]\mapsto f^\Om \in Z^1(G, \alpha^*, {\Pics(R)})$ defined in \eqref{fom}. Then it is clear that 
$  \zeta \circ \mathfrak c $ is the identity in $C(\T/R).$ 

(3) A routine calculation shows that the diagram is commutative.
\end{proof}

Our final remarks gives us a relation between the classes of partial generalized crossed products and epsilon-strongly graded rings.
\begin{obs}  Let $\G: G\to \Pics(R')$ be a unital partial representation endowed with a factor set. Then it follows by  Proposition   \ref{pcgpeanelcomutativocom1} and \eqref{prog}  that $\D(\G)$   is a $G$-graded ring. Moreover, by  \cite[Lemma 3.8]{DoRo}, \cite[(iii) Proposition 7]{NyOP} and \cite[Lemma 3.6]{BMP}  we have  that the class of partial generalized crossed products  associated  to unital partial representations of $G$ in $\Pics(R'),$ coincides with the class  of epsilon-strongly $G$-graded rings. \footnote{In  \cite[Lemma 3.6]{BMP}  it is assumed that  $R$ is commutative, but examining its proof one can realize  that this condition is not necessary.}  Denote by $\G^{\F}$  the partial representation $\G$  endowed with a factor set $\F.$ By  Remark  \ref{rem:final} we have  $C_0(\G/R')=\{[\G^{\F}]\mid \F\,\, \text{is a factor set for }\, \,\G\}.$ Since  two partial generalized crossed product are isomorphic, if and only if, they are isomorphic as graded rings, then $C_0(\G/R')$ coincides with the set $C(R',\G)$ defined in \cite[p. 234]{NysOinPin3}. Hence \cite[Theorem 3]{NysOinPin3} extends partially  \cite[Theorem 4.4]{DoRo} to the realm 
of  epsilon-strongly groupoid graded rings.
 \end{obs} 
 
 \section{Acknowledgments} The first named author was partially supported by 
Funda\c c\~ao de Amparo \`a Pesquisa do Estado de S\~ao Paulo (Fapesp), process n°:  2020/16594-0, and by  Conselho Nacional de Desenvolvimento Cient\'{\i}fico e Tecnol{\'o}gico (CNPq), process n°: 312683/2021-9. The second named author was supported by Fapesp, process n°: 2023/14066-5.

\end{document}